\documentclass[12pt]{amsart}
\usepackage{latexsym}
\usepackage[utf8]{inputenc}
\usepackage{amsmath}
\usepackage{amsthm}
\usepackage{amsfonts}
\usepackage{amssymb}
\usepackage{epsfig}
\usepackage{nicefrac}
\usepackage[all]{xy}
\usepackage{graphicx}
\usepackage{graphics} 
\graphicspath{ {./images/} }
\usepackage{subcaption}
\usepackage{enumerate}
\usepackage[T1]{fontenc}
\usepackage{subcaption}
\usepackage[font=small,skip=0pt]{caption}
\usepackage{tikz}
\usepackage{subfig}
 
\newtheorem{theorem}{Theorem}[section]
\newtheorem{corollary}[theorem]{Corollary}
\newtheorem{proposition}[theorem]{Proposition}
\newtheorem{lemma}[theorem]{Lemma}
\newtheorem{claim}[theorem]{Claim}
\newtheorem{question}[theorem]{Question}

\theoremstyle{definition}
\newtheorem{assumption}[theorem]{Assumption}
\newtheorem{remark}[theorem]{Remark}

\newtheorem{definition}[theorem]{Definition}

\newcommand{\CC}{{\mathbb C}}
\newcommand{\RR}{{\mathbb R}}
\newcommand{\ZZ}{{\mathbb Z}}
\newcommand{\NN}{{\mathbb N}}
\newcommand{\TT}{{\mathbb T}}

\begin{document}

\title[Constraints on Lagrangians, and the shape invariant]{New constraints on Lagrangian embeddings, and the shape invariant}
\date{\today}
\author{Richard Hind}
\address{RH: Department of Mathematics\\
University of Notre Dame\\
255 Hurley\\
Notre Dame, IN 46556, USA.}
\author{Ely Kerman}
\address{EK: Department of Mathematics\\
University of Illinois at Urbana-Champaign\\
1409 West Green Street\\
Urbana, IL 61801, USA.}
\thanks{Both authors are supported by grants from the Simons Foundation}

\begin{abstract}
For a large class of toric domains in $\RR^4$,  we determine which product Lagrangian tori can be mapped into the domain by a Hamiltonian diffeomorphism. In other words, we compute the Hamiltonian shape invariant of these toric domains, as defined in \cite{hizh}. The argument relies on new intersection results for product Lagrangian tori in symplectic polydisks. For Hamiltonian diffeomorphisms which map certain Lagrangian product tori back into the polydisk, we establish intersections between the images and a one-parameter family of product Lagrangian tori that includes (is based at) the original torus. For symplectic polydisks with area ratios less than two, we strengthen this to establish intersections between the Hamiltonian images and the original Lagrangian torus. As a soft complement to these intersection results we also present an embedding construction which demonstrates that this intersection rigidity vanishes when the one-parameter family of product Lagrangian tori is replaced by a natural packing by Lagrangian tori.

\end{abstract}

\maketitle

\section{Introduction}

Let $L$ be a Lagrangian submanifold of $\RR^{2n}$ equipped with its standard symplectic structure. Given an open subset $U \subset \RR^{2n}$ it is natural to ask whether there exists a Hamiltonian diffeomorphism of $\RR^{2n}$ that maps $L$ into $U$. One can then ask whether such maps must also satisfy special constraints. For example, if $L$ is contained in $U$, then  the 
set  $$\{ \phi(L) \mid \phi \in \mathrm{Ham}(\RR^{2n}), \phi(L) \subset U\}$$ is nonempty and one can ask if any pair of elements must intersect. Even in simple settings, these pairwise intersections may not be guaranteed. They may also only be a small part of a richer intersection story.

Consider, for example, the standard open disk $D(b) \subset \RR^2$ of area $b$ and the standard (Lagrangian) circle $L(s) $ enclosing area $s<b$. Suppose that $b > 2s$. Then  $L(s)$ and its Hamiltonian images in $D(b)$ need not intersect. On the other hand, the two intersection statements below always hold and follow immediately from simple area considerations.

\begin{theorem}\label{thm:2a} If $\phi$ is a Hamiltonian diffeomorphism of $\RR^2$ such that $\phi(L(s)) \subset D(b)$, then $\phi(L(s))$ must intersect the family $$\bigcup_{t\in [s,b-s]}L(t).$$ 
\end{theorem}


\begin{theorem}\label{thm:2b}
    Let $k \geq 2$ be the unique integer such that $ks<b \leq (k+1)s$ and suppose that  $$\bigsqcup_{j=1}^k\Lambda_j$$ is a disjoint union of embedded closed curves $\Lambda_j$ in $D(b)$ that each bound a domain of area  $s$.  If $\phi$ is a Hamiltonian diffeomorphism of $\RR^2$ that maps $L(s)$ into $D(b)$, then $\phi(L(s))$ must intersect at least one of the $\Lambda_j$.
\end{theorem}
\begin{figure}[!h]
    \centering
    \includegraphics[width=0.85\linewidth]{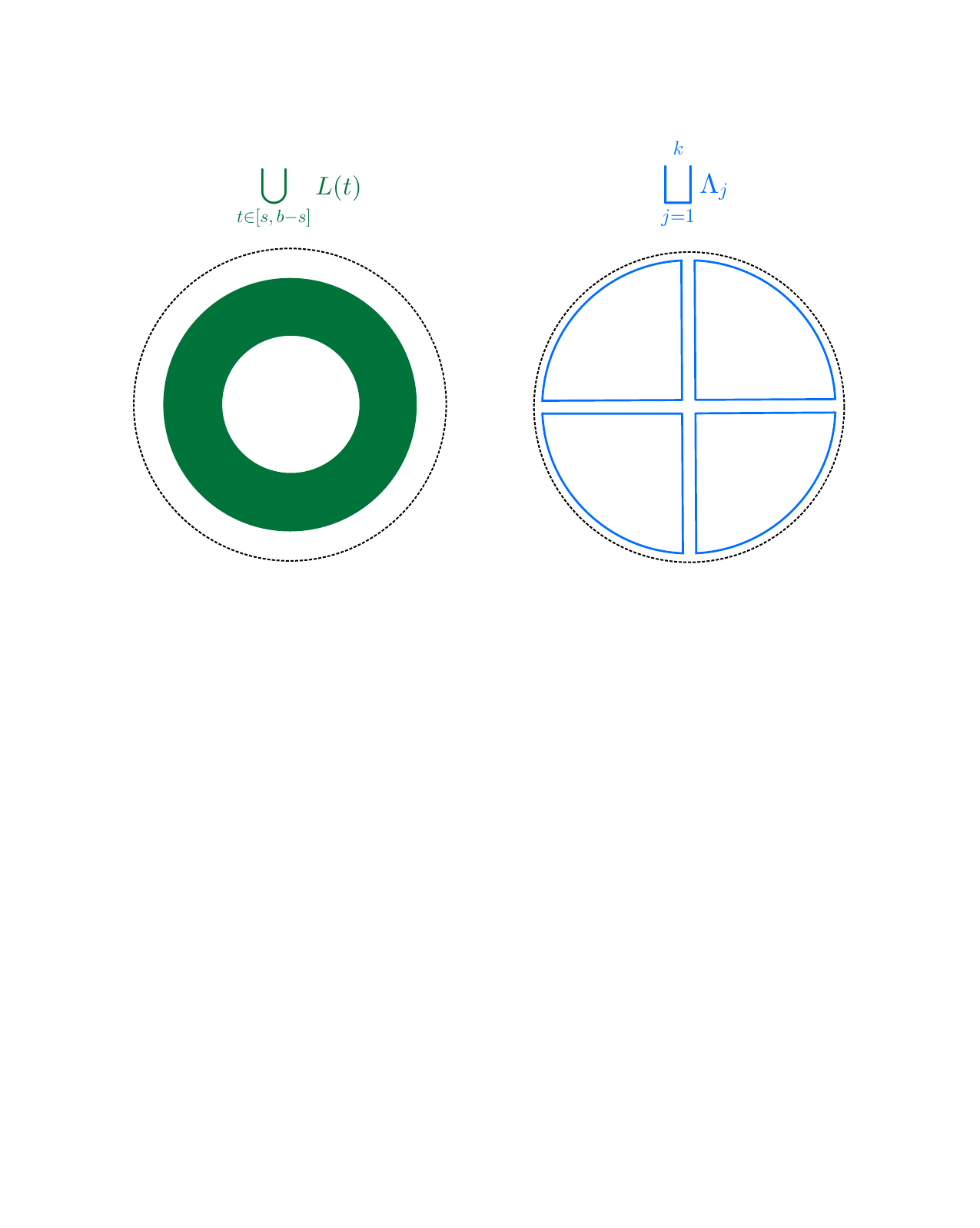}
    \vspace{-6.5cm}
    \caption{Any Hamiltonian image of $L(s)$ in $D(b)$ must intersect both $\bigcup_{t\in [s,b-s]}L(t)$ and $\bigsqcup_{j=1}^k\Lambda_j$.}
    \label{fig:2in2}
\end{figure}

\bigskip

In this paper we consider if and how the rigidity statements in Theorems \ref{thm:2a} and \ref{thm:2b} might extend to  four dimensions. Identifying $\RR^4$ with $\CC^2$, we consider product Lagrangian tori,  
\begin{equation*} 
L(r,s) =\{ (z_1,z_2) \in \CC^2 \mid \pi |z_1|^2 = r, \, \pi |z_2|^2 = s \},
\end{equation*}
and symplectic polydisks,
\begin{equation*} \label{poly}
P(a,b) = \{(z_1,z_2) \in \CC^2 \mid \pi |z_1|^2 <a, \, \pi |z_2|^2 <b \}.
\end{equation*}

\subsubsection{Generalizing Theorem \ref{thm:2a}}\label{secgen}
Our main result is the following extension of Theorem \ref{thm:2a} to dimension four.

\begin{theorem}\label{lagint}
Suppose that  $1/2\leq r<1$ and $1 \leq s <b/2$. If $\phi$ is a Hamiltonian diffeomorphism of $\RR^4$ that maps $L(r,s)$ into $P(1,b)$, then $\phi(L(r,s))$ intersects the codimension one submanifold
\begin{equation} 
\bigcup_{t \in [s,b-s]} L(r, t).
\end{equation}
\end{theorem}

The first three constraints on $r$ and $s$ in Theorem \ref{lagint} are each necessary. If $r$ is at least $1$, then Theorem 3 of \cite{ho} implies that there is no Hamiltonian diffeomorphism of $\RR^4$ that maps $L(r,s)$ into any polydisk of the form $P(1,b)$. If $r$ is less than $1/2$, then there are Hamiltonian diffeomorphisms of $D(1)$ that displace $D(r)$ from itself, and any such map can be used to construct a Hamiltonian diffeomorphism of $\RR^4$ that maps $L(r,s)$ into $P(1,b) \smallsetminus \bigcup_{t \in [s,b-s]} L(r, t).$ If $s$ is less than $1$, and greater than $r$, then $L(s,r)$, which is Hamiltonian isotopic to $L(r,s)$, is contained in $P(1,b)$ and is disjoint from the family $\bigcup_{t \in [s,b-s]} L(r, t)$. 


\medskip



In the limit $b\to 2s$, Theorem \ref{lagint} yields the following Lagrangian intersection result which appears to be new. 

\begin{corollary}\label{newlag}
Suppose that $1/2\leq r<1$ and $1 \leq s <b \le 2s$. If $\phi$ is a Hamiltonian diffeomorphism of $\RR^4$ that maps $L(r,s)$ into $P(1,b)$, then $$L(r,s) \cap \phi(L(r,s)) \neq \emptyset.$$
\end{corollary}

As described below in Section \ref{sec:2lag}, one can use the methods developed here to prove a more general result. We show, in Corollary \ref{newlagv2}, that any two Hamiltonian images of $L(r,s)$ in $P(1,b)$ must intersect.

\begin{remark}
It would be interesting to know if these intersections can be detected by some refinement of Lagrangian Floer theory. In the same spirit, it would also be interesting to know whether there is a topological lower bound for the number of intersection points here.
\end{remark}

\begin{remark}
Versions of Corollary \ref{newlag} fail in dimension greater than $4$. More precisely, it follows from \cite[Theorem A]{chek} that when $\alpha>1$ is irrational, the product torus $L=L(1, \alpha, 2)$ is Hamiltonian isotopic, in $\RR^6$, to infinitely many distinct product tori lying arbitrarily close to $L$. This is discussed in section 2 of the article \cite{atallah}, as well as in \cite{brsh}. 
\end{remark}

\begin{remark} If one compactifies $P(1,b)$ to $S^2 \times S^2$ with a nonmonotone symplectic form, then the analogue of Corollary \ref{newlag} fails almost completely. Corollary 1.10 of \cite{brki} implies that there is a set of split Lagrangian tori in $S^2 \times S^2$, of full measure, such that for every $L$ in the set every Weinstein neighborhood of $L$ contains a disjoint Hamiltonian image of $L$. For more discussion on the $S^2 \times S^2$ case see Remark \ref{brendelschmitz}.
    
\end{remark}

\medskip

In the limit $b\to \infty$, Theorem \ref{lagint} yields  the following Lagrangian refinement of Gromov's nonsqueezing theorem from \cite{gr}.
\begin{corollary}\label{binf}
   Suppose that $r<1 \le s$. If there is a Hamiltonian diffeomorphism of $\RR^4$ that maps $L(r,s)$ into 
   $$(D(1) \times \RR^2) \smallsetminus \bigcup_{t \in [s,\infty)} L(r, t),$$ then $r <1/2$.
\end{corollary}

\bigskip

\subsubsection{Generalizing Theorem \ref{thm:2b}} The setting of Theorem \ref{thm:2b} can be extended to dimension four in the following natural way. Given a collection of disjoint embedded closed curves  
$$\{\Lambda_1, \dots, \Lambda_{k}\} \subset D(b)$$
each bounding area $s$, as in Theorem \ref{thm:2b}, consider the $k$ disjoint Lagrangian tori
$$L_j = L(r) \times \Lambda_j$$ in $P(2r,b).$ Given a Hamiltonian diffeomorphism $\phi$ of $\RR^4$ with $\phi(L_1) \subset P(2r,b)$, one can ask whether $\phi(L_1)$ must intersect at least one of the $L_j$. 
By continuity, when asking this question, one can assume that  $$b < (k+1)s.$$

In this setting, the work of Polterovich and Shelukhin from \cite{ps} yields a local rigidity theorem in the spirit of Theorem \ref{thm:2b}. The relationship between $r$ and the positive number
 \begin{align}\label{eq:g}
    \mathfrak{o} = (k+1)s-b,
\end{align}
plays a crucial role.
\begin{theorem}(Theorem C, \cite{ps})\label{thm:ps}
If $r < \mathfrak{o}/(k-1),$ then there is no Hamiltonian diffeomorphism supported in $P(2r,b)$ that displaces $L(r,s)$ from all of the $L_j.$
\end{theorem}

In the present work, we show, by construction, that even this local rigidity disappears when $r$ is sufficiently large.

\begin{theorem}\label{thm:hm}
If $k$ is even, then for every $r> 6\mathfrak{o}/k $ there is a Hamiltonian diffeomorphism of $P(2r,b)$ that displaces $L(r,s)$ from all of the  $L_j.$
\end{theorem}

This construction also extends to the limit $b \to \infty$.

\begin{corollary}\label{cor:disc}
    Let $\Lambda \colon S^1 \times \NN \to \RR^2$ be a proper embedding such that each $\Lambda_j = \Lambda(S^1 \times \{j\})$ bounds a disk of area $s$. Set $L_j = L(1/2) \times \Lambda_j$. There exists a Hamiltonian diffeomorphism $\phi$ of the cylinder  $D(1) \times \RR^2 $ such that $\phi(L(1/2,s)) \cap L_j=\emptyset$ for all $j \in \NN$.
\end{corollary}

The following related questions remain unresolved.

\begin{question}
    For which values of $r$ between $\mathfrak{o}/(k-1)$  and $6\mathfrak{o}/k$, if any,  does the rigidity statement of Theorem \ref{thm:ps} hold?
\end{question}

\begin{question}
    If one allows all Hamiltonian diffeomorphisms of $\RR^4$, not just those supported in $P(2r,b)$, does the rigidity statement  of Theorem \ref{thm:ps} hold for any positive value of $r$?
\end{question}

\begin{remark}\label{brendelschmitz}
    If one compactifies $P(2r, b)$ to $S^2 \times S^2$, then the gap between rigidity and nonrigidity shrinks to a point. Theorem C of \cite{ps} asserts that for  $r < \mathfrak{o}/(k-1)$ there is no Hamiltonian diffeomorphism of $S^2 \times S^2$ that displaces $L(r,s)$ from all of the $L_j.$ On the other hand, by Corollary 1.16 of \cite{brki}, for almost every $s<b$ satisfying $r > \mathfrak{o}/(k-1)$,
    there is a infinite sequence of Hamiltonian diffeomorphisms $\phi_j$ of $S^2 \times S^2$ such that the images $\phi_j(L(r,s))$ are pairwise disjoint.

    There is more flexibility if we also consider Lagrangian tori $L(x,y) \subset P(2r,b)$ with $x \neq r$. Assuming our compactification is not monotone, that is, $2r \neq b$, recent work of Schmitz, \cite[Theorem B]{schmitz} says that, given $\varepsilon>0$, there exists a (single) Hamiltonian diffeomorphism $\psi$ of $S^2 \times S^2$ such that, for almost every $(x,y) \in \{0<x<2r, \, 0< y<b, \, |x-r| > \varepsilon \}$, the images $\psi^j(L(x,y))$ are pairwise disjoint.
\end{remark}

\subsection{The Hamiltonian shape invariant}
The original motivation for Theorem \ref{lagint} concerns the Hamiltonian shape invariant  defined  by Hind and Zhang in \cite{hizh}. The Hamiltonian shape of a subset $X$ of $\RR^4$ is the set 
\begin{align*} 
{\rm Sh}(X) & : = \left\{(r, s) \, \big| \, r \le s, \,  L(r, s) \hookrightarrow X \right\} \subset \RR_{>0}^2,
\end{align*}
where the notation $L(r, s) \hookrightarrow X$ implies that there exists a Hamiltonian diffeomorphism $\Phi$ of $\RR^4$ with $\Phi(L(r,s)) \subset X$. This set-valued symplectic invariant of $X$ has the following properties:\\
\begin{enumerate}
    \item ${\rm Sh}(cX) = c^2 {\rm Sh}(X), \text{ for all } c>0$,\\
    \item $X \hookrightarrow Y \implies {\rm Sh}(X)\subset {\rm Sh}(Y)$,\\
    \item ${\rm Sh}(\cup_i X_i) \supset \cup_i {\rm Sh}(X_i)$,\\
    \item ${\rm Sh}(\cap_i X_i) \subset \cap_i {\rm Sh}(X_i).$\\
\end{enumerate}
It can thus be viewed as a set-valued symplectic capacity. Like all capacities, the Hamiltonian shape is highly nontrivial to compute. A natural starting class of examples to consider are toric domains. 

Recall that a 
{\em toric domain} in $\RR^4$ is a subset of the form $X_{\Omega} = \mu^{-1} (\Omega)$ where $\Omega$ is a subset of $\RR^2_{\ge 0}$
and $\mu: \RR^4=\CC^2 \to \RR^2_{\ge 0}$ is the standard moment map defined by $$\mu(z_1, z_2) = (\pi|z_1|^2, \pi |z_2|^2).$$ 
For each point $(r,s)$ in the interior of $\Omega$, $\rm{int}(\Omega)$, the inverse image $\mu^{-1}(r,s)$ is the Lagrangian torus $L(r,s) \subset X_{\Omega}$. Hence, the set $\rm{int}(\Omega) \cap \{ r \leq s\}$ is contained in ${\rm Sh} (X_{\Omega})$. This direct connection between $\Omega$ and ${\rm Sh}(X_\Omega)$ inspires the hope that there might be a simple 
formula, or algorithm, for computing the Hamiltonian shape of $X_{\Omega}$ directly from $\Omega$.

This has been achieved for some important standard toric domains. In \cite{ho}, Hind and Opshstein compute the Hamiltonian shapes of balls and symplectic polydisks. In \cite{hizh}, Hind and Zhang compute the Hamiltonian shape of symplectic ellipsoids 
\begin{equation*} 
E(a,b) = \mu^{-1}(\{(r,s) \in \RR^2_{\ge 0} \mid x/a+y/b <1 \})
\end{equation*}
with $\frac{b}{a} \in \NN_{\geq 2}.$
Theorem \ref{lagint} and Corollary \ref{newlag} allow us to compute the Hamiltonian shape of two general families of toric domains.

\subsubsection{Information loss} Before stating our results, we recall two essential observations from  \cite{ho} concerning the ways in which information about $\Omega$ may become hidden in ${\rm Sh}(X_{\Omega})$. The first source of information loss in the passage from $\Omega$ to ${\rm Sh}(X_{\Omega})$, already accounted for in the definition of the shape, is due to a basic symmetry. Since the tori $L(r,s)$ and $L(s,r)$ are Hamiltonian isotopic, if one of them symplectically embeds in a domain, then so does the other.
Let $\sigma:(r,s) \mapsto (s,r)$ be the reflection in the diagonal in $\RR^2_{>0}$.
It follows that 
${\rm Sh}(X_{\Omega})$ contains the set $\Omega^+ \cap \{r \leq s\}$ where $$ \Omega^+= \Omega \cup \sigma \Omega$$
This transition from $\Omega$ to $\Omega^+$ can obscure features of $\Omega$. In particular, for distinct subsets $\Omega_1$ and $\Omega_2$ we can have $\Omega_1^+ = \Omega_2^+$, see Figure \ref{fig:blind1}.

\begin{figure}
    \centering
    \includegraphics[width=0.8\linewidth]{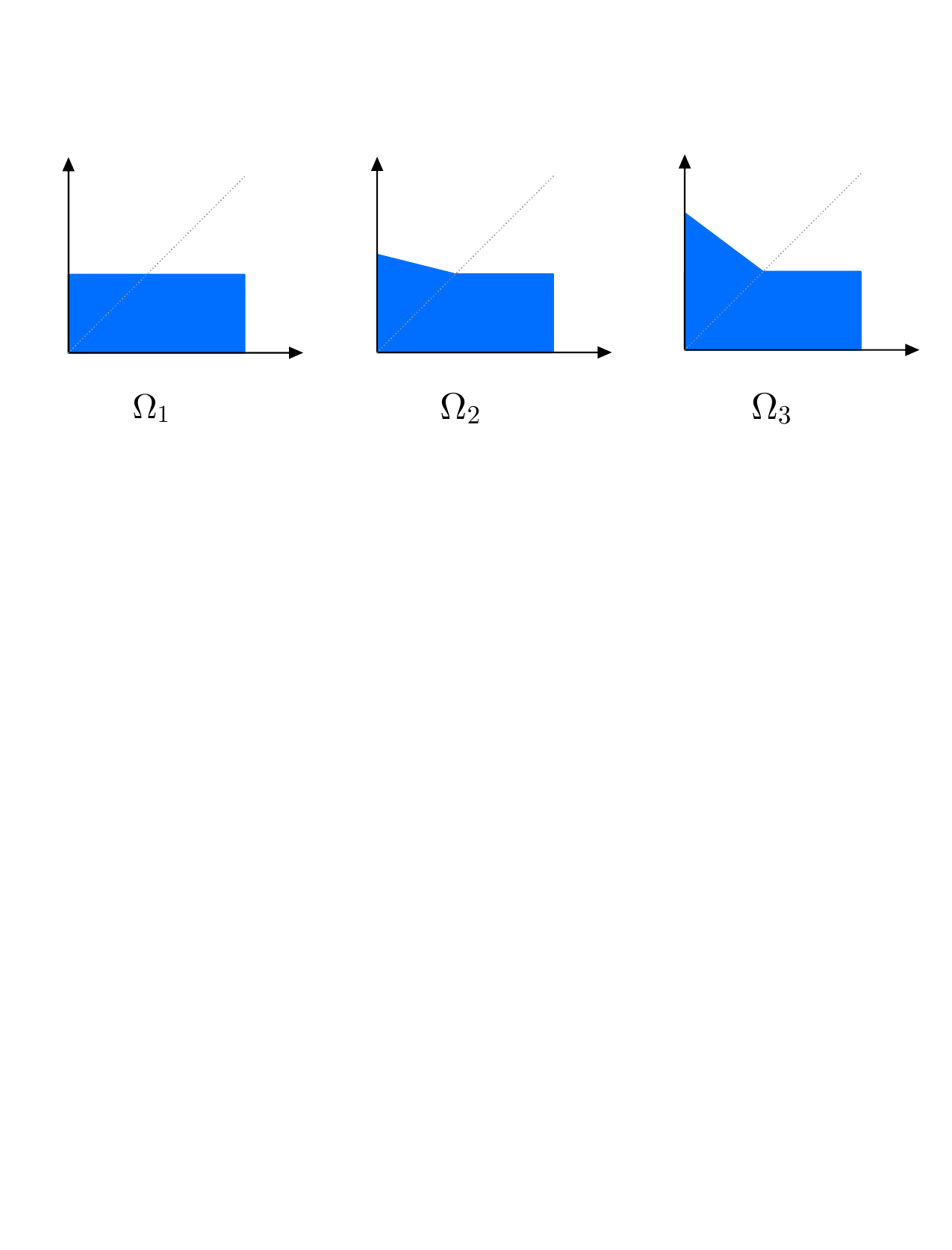}
    \vspace{-7.5cm}
    \caption{For these distinct regions $\Omega_1^+ = \Omega_2^+ =\Omega_3^+ $.}
    \label{fig:blind1}
\end{figure}

\bigskip

A second, more subtle, reason why some features  of $\Omega$ are not visible in ${\rm Sh}(X_{\Omega})$ is a consequence of the following Lagrangian embedding theorem of Hind and Opshtein.

\begin{theorem}(\cite{ho} \cite[Theorem 5.2]{hizh2})
For every $x>2$ and every $\epsilon >0$ there is a Hamiltonian diffeomorphism $\Phi$ of $\RR^4$ such that 
$\Phi(L(1,x))$ is contained in the toric domain $X_{(1+\epsilon)Q}$, where $Q$ is the convex hull of the set of points $\{(0,0), (2,0), (0,2), (1,2)\}$.
\end{theorem}

\begin{remark}
    The set $X_Q$ is the closure of $P(2,2) \cap E(2,4).$
\end{remark}








An immediate implication of this result is that 
${\rm Sh}(X_{\Omega})$ contains the unbounded set $$\{0\leq r < m(\Omega)\} \cap \{r\leq s\} $$ where the width of the infinite strip is given by $$m(\Omega) = \sup \{ c \mid cQ \subset \Omega \, \, \mathrm{or} \, \, cQ \subset \sigma \Omega \}.$$ Hence, no part of $\Omega$ that lies in the infinite strip $\{0\leq r < m(\Omega)\}$ is visible in ${\rm Sh}(X_{\Omega})$.

\bigskip

In all examples where the shape of $X_{\Omega}$ is known, there is nothing else to the story, and one has
\begin{equation}\label{eq:shape}
    {\rm Sh}(X_{\Omega}) = \Bigl(\Omega^+  \cup \{0\leq r < m(\Omega)\} \Bigr)\cap \{r\leq s\}.
\end{equation} This equality is  established for balls and symplectic polydisks in \cite{ho}, and for symplectic ellipsoids $E(a,b)$ with $\frac{b}{a} \in \NN_{\geq 2}$ in \cite{hizh}.
Here we show that \eqref{eq:shape} holds much more generally.

\subsubsection{New shape computations}

We will focus on subsets $\Omega \subset \RR^2_{\ge 0}$ that contain the unit square $[0,1)\times [0,1)$ and are contained in the unit strip $[0,1) \times \RR_{\ge 0}$. In this case, we have $\Omega^+ =\Omega$, $m(\Omega) =1/2$ and hence 
\begin{equation}\label{starter}
    {\rm Sh}(X_{\Omega}) \supset \Bigl(\Omega \cup \{0\leq r < 1/2\} \Bigr)\cap \{r\leq s\}.
\end{equation}
We will show that these sets are equal for two broad classes of subsets.

Theorem \ref{lagint} implies the following.

\begin{theorem}\label{main} Let $f: [0,1] \to [1, \infty]$ be any function. For the corresponding graphical subset $$\Omega(f)=\{ (r,s) \mid 0\leq r < 1, 0\leq s < f(r)\}$$ we have 
\begin{align*}
\label{fshape}
   {\rm Sh}(X_{\Omega(f)}) = \Bigl(\Omega(f)  \cup \{0\leq r < 1/2\} \Bigr)\cap \{r\leq s\}. 
\end{align*}

\end{theorem}

\begin{proof}
 For $0< r\leq s$, it follows from Proposition 2.3 of \cite{chsc} that the displacement energy of $L(r,s)$ is equal to $r$. The displacement energy of the symplectic cylinder $D(1) \times \RR^2 = \mu^{-1}([0,1) \times \RR_{\ge 0})$ is equal to $1$. Since the displacement energy and shape are monotone and $X_{\Omega(f)} \subset D(1) \times \RR^2$,  if $(r,s)$ is in ${\rm Sh}(X_{\Omega(f)})$ we must have $r\leq 1$. It suffices to show that if $(r,s)$ is in ${\rm Sh}(X_{\Omega(f)})$ and $1/2 \leq r \leq 1$ then $s < f(r).$ Equivalently, it suffices to show that for any $(r,s)$ with $r \in [1/2,1]$ and $s \ge f(r)$ there is no Hamiltonian diffeomorphism of $\RR^4$ that maps $L(r,s)$ into $X_{\Omega(f)}$. Choosing any $b > 2f(r)$, this is an immediate implication of Theorem \ref{lagint}.
\end{proof}

\begin{figure}[!h]
    \centering
    \includegraphics[width=0.8\linewidth]{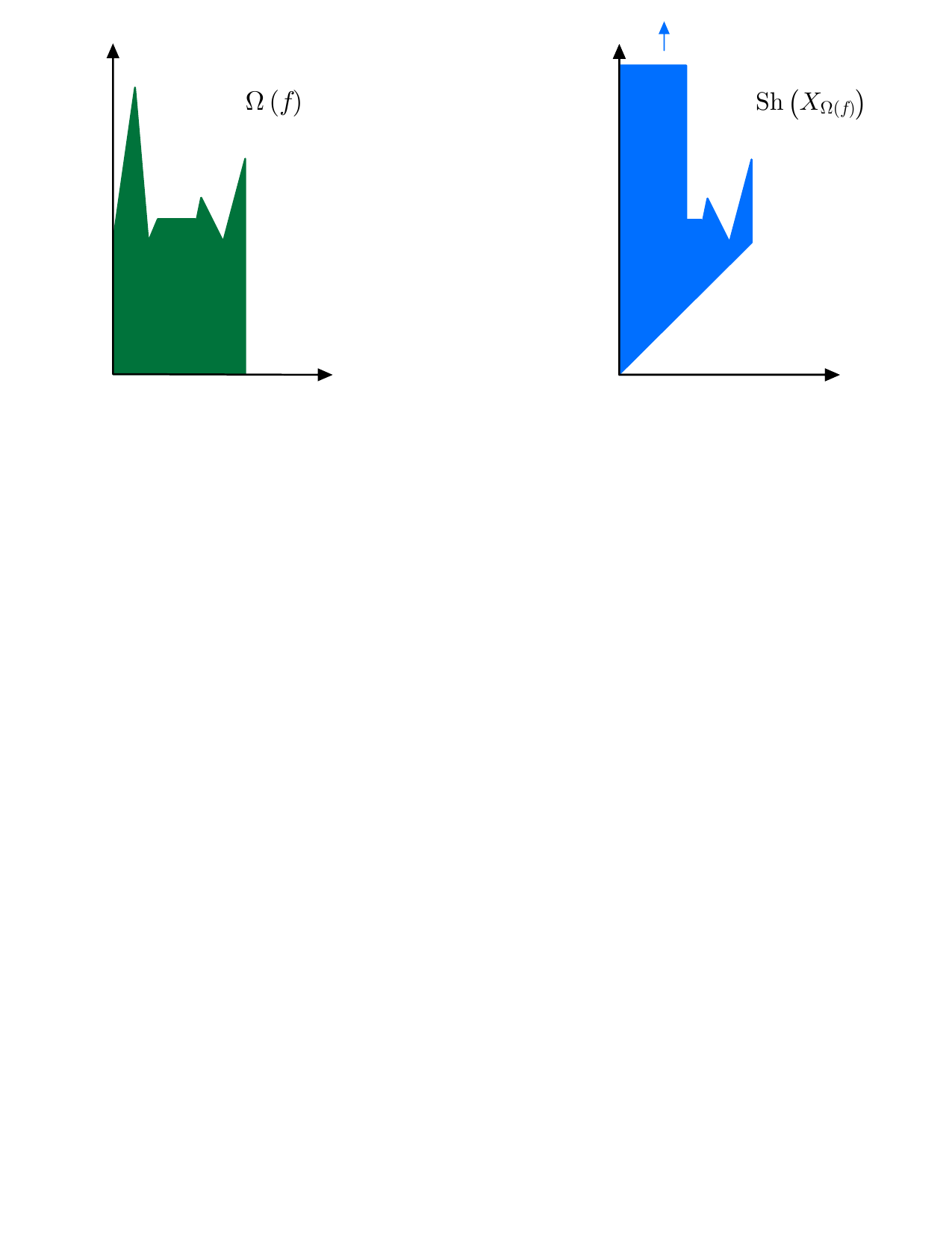}
    \vspace{-7.5cm}
    \caption{Equality \eqref{eq:shape} holds for $X_{\Omega(f)}$.}
    \label{fig:shape1}
\end{figure}

\begin{remark} Consider two functions $f,g: [0,1] \to [1, \infty]$. Theorem \ref{main} implies that if the toric domain $X_{\Omega(f)}$ and $X_{\Omega(g)}$ are symplectomorphic, then the restrictions of the functions $f$ and $g$ to $[1/2,1]$ must be identical. A {\em rigidity of profile} result of a similar spirit has recently been asserted by Hutchings, \cite{hutch-talk}, for generic strictly convex toric domains in $\RR^4$. 
\end{remark}

Arguing similarly, Corollary \ref{newlag} yields the following. 

\begin{theorem}\label{main2} Let $\Omega \subset \RR^2_{\geq 0}$ be any set that contains the square 
$[0,1)\times[0,1)$ and is contained in the rectangle $[0,1) \times [0,2).$
Then 
\begin{equation*}
\label{}
   {\rm Sh}(X_{\Omega}) = \Bigl(\Omega  \cup \{0\leq r < 1/2\} \Bigr)\cap \{r\leq s\}. 
\end{equation*}
\end{theorem}

The sets whose Hamiltonian shape is given by  Theorem \ref{main2} are not assumed to have any additional structure beyond the two inclusions asserted, see Figure \ref{fig:shape2}.

\begin{figure}[!h]
    \centering
    \includegraphics[width=0.8\linewidth]{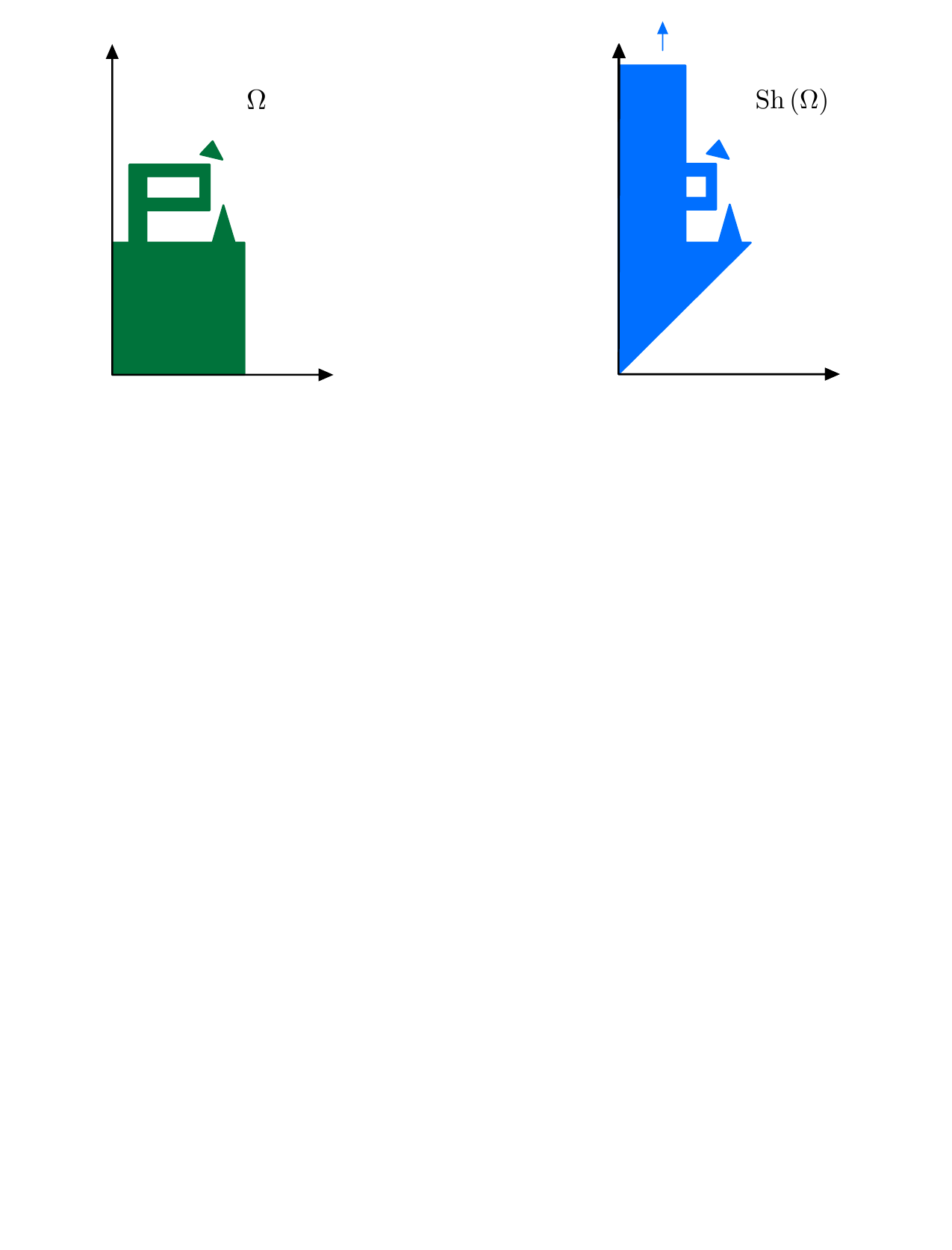}
    \vspace{-7.5cm}
    \caption{Equality \eqref{eq:shape} holds when $[0,1)\times[0,1)\subset \Omega \subset [0,1) \times [0,2).$}
    \label{fig:shape2}
\end{figure}

\subsubsection{On the shape of ellipsoids} For symplectic ellipsoids $E(a,b)$ with $b>2a$ and $\frac{b}{a} \notin 2\NN$, it is still not known whether the inclusion
$${\rm Sh}(E(a,b)) \subset \Bigl(\mu(E(a,b))  \subset \{0\leq r < a/2\} \Bigr)\cap \{r\leq s\}$$
is an equality. Theorem \ref{main} can be used to refine this problem. 

\begin{corollary}
For $b\geq 2a >0$ let $\mathcal{T}(a,b)$ be the triangular domain defined by the inequalities $r\le s$, $s \geq b -br/a$ and $s <a$. Then 
$${\rm Sh}(E(a,b))=\Bigl(\mu(E(a,b))  \cup \{0\leq r < a/2\} \Bigr)\cap \{r\leq s\}$$
iff no torus $L(r,s)$ with $(r,s) \in \mathcal{T}(a,b)$ can be mapped into $E(a,b)$ by a Hamiltonian diffeomorphism of $\RR^4$ .
\end{corollary}

\begin{proof}
    By monotonoicity, the set ${\rm Sh}(E(a,b))$ is contained in ${\rm Sh}(E(a,b) \cup P(a,a))$ which, by Theorem \ref{main}, is equal to $$\left(\Bigl(\mu(E(a,b))  \cup \{0\leq r < a/2\} \Bigr)\cap \{r\leq s\}\right) \cup \mathcal{T}(a,b). $$
\end{proof}


\begin{figure}[!h]
    \centering
    \includegraphics[width=1\linewidth]{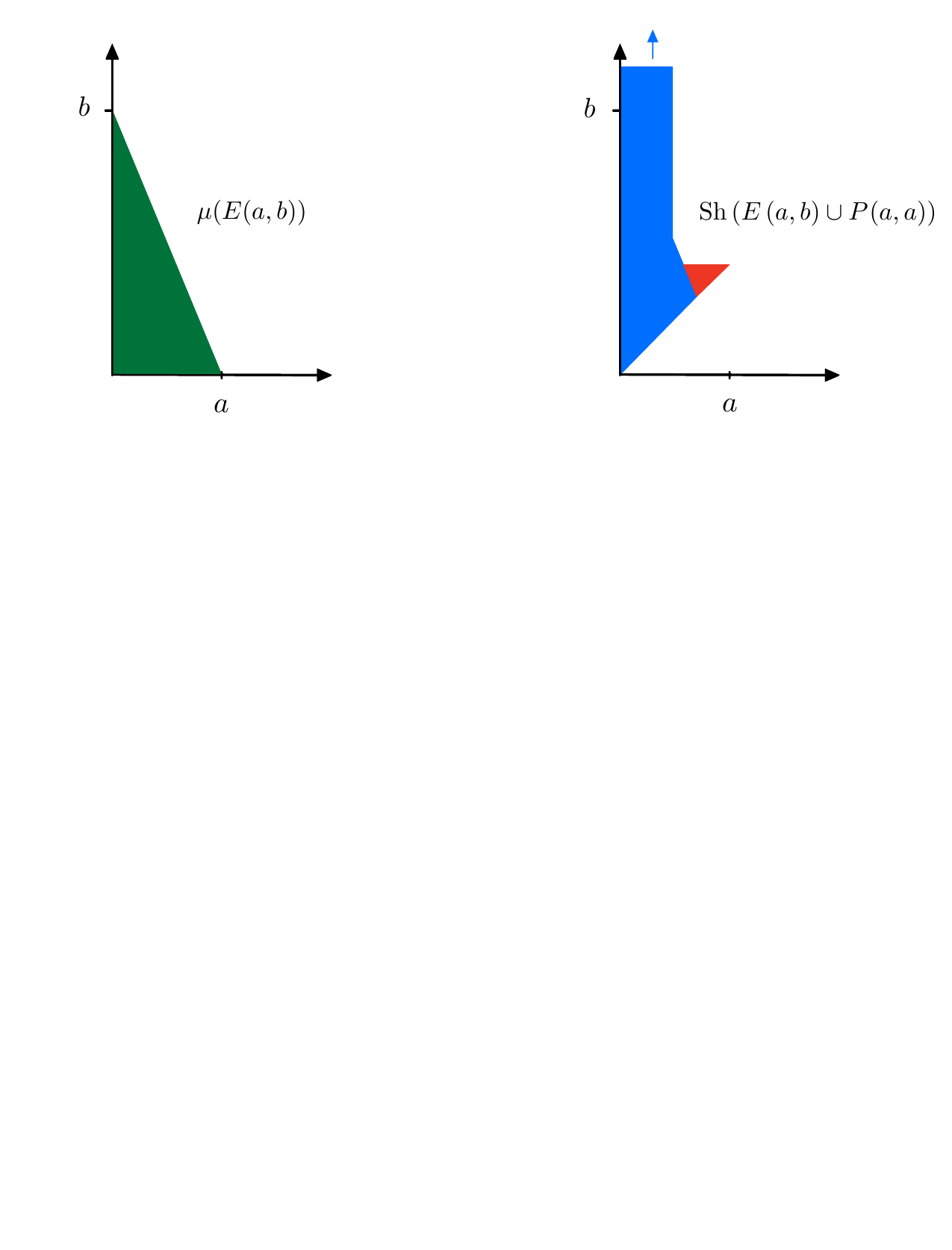}
    \vspace{-10.5cm}
    \caption{It is not known when, if ever, a Lagrangian torus $L(r,s)$ with $(r,s)$ in the red region, $\mathcal{T}(a,b)$, can be symplectically embedded into $E(a,b)$.}
    \label{fig:ellipse}
\end{figure}

{\bf Acknowledgements.} The authors would like to thank Felix Schlenk for detailed comments on a draft, and also Jun Zhang for his interest and encouragement.

\section{The proof of Theorem \ref{lagint} modulo two propositions}

In this section we present the proof of Theorem \ref{lagint} assuming two propositions whose proofs are presented in Section \ref{foliation} and Section \ref{dbuildings}. More precisely, we first present a proof of a slightly weaker result, Theorem \ref{center}, below. In Section \ref{adjust}, we then show that the proof of Theorem \ref{center} can be extended to cover Theorem \ref{lagint}.
\begin{theorem}\label{center}
Suppose that  $1/2\leq r<1$ and $1 \leq s <b$. If $\phi$ is a Hamiltonian diffeomorphism of $\RR^4$ that maps $L(r,s)$ into $P(1,b)$, then $\phi(L(r,s))$ intersects the codimension one submanifold
\begin{equation*} 
\bigcup_{t \in [s,b)} L(r, t).
\end{equation*}
\end{theorem}

\begin{remark}
    The choice to present Theorem \ref{lagint} as an extension of Theorem \ref{center} is made for the sake of clarity. The proof of Theorem \ref{center} contains all the key technical ideas and arguments, and the extension to Theorem \ref{lagint} is easily described in these terms. 
\end{remark}

\subsection{Proof of Theorem \ref{center}}

To prove Theorem \ref{center} we argue by contradiction. The assumption to be contradicted is the following.

\begin{assumption}\label{assump}
There are real numbers $r$, $s$ satisfying 
\[ \frac{1}{2} \leq r <1 \leq s,\]
a number $b>s$, and a Hamiltonian diffeomorphism $\phi$ of $\RR^4$ that maps $L(r,s)$ into $P(1,b)$, such that  $\phi(L(r,s))$ is disjoint from the family of Lagrangian tori $\bigcup_{t \in [s,b)} L(r, t)$.
\end{assumption}

\bigskip

To simplify our notation we set $L_1=L(r,s)$, $L_2= \phi(L(r,s))$ and $$L[t] = L(r, s-1+t)$$ for $t \in [1,b-s+1)$. With this, Assumption \ref{assump} asserts that $L_2$ is  disjoint from each of the $L[t]$. As in \cite{hiker}, the general philosophy behind the proof is that the assertion: {\em $L_1$ and $L_2$ are disjoint Lagrangian submanifolds in $(M,\omega)$} implies an identical assertion for any symplectic manifold that can be obtained from $(M,\omega)$ by applying standard symplectic procedures in the complement of $L_1 \cup L_2$; like compactification, symplectic blow-up, symplectic blow-down and inflation. Such changes of the ambient symplectic manifold will be referred to as {\em scene changes}. To contradict Assumption \ref{assump}, we will show that if $L_2$ is disjoint from the $L[t]$, then we can construct a scene in which a classical Lagrangian intersection result of Gromov implies that $L_1$ and $L_2$ must intersect.  

\bigskip
\noindent{\bf Scene 1.} Let $\omega$ be the standard symplectic form on the sphere $S^2$ with total area equal to $1$. For $\rho>0$, denote the symplectic manifold $(S^2, \rho \omega)$ by $S^2(\rho)$. For our first scene change we compactify the polydisk $P(1,b)$ to the symplectic manifold $$(X_1,\Omega_1) := (S^2 \times S^2, \omega \oplus b \, \omega) =S^2(1) \times S^2(b).$$ 

In this scene, $L_2$ is still disjoint from each of the $L[t]$ and all these tori lie in the complement of the union of symplectic spheres $S_{\infty} \cup T_{\infty}$, where $S_{\infty} = \{\infty\} \times S^2(b)$ and $T_{\infty} = S^2(1) \times \{\infty\}$.
 


\bigskip
\noindent{\bf Scene 2.}
Next, we inflate along the symplectic sphere $S_{\infty}$ by adding a tubular neighborhood of capacity $2r-1$. Denoting the resulting symplectic manifold by $(X_2, \Omega_2)$ we note that $(X_2, \Omega_2)$ is symplectomorphic to $S^2(2r) \times S^2(b)$. 
Here again, $L_2$ is still disjoint from each of the $L[t]$ and all these tori lie in the complement of $S_{\infty} \cup T_{\infty}$ (this time viewed as symplectic spheres in $(X_2,\Omega_2)$ so that $S_{\infty} = \{\infty\} \times S^2(b)$ and $T_{\infty} = S^2(2r) \times \{\infty\}$).

In Section \ref{foliation} we extend the theory of relative finite energy foliations from \cite{rgi} and the techniques developed in \cite{hiker,boust} to prove the following result.

\begin{proposition}\label{prop:proj} There is an $\Omega_2$-tame (or compatible) almost complex structure $J_2$ on $X_2$  and a continuous (projection) map $p \colon X_2 \to S_{\infty}$ with the following properties.
\begin{enumerate}
    \item The image $p(T_{\infty})$ is a point.\\
    \item For every $z$ in the complement of $p(L_1) \cup p(L_2)$ the set $p^{-1}(z)$ is a $J_2$-holomorphic sphere in the class $(1,0) \in H_2(X_2;\ZZ) = \ZZ \times \ZZ$.\\
    \item $p(L_1)$ and $p(L_2)$ are disjoint embedded closed curves in $S_{\infty}$.\\
    \item The component of $S_\infty \smallsetminus p(L_1)$ that contains $p(T_\infty)$ is contained in the component of $S_\infty \smallsetminus p(L_2)$ that contains $p(T_\infty)$.
\end{enumerate}
\end{proposition}

\begin{remark} The proof of the final assertion is precisely where we use the assumption that $L_2$ is disjoint from $L[t]$ for all $t \in [1,b-s+1)$.
\end{remark}

Let $A \subset S_\infty$ be the component of $S_\infty \smallsetminus p(L_2)$ that does not contain $p(T_\infty)$, and let $C$ be the component of $S_\infty \smallsetminus p(L_1)$ that does. Then $A$ and $C$ are disjoint, and their complement is a cylinder, $B$, as pictured in Figure \ref{fig:abC}.

\begin{figure}[!h]
    \centering
    \includegraphics[width=0.6\linewidth]{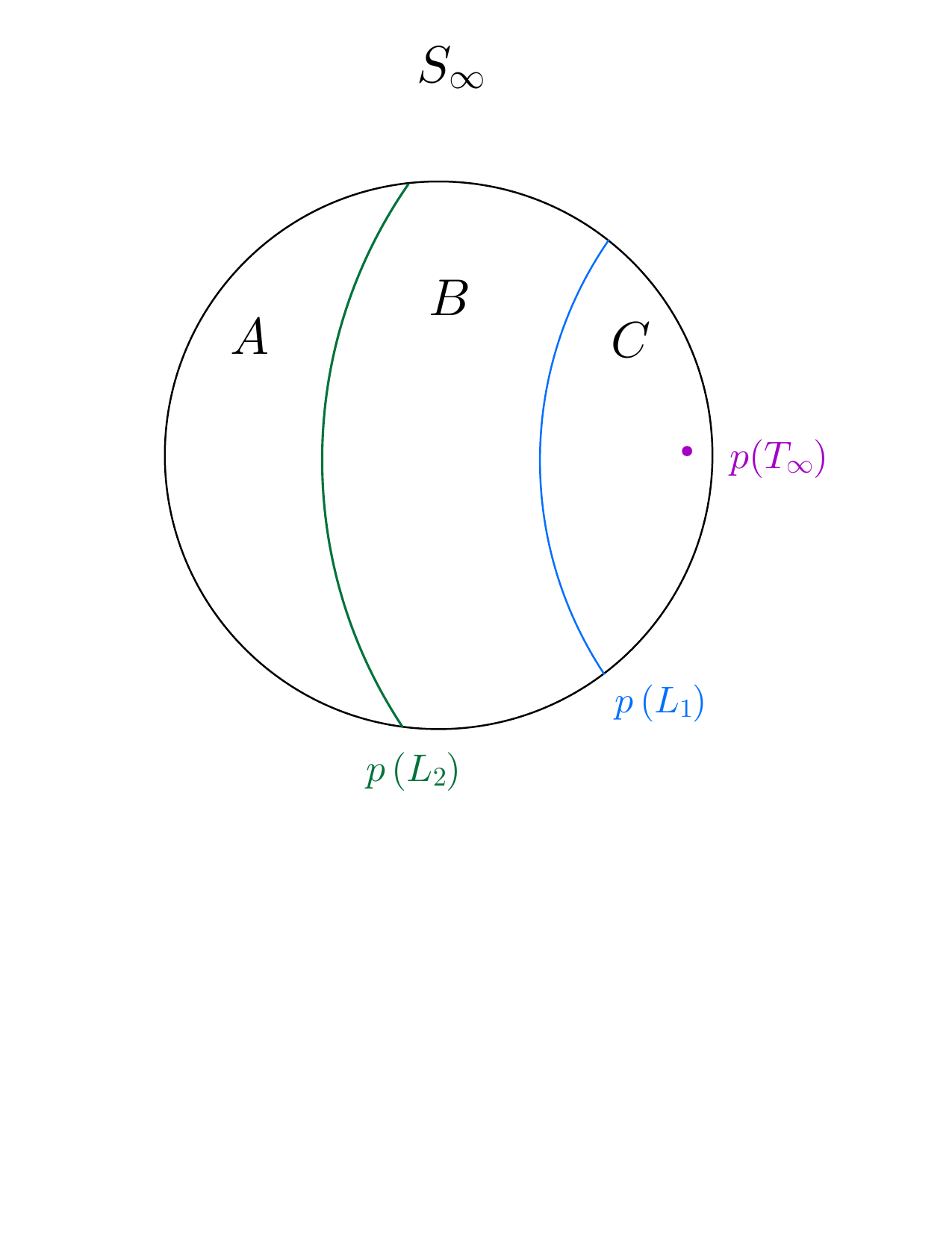}
    \vspace{-2.5cm}
    \caption{Projections of $L_1$ and $L_2$ onto $S_{\infty}$ under Assumption \ref{assump}.}
    \label{fig:abC}
\end{figure}

\medskip

Choose a point $z_0$ in $A$ and set $T_0 = p^{-1}(z_0).$ By Proposition \ref{prop:proj}, $T_0$ is a regular $J_2$-holomorphic sphere that is disjoint from $L_1 \cup L_2$. In Section \ref{dbuildings}, we study buildings obtained by taking the limit of high degree pseudo-holomorphic curves in $(X_2, \Omega_2)$ with point constraints along $L_1 \cup L_2$ as we stretch-the-neck along these Lagrangians. These buildings are used, together with Proposition \ref{prop:proj}, to prove the following.

\begin{proposition}\label{prop:submanifolds}
For every sufficiently large $d \in \NN$, there exist symplectic embeddings
\begin{align*}
    &F,G \colon S^2 \to X_2 \smallsetminus (L_1 \cup L_2) \\
    &D_A, D^0_2, D^\infty_2 \colon  (D^2,S^1) \to (X_2,L_2) \\
    &D_C, D^0_1, D^\infty_1 \colon  (D^2,S^1) \to (X_2,L_1) \\
    &O_B \colon (S^1 \times [0,1],S^1 \times \{0,1\}) \to (X_2,L_1 \cup L_2)
\end{align*}
such that:\\

\begin{enumerate}

\item  The maps $F$ and $G$ both represent the class $(d,1) \in H_2(X_2; \ZZ)$\\

\item The six disks all have Maslov index $2$.\\

\item The set $\{[\partial D_A], [\partial D^\infty_2]\}$ is an integral basis of $H_1(L_2;\ZZ)$, and $[\partial  D^0_2]=-[\partial  D^\infty_2]$.\\

\item The set $\{[\partial D_C], [ \partial D^\infty_1]\}$ is an integral basis of $H_1(L_1;\ZZ)$, and $[\partial  D^0_1]=-[\partial D^\infty_1]$.\\

\item Denoting the component of the boundary of $O_B$ on $L_i$ by $\partial^i O_B$, the set $\{[\partial^i O_B], [\partial D^\infty_i]\}$ is an integral basis of $H_1(L_i;\ZZ)$\\

\item The images of  $p \circ D_A$, $p \circ O_B$ and  $p \circ D_C$ are the closures of $A$, $B$ and $C$, respectively.\\

\item The symplectic areas of  $D_A$, $O_B$ and  $D_C$ are $s$, $r$ and $b-s$, respectively. The disks $D^\infty_1$ and $D^\infty_2$ both have symplectic area equal to $r$.\\

\item  All nine maps are $J_2$-holomorphic away from arbitrarily small neighborhoods of a collection of Lagrangian tori whose elements are close to, and Lagrangian isotopic to, either $L_1$ or $L_2$.\\

\item We have the following intersection pattern:\\
\begin{enumerate}
    \item $F \bullet G=2d$
    \item $F \bullet D^0_i + F \bullet D^\infty_i = G \bullet D^0_i + G \bullet D^\infty_i =1$ for $i=1,2$
    \item $\{F \bullet D^\infty_1, G \bullet D^\infty_1\} = \{F \bullet D^\infty_2, G \bullet D^\infty_2\} = \{0,1\}$    
    \item $F \bullet D_A + G \bullet D_A = d$
    \item $F \bullet D_C + G \bullet D_C = d$
    \item $F \bullet O_B =1$ and $G \bullet O_B=0$
    \item $D^\infty_1 \bullet S_\infty = D^\infty_1 \bullet S_\infty =1$
    \item $F \bullet T_0 = F \bullet T_\infty=G \bullet T_0 = G \bullet T_\infty=1$.\\
\end{enumerate}




\item The spheres $F$ and $G$ have exactly $2d$ intersection points. Of these, $d$ are located in $p^{-1}(A)$ and $d$ are located in $p^{-1}(C)$.\\

\end{enumerate}

\end{proposition}

\bigskip

\noindent{\bf Scene 3.} Let $x_1, \dots, x_{2d}$ be the intersection points of $F$ and $G$. We may assume that $x_1, \dots, x_{d}$ lie in $p^{-1}(A)$ and that $x_{d+1}, \dots, x_{2d}$ lie in $p^{-1}(C)$. Let $H_i$ be the fiber $p^{-1}(p(x_i))$. By the second assertion of Proposition \ref{prop:proj}, the $H_i$ are disjoint symplectic spheres in the class $(1,0)$.

The next scene is obtained by performing a sequence of small symplectic blow-ups followed by a sequence of symplectic blow-downs. First we choose $J_2$ to be $\Omega_2$-compatible, from the start, and perturb it, as in \cite{boust}, to an almost complex structure $J_3$ for which $F$ and $G$ are complex, and which is the standard integrable structure in Darboux balls of capacity $\varepsilon>0$ centered at each of the $x_i \in F \cap G$. 
We may also assume that $F$, $G$ and the $H_i$ coincide with complex planes in the balls.
We then blow up these Darboux balls. The proper transforms of $F$, $G$ and the $H_i$ are well defined smooth symplectic surfaces because of our standard coordinates in the balls. Hence we can blow down the corresponding proper transforms of the $H_i$, symplectically again. The resulting manifold, $(X_3, \Omega_3)$, is diffeomorphic to $S^2 \times S^2$. Given a symplectic submanifold $S$ of $(X_2, \Omega_2)$, we will denote its proper transform in $(X_3, \Omega_3)$ by $\hat{S}$. The symplectic submanifolds $\hat{F}$ and $\hat{G}$ are disjoint spheres in $X_3$ and their self-intersection numbers are both $0$. The same is true of $\hat{T}_0$ and $\hat{T}_{\infty}$.  Moreover, both $\hat{T}_0$ and $\hat{T}_{\infty}$ intersect each of $\hat{F}$ and $\hat{G}$ positively in a single point. In other words, the collection $\hat{F} \cup  \hat{G} \cup \hat{T}_0 \cup \hat{T}_{\infty}$ serves as a set of (topological) axes in the new copy of $S^2\times S^2$ and these axes are disjoint from $L_1 \cup L_2$.
Note, however, that $(X_3, \Omega_3)$ is not monotone. The spheres $\hat{T}_0$ and $\hat{T}_{\infty}$ have area $2r$ while the spheres $\hat{F}$ and $\hat{G}$ have area $2dr + b - 2d \varepsilon$.

\bigskip

\noindent{\bf Scene 4.} Finally, we work away from $\hat{F} \cup  \hat{G} \cup \hat{T}_0 \cup \hat{T}_{\infty}$ and view $L_1$ and $L_2$ as disjoint Lagrangian tori in the open symplectic manifold 
$$X_4 := X_3 \smallsetminus ( \hat{F} \cup  \hat{G} \cup \hat{T}_0 \cup \hat{T}_{\infty} )$$
equipped the restriction of $\Omega_3$ which we denote by $\Omega_4$. By classification results for symplectic spheres, see \cite{mcdrat}, this final scene can be identified with the subset of $T^* T^2 = T^2 \times\RR^2$ defined by restricting the momentum coordinates to the rectangle with side lengths $2r$ and $2dr + b - 2d \varepsilon$.

\begin{lemma}\label{lem:pre}
    The Lagrangian tori $L_1$ and $L_2$ are relatively exact in $(X_4, \Omega_4)$ and both are homologous to the zero section.
\end{lemma}

\begin{proof}

To prove that $L_1$ and $L_2$ are relatively exact, it suffices to find two relative cycles,  $C_1$ and $C_2$, representing classes in $H_2(X_4, L_1 \cup L_2;\ZZ)$, such that 
\begin{equation}
\Omega_4(C_1) = \Omega_4(C_2) =0
\end{equation}
\begin{equation}
\{[\partial C_1 \cap L_1], [\partial C_2 \cap L_1]\}\quad \text{is a basis of} \quad   H_1(L_1;\ZZ),
\end{equation}
and 
\begin{equation}
\{[\partial C_1 \cap L_2], [\partial C_2 \cap L_2]\} \quad \text{is a basis of} \quad H_1(L_2;\ZZ).
\end{equation}

To produce these cycles we will use the transforms of the submanifolds of $X_2$ described in Proposition \ref{prop:submanifolds}.

By Proposition \ref{prop:submanifolds} we can choose disks $D_1$ equal to one of the transforms $\hat{D}^0_1$ or $\hat{D}^\infty_1$, and $D_2$ equal to one of the transforms $\hat{D}^0_2$ or $\hat{D}^\infty_2$, such that both disks intersect $\hat{F}$ exactly once but are disjoint from $\hat{G}$. We denote the intersection points with $\hat{F}$ by $q_1$ and $q_2$ respectively.
Let $\gamma_1$ be a smooth curve in $\hat{F}$ connecting $q_1$ to $q_2$, and $\Gamma_1$ a smooth cylinder in the normal bundle to $\hat{F}$ such that $\Gamma_1$ is disjoint from $\hat{F}$ and projects to $\gamma_1$. We may assume $\Gamma_1$ intersects $D_1$ and $D_2$ in circles around their intersections with $\hat{F}$. The circles bound disks $E_1$ and $E_2$ such that $D_1 \smallsetminus E_1$ and $D_2 \smallsetminus E_2$ are cylinders disjoint from $\hat{F}$, and we can define $C_1$ to be the glued cycle $C_1 = (D_1 \smallsetminus E_1) \cup \Gamma_1 \cup (D_2 \smallsetminus E_2)$. 
It clearly has a zero $\Omega_4$-area. Moreover,  $[\partial C_1 \cap L_1]=-[\partial D_1]$ and $[\partial C_1 \cap L_2]=[\partial D_2]$.

To construct the cycle $C_2$, we first consider the transform $\hat{O}_B$ of $O_B$. It follows from assertion (5) of Proposition \ref{prop:submanifolds} that 
$\{[\partial C_1 \cap L_i], [\partial \hat{O}_B \cap L_i]\}$ is a basis of  $H_1(L_i;\ZZ)$ for $i=1,2$.
However the $\Omega_4$-area of $\hat{O}_B$ is $r$. Moreover, $\hat{O}_B$ intersects the axis $\hat{F}$ positively in precisely one point and so does not represent a class in $H_2(X_4, L_1 \cup L_2;\ZZ)$. Both problems can be fixed simultaneously by choosing a smooth curve $\gamma_2$ connecting $q_1$ to $\hat{O}_B \cap \hat{F}$. We can then form a cylinder $\Gamma_2$ over $\gamma_2$ and use it to glue $D_1$ and $\hat{O}_B$ to produce the required cycle $C_2$ in the complement of $\hat{F}$ as above.

It remains to show that $L_1$ (or $L_2$) is homologous to the zero section. It is enough to show that the inclusion $H_1(L_1;\ZZ) \to H_1(X_4;\ZZ)$ is injective or, equivalently, that the boundary map $$\partial \colon H_1(X_4,L_1;\ZZ) \to H_1(L_1;\ZZ)$$ is trivial. If the boundary map is not trivial then there is a relative 2-cycle $C$ such that $[\partial C] = (m,n) \neq (0,0) \in H_1(L_1, \ZZ).$ Note that $C$ also represents a class in $H_2(X_3, L_1;\ZZ)$. Hence, adding suitable multiples of $\hat{D}_C$ and $\hat{D}_1^{\infty}$ to $C$ we get a closed cycle $C_{m,n}$ in $X_3$ whose intersections with the topological axes are determined by the disks added. If $m\neq 0$, then $C_{m,n}$, like $\hat{D}_C$,  will intersect $\hat{T}_0$ but not $\hat{T}_{\infty}$. Otherwise, if $m=0$ and $n\neq0$, then $C_{m,n}$, like $\hat{D}_1^{\infty}$,  will intersect $\hat{F}$ but not $\hat{G}$. In both cases, the resulting intersection pattern is impossible since $[\hat{T}_0]=[\hat{T}_{\infty}] \in H_2(X_3;\ZZ)$ and $[\hat{F}] =[\hat{G}]\in H_2(X_3;\ZZ)$.
\end{proof}

Lemma \ref{lem:pre} immediately yields the desired contradiction. Within $(X_4, \Omega_4) \subset (T^*T^2, d \lambda)$, it follows from Lemma \ref{lem:pre} and Theorem B of \cite{rgi}, that $L_1$ and $L_2$ are both Hamiltonian isotopic to the same constant section of $T^*T^2$. With this, Gromov's Theorem 2.3.$B_4''$ from \cite{gr} implies that $L_1$ and $L_2$ must intersect.  This contradicts Assumption \ref{assump}. To prove Theorem \ref{center} then, it remains for us to prove Proposition \ref{prop:proj} and Proposition \ref{prop:submanifolds}.

\subsection{Proof of Theorem \ref{lagint}}\label{adjust}
 
To extend the proof of Theorem \ref{center} to a proof of Theorem \ref{lagint}, we need to consider a third Lagrangian torus, $L_3=L(r, b-s)$. The (weaker) standing assumption is now as follows.

\begin{assumption}\label{assump2}
There are real numbers $r$, $s$ satisfying 
\[ \frac{1}{2} \leq r <1 \leq s,\]
a number $b>2s$, and a Hamiltonian diffeomorphism $\phi$ of $\RR^4$ that maps $L(r,s)$ into $P(1,b)$, such that  $\phi(L(r,s))$ is disjoint from the family of Lagrangian tori $\bigcup_{t \in [s,b-s]} L(r, t)$ from $L_1$ to $L_3$.
\end{assumption}

Under this assumption, relative foliation theory implies the following analogue of Proposition \ref{prop:proj}.

\begin{proposition}\label{adjustprop} There is an $\Omega_2$-tame (or compatible) almost complex structure $J_2$ on $X_2$  and a continuous (projection) map $p \colon X_2 \to S_{\infty}$ with the following properties.
\begin{enumerate}
    \item The image $p(T_{\infty})$ is a point.\\
    \item For every $z$ in the complement of $p(L_1) \cup p(L_2)  \cup p(L_3)$ the set $p^{-1}(z)$ is a $J_2$-holomorphic sphere in the class $(1,0) \in H_2(X_2;\ZZ) = \ZZ \times \ZZ$.\\
    \item The images $p(L_i)$, for $i= 1,2,3$ are disjoint embedded closed curves in $S_{\infty}$.\\
    \item The component of $S_\infty \smallsetminus p(L_3)$ that contains $p(T_\infty)$ is contained in the component of $S_\infty \smallsetminus p(L_1)$ that contains $p(T_\infty)$. As well, the projection $p(L_2)$ is contained in the component of $S_\infty \smallsetminus p(L_3)$ that contains $p(T_\infty)$ or it is contained in the component of $S_\infty \smallsetminus p(L_1)$ that does not contain $p(T_\infty)$.
\end{enumerate}
\end{proposition}

Let $\mathcal{A}$ to be the component of $S_\infty \smallsetminus p(L_1)$ that does not contain $p(T_\infty)$, and $\mathcal{C}$ be the component of $S_\infty \smallsetminus p(L_3)$ that contains $p(T_\infty)$. Then $\mathcal{A}$ and $\mathcal{C}$ are disjoint disks and their complement is a cylinder, 
$\mathcal{B}$. This is pictured in Figure \ref{fig:abc2}. The three possible positions of $p(L_2)$ are pictured in Figure \ref{fig:abc-l2}.

\begin{figure}[!h]
    \centering
    \includegraphics[width=0.6\linewidth]{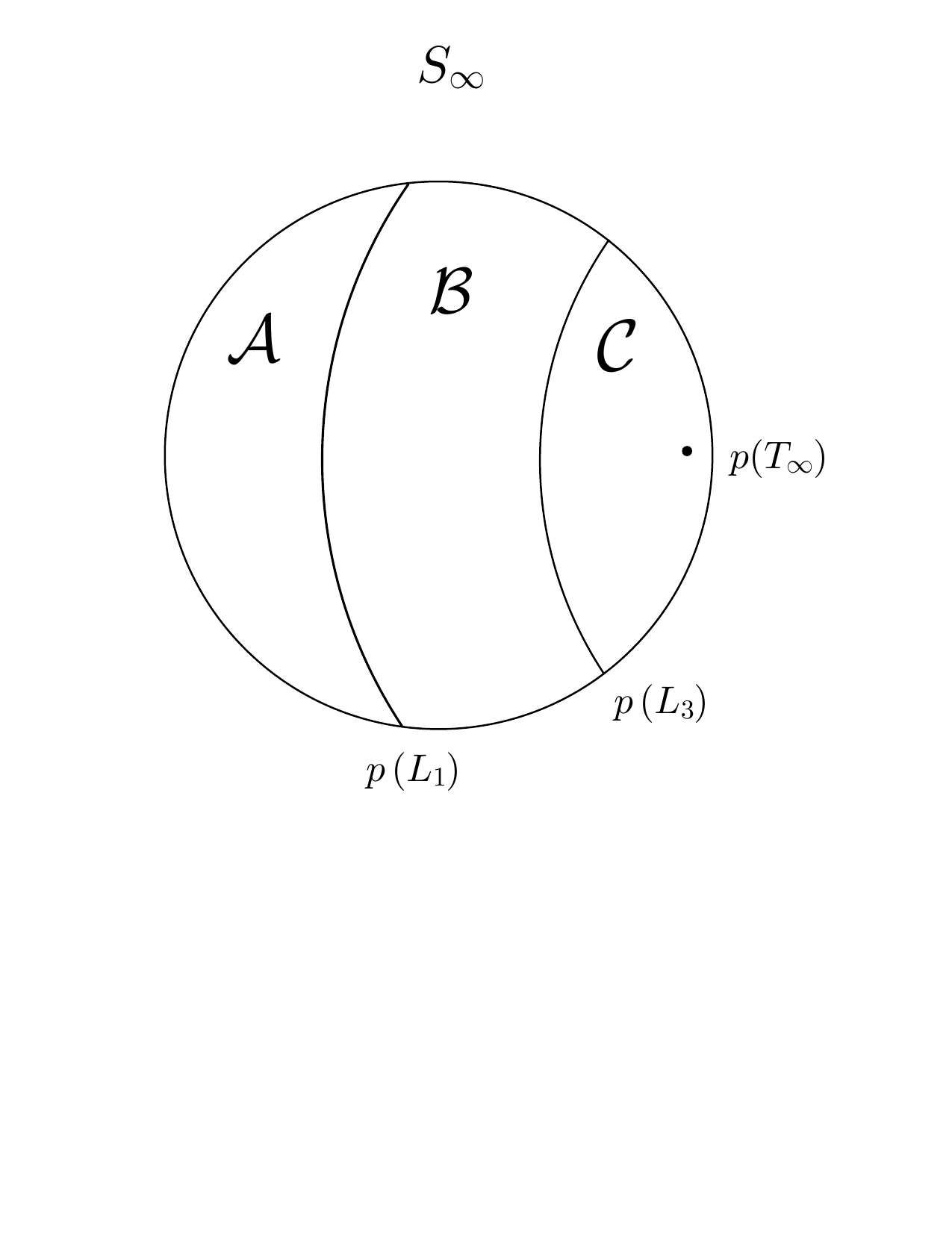}
    \vspace{-2.5cm}
    \caption{Projections of $L_1$ and $L_3$ onto $S_{\infty}$ under Assumption \ref{assump2}.}
    \label{fig:abc2}
\end{figure}

\medskip

\begin{figure}[!h]
    \centering
    \includegraphics[width=0.9\linewidth]{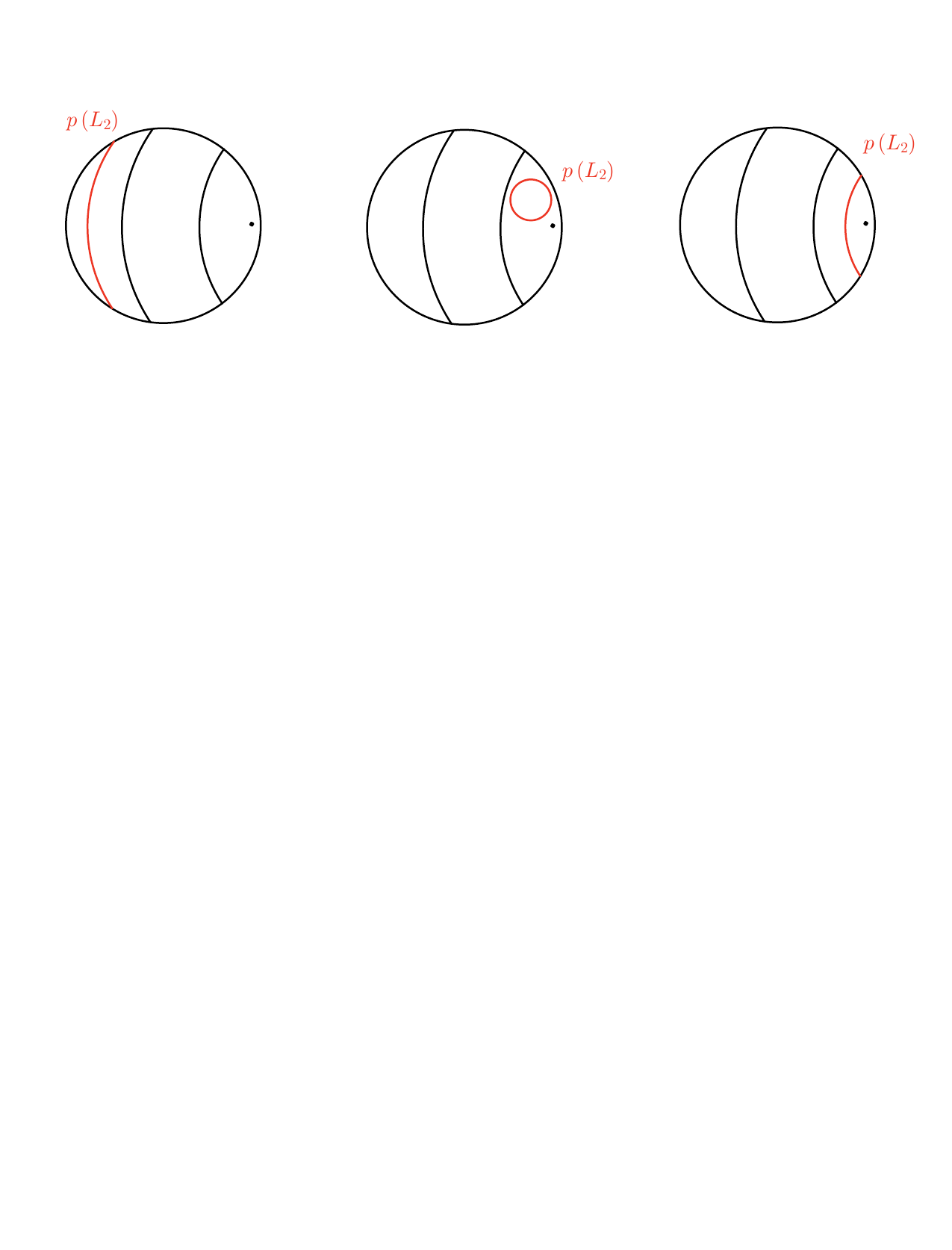}
    \vspace{-10.5cm}
    \caption{The three possibilities for the location of $p(L_2)$.} 
    \label{fig:abc-l2}
\end{figure}


If we denote the component of $S_\infty \smallsetminus p(L_2)$ that does not contain $p(T_\infty)$ by $A$, then we have the following three cases to consider: 

\begin{enumerate}
    \item $A$ is contained in $\mathcal{A}$,
    \item $A$ is contained in $\mathcal{C}$,
    \item $A$ contains both $\mathcal{A}$ and $\mathcal{B}$.
\end{enumerate}
As summarized in Figure  \ref{fig:abc3}, for each of these three cases the proof of Theorem \ref{lagint} can be reduced to that of Theorem \ref{center}.
In Case 1 we are in the exact scenario described in statement (4) of Proposition \ref{prop:proj} and the proof of Theorem \ref{center} applies directly. For Case 2, we fix a point $z_0 \in \mathcal{A}$ and set $T_0 = p^{-1}(z_0)$. Then, substituting $T_0$ for $T_{\infty}$ and $L_3$ for $L_1$, we can again follow the proof of Theorem \ref{center}. 
In the end, the argument implies that $L_2$ must intersect $L_3$, which contradicts  Assumption \ref{assump2}.
Finally, in Case 3, the proof of Theorem \ref{center} can be applied directly, but with the roles of $L_1$ and $L_2$ reversed. 

\begin{figure}[!h]
    \centering
    \includegraphics[width=0.9\linewidth]{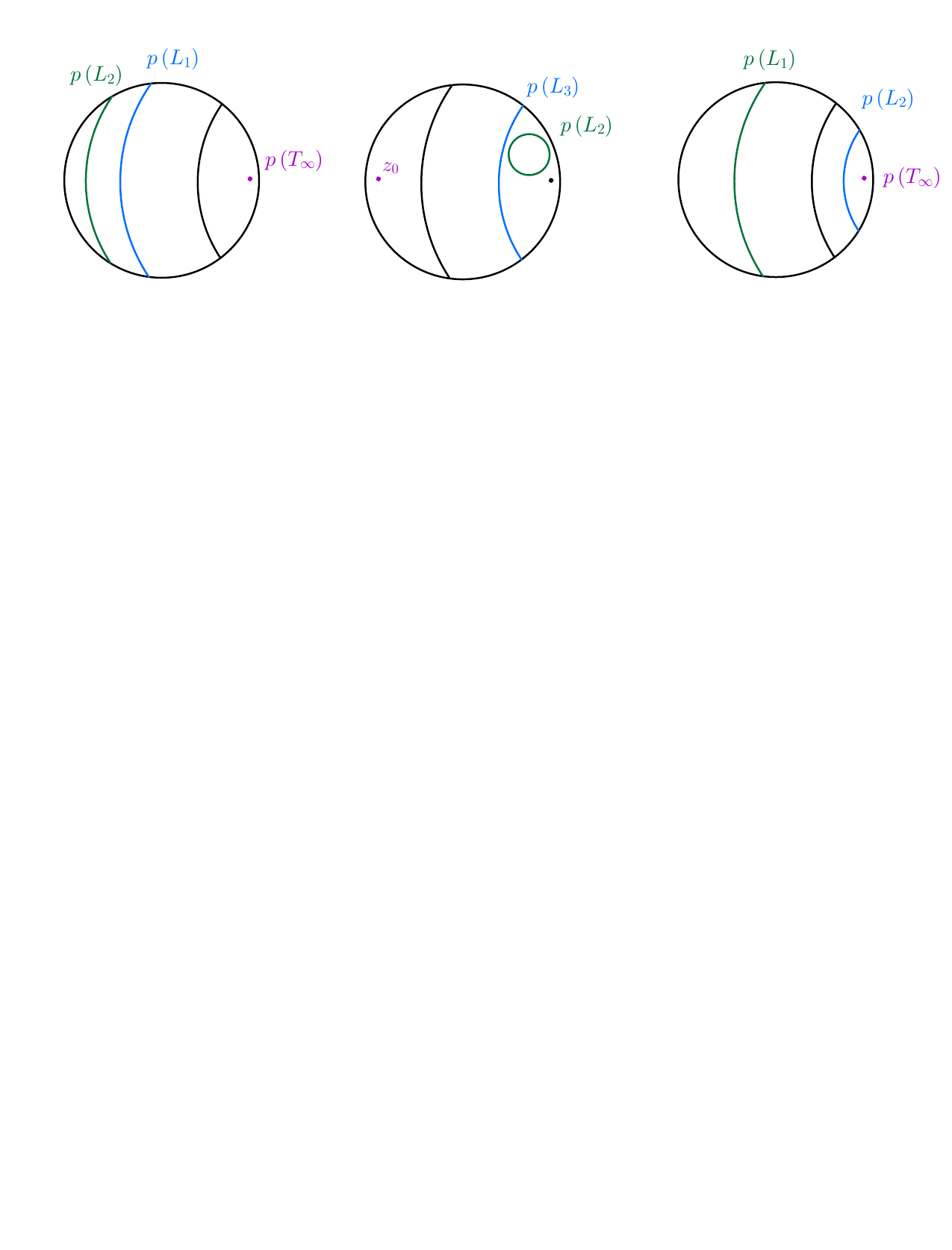}
    \vspace{-10.5cm}
    \caption{The colorings, when compared to that Figure \ref{fig:abC}, encode the manner in which the proof of Theorem \ref{lagint} can be reduced to that of Theorem \ref{center} in each case.}
    \label{fig:abc3}
\end{figure}

\medskip

\subsection{More Lagrangian Intersections} \label{sec:2lag}

As an alternative to taking the limit $b\to 2s$, Corollary \ref{newlag} could also be obtained by following the argument in Section \ref{adjust} but setting $L_1 = L_3$. In this case, $\mathcal{B} = \emptyset$ and Proposition \ref{adjustprop} (4) is automatic, that is, we do not need to study the $1$-parameter family of tori. Hence, rather than assuming $L_1 = L(r,s)$, we may instead assume that $L_1 = \phi_1(L(r,s))$ for some Hamiltonian diffeomorphism $\phi_1$ of $\RR^4$ that maps $L(r,s)$ into $P(1,b)$. With this we get the following more general intersection result.

\begin{corollary}\label{newlagv2}
Suppose that $1/2\leq r<1$ and $1 \leq s <b \le 2s$. If $\phi_1$ and $\phi_2$ are Hamiltonian diffeomorphisms of $\RR^4$ which each map $L(r,s)$ into $P(1,b)$, then $$\phi_1(L(r,s)) \cap \phi_2(L(r,s)) \neq \emptyset.$$
\end{corollary}

\begin{remark} Corollary \ref{newlagv2} fails if we simply assume that $L_3$ has the Maslov and area class of $L(r,s)$. Indeed, there is an isomorphism from $H_1(L(3/5, 4/5))$ to $H_1(L(4/5,1))$ which preserves both Maslov and area classes (while the product tori are clearly disjoint in $P(1,2)$). The assumption on the Hamiltonian isotopy class is used in Corollary \ref{>r}.
\end{remark}





\section{The proof of Proposition \ref{prop:proj}}\label{foliation}

We first show that the relative foliation theory from \cite{ivriit} and \cite{rgi} can be applied to the pair $L_1 \cup L_2 \subset X_2$, despite the fact that neither Lagrangian torus is monotone. To simplify our notation, we set $(X_2,\Omega_2) = (X,\Omega)$.

\subsection{Foliations of $(X, \Omega)$ relative to $L_1 \cup L_2$.}\label{subsec:morefol} 

Let $J$ be an $\Omega$-tame almost complex structure on $X$ such that $J$ is equal to the standard split complex structure near $S_{\infty} \cup T_{\infty}$. 
In fact, by the construction of $X=X_2$, in Scene $2$, we may assume that $J$ is standard in a tubular neighborhood of $S_{\infty}$ of capacity $2r-1$. This implies that $J$-holomorphic curves intersecting $S_{\infty}$ have area at least $2r-1$.

Since $J$ is standard near $T_{\infty}$, it follows from Theorem 0.2.A of \cite{gr} that the $J$-holomorphic sphere $T_{\infty}$ extends to a foliation $\mathcal{F}(J)$ of $X$ comprised of $J$-holomorphic spheres in the class $[S^2 \times \{pt\}]$. 
Moreover, there is a smooth projection map from $X$ to $S_{\infty}$ defined by taking a point $p$ in $X_2$ to the unique intersection point of $S_{\infty}$ and the leaf of $\mathcal{F}(J)$ that contains $p$.

To produce a relative version of this picture, we restrict $J$ further. For each Lagrangian $L_i$ we fix a parameterization  $\psi_i \colon \mathbb{T}^2 \to L_i$. This extends to a symplectic embedding  $\Psi_i \colon \mathcal{U}(\epsilon) \to X$ where $$\mathcal{U}(\epsilon) = \{|p_1| <\epsilon,\, |p_2|< \epsilon\} \subset T^* \mathbb{T}^2$$ for some sufficiently small $\epsilon>0.$  Set $$\mathcal{U}_i =\Psi_i(\mathcal{U}(\epsilon)).$$ We assume that $J$ is {\it compatible} with the parameterizations $\psi_i$ in the sense that in each neighborhood $\mathcal{U}_i$ we have 
$$J \frac{\partial}{\partial q_i} = - \frac{\partial}{\partial p_i}.$$
Given this, $J$ can be stretched along $L_1 \cup L_2$, as in \cite{behwz}, to yield  an almost complex structure $J_{\infty}$ on $X \smallsetminus (L_1 \cup L_2)$. 
This new almost complex structure is {\it adapted} to the parameterizations $\psi_i$, i.e. in $\mathcal{U}_i \smallsetminus L_i$ we have $$J_\infty \frac{\partial}{\partial q_i} = -  \sqrt{p^2_1 + p_2^2}\,\frac{\partial}{\partial p_i}.$$ The almost complex structure $J_\infty$ is also standard near $S_{\infty} \cup T_{\infty}$. 
As described in \cite{behwz}, in stretching along $L_1 \cup L_2$, the curves of the Gromov foliation $\mathcal{F}(J)$ converge to a collection of holomorphic buildings, $\bar{\mathcal{F}}_{L_1,L_2}(J)$. Each building consists of genus zero pseudo-holomorphic curves of three possible level types. The {\em top level} curves  map to $X \smallsetminus (L_1 \cup L_2)$ and are $J_{\infty}$-holomorphic. The {\em middle level} curves map to one of two copies of the symplectization of the unit cotangent bundle of the flat torus that correspond to 
$L_1$ and $L_2$. These middle level curves are all $J_{\mathrm{cyl}}$-holomorphic  where $J_{\mathrm{cyl}}$ is the cylindrical almost-complex structure defined by $$J_{\mathrm{cyl}} \frac{\partial}{\partial q_i} = -  \sqrt{p^2_1 + p_2^2}\,\frac{\partial}{\partial p_i}.$$ The {\em bottom level} curves map to one of two copies of $T^*\mathbb{T}^2$  that again correspond to 
$L_1$ and $L_2$.\footnote{In what follows we will identify the copy of $T^*\mathbb{T}^2$ corresponding to $L_i$ with $T^*L_i$.} These bottom level curves are $J_{\mathrm{std}}$-holomorphic where  $J_{\mathrm{std}}$ is of the form 
$$J_{\mathrm{std}} \frac{\partial}{\partial q_i} = -  \rho \left(\sqrt{p^2_1 + p_2^2}\right)\,\frac{\partial}{\partial p_i}$$ for a fixed positive increasing function $\rho(t)$ such that $\rho(0)=1$ and $\rho(t) =t$ for $t \ge 2.$

A generic building of the collection $\bar{\mathcal{F}}_{L_1,L_2}(J)$ is (still) a simple $J$-holomorphic sphere in $X \smallsetminus (L_1 \cup L_2)$. For example, (the standard parameterization of) $T_{\infty}$ is such a building. The other buildings in $\bar{\mathcal{F}}_{L_1,L_2}(J)$ are split along either $L_1$ or $L_2$ and have multiple components. 


\begin{definition}\label{one}
    A split building in $\bar{\mathcal{F}}_{L_1,L_2}(J)$ is of {\it Type A} if it has three component curves; two $J_{\infty}$-holomorphic planes  in $X \smallsetminus (L_1 \cup L_2)$ of index $1$ and asymptotic to the same Lagrangian torus, say  $L_i$, and a $J_{\rm std}$-holomorphic cylinder of index $2$ with image in $T^*L_i$. Exactly one of the planes intersects $S_{\infty}$.  
\end{definition}

\begin{definition}\label{two}
    A split building in $\bar{\mathcal{F}}_{L_1,L_2}(J)$ is of {\it Type B} if it has five component curves; two  $J_{\infty}$-holomorphic planes of index $1$ in $X \smallsetminus (L_1 \cup L_2)$, 
    a nontrivial $J_{\infty}$-holomorphic cylinder of index 0 in  $X \smallsetminus (L_1 \cup L_2)$, and  
pseudo-holomorphic cylinders in $T^*L_1 \cup T^*L_2$ of index $2$. 
\end{definition}

\begin{definition}
    An almost complex structure $J$ on $(X, \Omega)$ that is compatible with the parameterized Lagrangian tori $L_1$ and $L_2$ is of {\it foliation type} relative to the parameterized Lagrangian tori $L_1$ and $L_2$ if $\bar{\mathcal{F}}_{L_1,L_2}(J)$ consists only of $J_{\infty}$-holomorphic spheres and split buildings of Type A.
\end{definition}

\begin{lemma}\label{lem:noB}
Suppose  $J$ is standard near  $S_{\infty} \cup T_{\infty}$, compatible with $L_1$ and $L_2$,  and $J_{\infty}$ is regular for somewhere injective curves. Then $J$ is of foliation type relative to $L_1$ and $L_2$.  
\end{lemma}

\subsubsection{Proof of Lemma \ref{lem:noB}}
It follows as in Proposition 3.5 and Proposition 5.3 of \cite{rgi}, that for $J$ as in the statement, every split building in 
$\bar{\mathcal{F}}_{L_1,L_2}(J)$ is either of Type A or Type B. If $L_1$ and $L_2$ are monotone then, as described in Corollary 5.4 of \cite{rgi}, there are no split buildings of Type B. However, in the current setting, $L_1$ and $L_2$ are nonmonotone. 
To preclude buildings of Type B in our setting, we invoke the following result implied by Lemma 3.7 of \cite{ho}.

\begin{lemma}\label{>r2} Let $L$ be a parameterized Lagrangian torus in a symplectic manifold $(M,\omega)$ of dimension four. For an $\omega$-compatible almost complex structure $J$ on $M$ which is compatible with the parameterization of $L$ set $$a(J) = \min\{  \omega(u) \, | \, \mathrm{index}(u) \ge \mathrm{end}(u) \}$$
where the minimum is taken over $J_{\infty}$-holomorphic curves $u$ of genus zero with ends asymptotic to $L$, and $\mathrm{end}(u)$ denotes the number of ends of $u$.

\begin{enumerate}
\item If $M$ is closed, then $a(J)$ is independent of $J$.\\

\item Suppose that $(M,\omega)$ is an open and  geometrically bounded symplectic $4$-manifold. Let $J_{\mathrm{fix}}$ be an $\omega$-compatible almost complex structure such that the Riemannian metric $\omega(\cdot, J_{\mathrm{fix}} \cdot )$ is complete, has bounded sectional curvature and a positive lower bound for its injectivity radius. For a bounded open neighborhood $U$ of $L$, let $\mathcal{J}_{\mathrm{fix}} (L,U)$ be the set of $\omega$-compatible almost complex structures which are compatible with the parameterization of $L$ and equal to  $J_{\mathrm{fix}}$ in the complement of $U$. Then $a(J)$ is constant on $\mathcal{J}_{\mathrm{fix}} (L,U)$.
\end{enumerate}

\end{lemma}

\begin{proof}

In both cases, the proof follows directly from that of  \cite[Lemma 3.7]{ho} which we briefly recall. Arguing by contradiction, we assume that there are $\omega$-compatible almost complex structures $J_0$ and $J_1$ that are compatible with the parameterization of $L$, and lie 
in the same set $\mathcal{J}_{\mathrm{fix}} (L,U)$, such that there is a $J_1$-holomorphic curve $u_1$ with $\mathrm{index}(u_1) \ge \mathrm{end}(u_1)$ and $\omega(u_1) < a(J_0).$  Let $\{J_t\}_{t \in[0,1]}$ be a smooth family of almost complex structures from $J_0$ to $J_1$ which all have the properties of $J_0$ and  $J_1$ described above. As noted in \cite[Lemma 3.7]{ho}, the results from \cite{we}, \cite{ze} and \cite{ohzh} imply that for a generic family $\{J_t\}$  the corresponding family moduli space, $\widetilde{\mathcal{M}}$, of curves in the class of $u_1$, is a smooth finite dimension manifold with boundary corresponding to $t=0$ and $t=1$. Since all curves in the moduli space have the same area, $\omega(u_1)$, and, by assumption, $\omega(u_1)<a(J_0)$, there must be no curves in $\widetilde{\mathcal{M}}$ corresponding to $t=0$. Fixing a constraint point on each of the asymptotic Reeb orbits of $u_1$ and a suitable number of constraint points in the image of $u_1$, we can cut $\widetilde{\mathcal{M}}$ down to a manifold $\mathcal{M}$ of dimension one. Since $\mathcal{M}$ also has no curves corresponding to $t=0$, it must be noncompact. However, as shown in \cite[Lemma 3.7]{ho}, any degeneration, at say $t_c \in (0,1)$, will result in a building which includes a $J_{t_c}$-holomorphic curve $u_{t_c}$ with $\mathrm{index}(u_{t_c}) \ge \mathrm{end}(u_{t_c})$ and $\omega(u_{t_c}) < a(J_0)$. Repeating the argument on the interval $[0,t_c]$, one eventually arrives at such a curve at $t=0$, and hence a contradiction.


\end{proof}

Applying Lemma \ref{>r2} to product Lagrangian tori in $\RR^4$ we get the following useful Corollary.

\begin{corollary}\label{>r}
Suppose that $0< r\leq s$. Let $J$ be any tame almost complex structure on $\RR^4$ that is compatible with a parametrization of $L(r,s)$, and agrees with the standard complex structure, $J_0$, outside of a bounded domain.
If $J_{\infty}$ is the almost complex structure on $\RR^4 \smallsetminus L(r,s)$ resulting from stretching-the-neck along $L(r,s)$, then every $J_{\infty}$-holomorphic plane in $\RR^4 \smallsetminus L(r,s)$ with  deformation index 1 has symplectic area at least $r$. 
\end{corollary}

\begin{proof} For a product Lagrangian $L(r,s)$ with $r\le s$ it is straightforward to show that $a(J_0) =r$.
    
\end{proof}

With Corollary \ref{>r} in hand, we can now complete the proof of Lemma \ref{lem:noB}.
Arguing by contradiction, we assume that $\bar{\mathcal{F}}_{L_1,L_2}(J)$ includes a building $\mathbf{B}$ of Type B. As $T_{\infty}$ is a simple (unbroken) holomorphic sphere, all the curves of $\mathbf{B}$ must be disjoint from $T_{\infty}$. A single curve of $\mathbf{B}$ intersects $S_{\infty}$, and so by monotonicity has area greater than $2r-1$.

If the two top level planes of $\mathbf{B}$ both avoid $S_{\infty}$, then we can apply Corollary \ref{>r} to conclude they both have area at least $r$. As limiting buildings have the same total area as curves in the Gromov foliation, namely $2r$, this gives a contradiction since the top level cylinder of $\mathbf{B}$ must also have positive area. Hence, we conclude that exactly one of the planes of $\mathbf{B}$ intersects $S_{\infty}$, and so by monotonicity has area greater than $2r-1$. By Corollary \ref{>r}, the other plane has area at least $r$, and fits together with the cylinder to give a Maslov $2$ disk in $\RR^4$ with boundary on either $L_1$ or $L_2$. This class has area strictly greater than $r$, which implies that it has area at least $s$. But then the area of the building $\mathbf{B}$ exceeds $s + (2r-1) \ge 2r$, which is again a contradiction. This completes the proof of Lemma \ref{lem:noB}.



\subsubsection{The features of a foliation of $X$ relative to $L_1 \cup L_2$}\label{features}
Assume that $J$  is of foliation type relative to $L_1$ and $L_2$. Discarding the middle and bottom  level curves of the Type A buildings in $\bar{\mathcal{F}}_{L_1,L_2}(J)$, one obtains a foliation 
$\mathcal{F}_{L_1,L_2}(J)$ of $X\smallsetminus (L_1 \cup L_2)$ with the following properties.

\bigskip

\begin{itemize}
\item[(F1)] The foliation $\mathcal{F}_{L_1,L_2}(J)$ has three kinds of leaves: unbroken ones consisting of a single closed $J_{\infty}$-holomorphic sphere in $X\smallsetminus (L_1 \cup L_2)$ of class $[S^2 \times \{pt\}]$, broken leaves along $L_1$ that consist of a pair of finite energy $J_{\infty}$-holomorphic planes in $X\smallsetminus (L_1 \cup L_2)$ that are asymptotic to $L_1$, and broken leaves along $L_2$ that consist of a pair of finite energy $J_{\infty}$-holomorphic planes in $X\smallsetminus (L_1 \cup L_2)$ that are asymptotic to $L_2$. 
  
\item[(F2)] By positivity of intersection, each leaf of $\mathcal{F}_{L_1,L_2}(J)$ intersects $S_{\infty}$ in exactly one point. For a broken leaf this means that exactly one of its planes intersects $S_{\infty}$. We denote the two families of broken planes asymptotic to $L_1$ by $\mathfrak{r}_0$ and $\mathfrak{r}_{\infty}$, and the two families of planes asymptotic to $L_2$ by  $\mathfrak{s}_0$ and $\mathfrak{s}_{\infty}$. We label these so that the planes in $\mathfrak{r}_{\infty}$ and $\mathfrak{s}_{\infty}$ intersect $S_{\infty}$. The planes in $\mathfrak{r}_0$ and $\mathfrak{s}_0$ are disjoint from $S_{\infty}$.
  
\item[(F3)] The ends of the two planes of a broken leaf are asymptotic to the same embedded geodesic in $L_i$, but with opposite orientations. We denote the homology class of this geodesic, equipped with the orientation determined by the plane that intersects $S_{\infty}$, by $\beta_i\in H_1(L_i; \ZZ)$. The class, $\beta_i$, is the same for all broken leaves of $\mathcal{F}_{L_1,L_2}(J)$ along $L_i$. It will be  referred to as the foliation class of $\mathcal{F}_{L_1,L_2}(J)$ for $L_i$.

\item[(F4)] Each point in $X \smallsetminus (L_1 \cup L_2)$ lies in a unique leaf of $\mathcal{F}_{L_1,L_2}(J)$, and each point of $L_i$ lies on a unique geodesic in the foliation class $\beta_i$. Hence, there is a well-defined map $p \colon X \to S_{\infty}$ which takes $z \in X \smallsetminus (L_1 \cup L_2)$ to the unique intersection of its leaf with $S_\infty$, and takes $z \in L_i$ to the intersection with $S_{\infty}$ of the broken leaf asymptotic to the unique geodesic through $z$ representing the foliation class $\beta_i$. 
The images $p(L_1)$ and $p(L_2)$ are closed, embedded, disjoint curves in $S_{\infty}$.

\end{itemize}

This picture can also be refined by {\it straightening} the leaves of $\mathcal{F}_{L_1,L_2}(J)$ near $L_1 \cup L_2$ following 
Proposition 5.16 of \cite{rgi} and Lemma 3.4 of \cite{hiker}.

\begin{lemma} Perturbing $J$ away from  $S_\infty \cup T_\infty \cup\,\mathcal{U}_1 \cup \,\mathcal{U}_2$ we may assume that the unbroken leaves of 
 $\mathcal{F}_{L_1,L_2}(J)$ that intersect $ \mathcal{U}_1 \cup \mathcal{U}_2$ do so along the annuli $$\{p_1=\delta, q_1= \theta, -\epsilon <p_2<\epsilon\}$$ for some $\theta \in S^1$ and $\delta \in (-\epsilon, \epsilon).$  As well, the planes of the broken leaves of $\mathcal{F}_{L_1,L_2}(J)$ through  $S_\infty$ intersect each $\mathcal{U}_i$ along the annulus $$\{p_1 = 0, q_1 =
\theta , 0 < p_2 < \epsilon \},$$ for some $ \theta \in S^1,$  and the planes of the broken leaves of $\mathcal{F}_{L_1,L_2}(J)$ through  $S_0$ intersect each $\mathcal{U}_i$ along the annulus $$\{p_1 = 0, q_1 =
\theta , -\epsilon < p_2 < 0 \},$$ for some $ \theta \in S^1.$
\end{lemma}

\begin{remark}\label{rem:reg}
 Using the regularity assumption on $J_\infty$ we can also assert that each sphere and plane that appears in a leaf of $\mathcal{F}_{L_1,L_2}(J)$ is the image of regular somewhere injective $J_{\infty}$-holomorphic curve.
\end{remark}

\begin{remark}\label{rem:flex}
The assumption that both $L_1$ and $L_2$ are symplectomorphic to $L(r,s)$ is not necessary here. The argument, in particular the use of Corollary \ref{>r}, goes through without change if $L_1$ is replaced by  $L(r, s-1+t)$ for any $t \in [1, b-s+1).$ This fact is crucial to the proof of Proposition \ref{prop:proj} below.
\end{remark}

\subsection{Completion of the proof of Proposition \ref{prop:proj}}
Choose an $\Omega$-tame almost complex structure $J$ on $X$ that is standard in some open neighborhood $U_\infty$ of $S_{\infty} \cup T_{\infty}$, is compatible with the parameterizations of $L_1$ and $L_2$, and whose limit, $J_{\infty}$, is regular for somewhere injective curves. By Lemma \ref{lem:noB}, we have the foliation $\mathcal{F}_{L_1,L_2}(J)$ of $X$ relative to $L_1 \cup L_2$. The properties of its projection map $p \colon X \to  S_{\infty}$, as described above in (F4), imply all but the final assertion of Proposition \ref{prop:proj}.  It remains to verify that the component of $S_\infty \smallsetminus p(L_1)$ containing $p(T_\infty)$,  is contained in the component of $S_\infty \smallsetminus p(L_2)$ containing $p(T_\infty)$. To prove this
we will use foliations of $X$ relative to $L[t]$ and $L[t] \cup L_2$  where $$L[t]= L(r, s-1+t)$$ and $t \in [1, b -s +1).$
The crucial fact to be exploited is that, by Assumption \ref{assump}, $L_2$ is disjoint from all of the $L[t]$.

\bigskip

\subsubsection{Stretching in a family.}
Let $\epsilon>0$ be small enough so that for all $t \ge b-s+1 -\epsilon$ the torus $L[t]$ lies in the neighborhood  $U_\infty$ where $J$ is standard. Setting $c-b-s+1-\epsilon$, let $\{J_t\}$ for $t \in [1,c]$ be a family of $\Omega$-compatible almost complex structures such that 
\begin{itemize}
    \item $J_1=J$,\\
    \item each $J_t$ is equal to $J$ in $U_\infty.$\\
    \item each $J_t$ is compatible with the parameterizations of $L[t]$ and $L_2$.
\end{itemize}

\medskip
In what follows, we will either stretch the $J_t$ along $L[t] \cup L_2$ to get a limiting almost complex structure $J_{t, \infty}$, or we will stretch the $J_t$ along just $L[t]$ to get a limiting almost complex structure $J'_{t, \infty}$. 

\begin{definition}\label{regulart}
 We say that $t \in [1,c]$ is regular if $J_{t, \infty}$ and $J'_{t, \infty}$  are both regular for somewhere injective curves.
\end{definition}

If $t$ is regular then we get a foliation of $X$ relative to $L[t] \cup L_2$, which we denote by $\mathcal{F}(J_t)$, and 
a foliation of $X$ relative to $L[t]$ which we denote by $\mathcal{F}'(J_t)$. We denote the corresponding projection maps by $p_t \colon X \to S_\infty$ and $p'_t \colon X \to S_\infty$, respectively.

\medskip

We may assume that $t=1$ and $t=c$ are regular and that the set of  $t \in (1,c)$ which are not regular is finite.
We now show that on intervals of regular $t$ the families of foliations $\mathcal{F}(J_t)$ and $\mathcal{F}'(J_t)$ vary smoothly.


\medskip

\begin{lemma}\label{lem:fam}
Suppose that $t \in (1,c)$ is regular. For sufficiently small $\delta>0$,  the points $\tau \in (t - \delta, \, t+\delta)$ are all regular and the corresponding foliations $\mathcal{F}(J_\tau)$ vary smoothly over this interval. 
\end{lemma}

\begin{proof}
We use Fukaya's trick for moving regular curves with Lagrangian boundary as the boundary moves. Given $\delta > 0$, consider a family of diffeomorphisms $f_{\tau}$ of $X$, for $\tau \in [t - \eta, \, t+\eta]$, 
such that 
\medskip

\begin{itemize}
    \item $f_{\tau}(L[\tau]) =L[t]$ \\
    \item $f_{\tau} \circ \psi_{\tau} = \psi_{t}$ \\
    \item $f_{\tau}$ agrees with the identity map in a neighborhood of $S_{\infty}\cup T_{\infty} \cup L_2$ \\
    \item $\|f_{\tau}\|_{C^1}$ is of order $1$ in $\tau -t$.\\
\end{itemize}
Consider the family of almost complex structures on $X$ given by $$\widetilde{J}_{\tau} = (f_{\tau})_*J_{\tau}.$$ By the second and third conditions these are compatible with the parameterizations $\psi_{t}$ on $L[t]$ and $\psi_2$ on $L_2$, and their stretched limits $$\widetilde{J}_{\tau, \infty} = (f_{\tau})_*J_{\tau,\infty} $$  on $X \smallsetminus (L[t] \cup L_2)$ are also adapted to these parameterizations.\footnote{Each $f_{\tau}$ can also be viewed as a diffeomorphism from $X \smallsetminus (L[\tau] \cup L_2)$ to $ X \smallsetminus (L[t] \cup L_2)$.}

By automatic regularity for holomorphic spheres and holomorphic planes, see \cite[Theorem 1]{we}, for sufficiently small $\delta>0$ our $J_{t, \infty}$-holomorphic foliation deforms to give $\widetilde{J}_{\tau, \infty}$-holomorphic relative foliations $\mathcal{F}(\widetilde{J}_\tau)$ of $X$ relative to $L[t] \cup L_2$ for $\tau \in (t-\delta, t+\delta)$. Moreover, (the curves of) these foliations vary smoothly with $\tau$. The map $f^{-1}_{\tau}$ then takes the relative foliations $\mathcal{F}(\widetilde{J}_\tau)$ to the relative foliation $\mathcal{F}(J_\tau)$ of $X$ relative the $L[\tau] \cup L_2$.
\end{proof}

An identical argument yields the following. 

\begin{lemma}\label{lem:fam'}
Suppose that $t \in (1,c)$ is regular. For sufficiently small $\delta>0$,  the points $\tau \in (t - \delta, \, t+\delta)$ are all regular and the corresponding foliations $\mathcal{F}'(J_\tau)$ vary smoothly over this interval. 
\end{lemma}

To manage the behavior of these relative foliations at the  nonregular points in $[1,c]$ it is helpful to shift from the language of foliations to the language of moduli spaces of pseudo-holomorphic curves. Abusing notation, we will denote by $\mathfrak{r}_{0,t}$ the moduli space of $J_{t, \infty}$-holomorphic planes asymptotic to $L[t]$ which are disjoint from $S_{\infty}$, have Maslov class $2$ and area $r$. If $t$ is regular,  then $\mathfrak{r}_{0,t}$ is an $S^1$-family of classes each represented by a plane asymptotic to a geodesic representing negative the foliation class of $L[t]$. Moreover, each class in $\mathfrak{r}_{0,t}$ has a representative plane that is one half of a broken leaf of the foliation $\mathcal{F}(J_t)$. 
%
The moduli space 
$\mathfrak{r}_{\infty,t}$ is defined analogously, as are the moduli spaces $\mathfrak{s}_{0,t}$ and $\mathfrak{s}_{\infty,t}$ of planes asymptotic to $L_2$. If we stretch just along $L[t]$ then we have moduli spaces $\mathfrak{r}'_{0,t}$ and $\mathfrak{r}'_{\infty,t}$. Finally, we can define family moduli spaces $\tilde {\mathfrak{r}}_0$, $\tilde {\mathfrak{r}}_{\infty}$, $\tilde {\mathfrak{s}}_0$, $\tilde {\mathfrak{s}}_{\infty}$, $\tilde {\mathfrak{r}}'_0$, $\tilde {\mathfrak{r}}'_{\infty}$, where
$$\tilde {\mathfrak{r}}_0= \{(u,t) \mid u \in \mathfrak{r}_{0,t},\, t\in [1,c]\},$$
and the other spaces are defined analogously.

\bigskip

Lemma \ref{>r2} can be used to show that three of these family moduli spaces are compact.

\begin{lemma}\label{lem:fam2} Let $t_k$ be a regular sequence converging to a $t \in [1,c]$ that is not regular.

\begin{enumerate}

\item 
Suppose that $u_k$ is a sequence of $J'_{t_k, \infty}$-holomorphic planes in $\mathfrak{r}'_{0,t_k}$ ($\mathfrak{r}'_{\infty,t_k}$). Any subsequence of $u_k$ which converges to a holomorphic building must in fact converge to a single  $J'_{t, \infty}$-holomorphic plane in $X \smallsetminus L[t]$. Hence, the family moduli spaces $\tilde {\mathfrak{r}}'_0$ and $\tilde {\mathfrak{r}}'_{\infty}$ are compact.  

\item 
Suppose that $v_k$ is a sequence of $J_{t_k, \infty}$-holomorphic planes in $\mathfrak{s}_{0,t_k}$. Any subsequence of $v_k$ which converges to a holomorphic building must in fact converge to a single  $J_{t, \infty}$-holomorphic plane in $X \smallsetminus L_2$.
Hence, the family moduli space $\tilde {\mathfrak{s}}_0$ is compact.
Moreover, for each $t$, there exists a unique plane in $\mathfrak{s}_{0,t}$ asymptotic to each geodesic in $L[t]$ in the negative  of the foliation class.
\end{enumerate}
\end{lemma}

\begin{proof}

It follows from the compactification results of \cite{behwz} that to prove the first assertion  it suffices to show that the only possible top level curve in the limiting building is a plane of $\Omega$-area r.  With this, Lemma \ref{>r2} implies that it suffices to show that for some $\Omega$-compatible almost complex structure $J_0$ compatible with $L[t] =L(r, s-1+t)$, the minimal $\Omega$-area of $J_0$-holomorphic curves $u$ with $\mathrm{index}(u) \ge \mathrm{end}(u)$ is $r$. This is easily checked when $J_0$ is the standard product complex structure.

\medskip

To prove (2), we first note that since $L_2$ is Hamiltonian isotopic to the product torus $L(r,s)$ in $\RR^4$, there exists an almost complex structure on $\RR^4$ for which the minimal area of holomorphic curves is $r$. Then, since $X \setminus (S_{\infty} \cup T_{\infty})$ can be identified with a subset of $\RR^4$, we see from Lemma \ref{>r2} that curves $v_k$ cannot degenerate to nontrivial buildings that are broken only along $L_2$.

Suppose then that a limiting building, $\mathbf{B}$, of the curves $v_k$ is broken along $L[t]$. In other words $\mathbf{B}$ contains curves in $T^* L[t]$. Since the $t_k$ are regular, each $J_{t_k}$ is of foliation type and the curves in $\mathfrak{s}_{0,t_k}$ belong to the limiting foliation $\mathcal{F}(J_{t_k})$. In particular, they are disjoint from the planes in $\mathfrak{r}_{0,t_k}$ and $\mathfrak{r}_{\infty,t_k}$. It then follows, from positivity of intersection, that the curves of $\mathbf{B}$ mapping to $T^* L[t]$ must cover cylinders in the foliation class of $L[t]$. Given such a curve, the building $\mathbf{B}$ can be divided into two nontrivial connected sub-buildings, one of which has the property that all its unmatched ends cover geodesics in multiples of the foliation class in $L[t]$. This sub-building will have area at least $r$. Since the other sub-building is notrivial it will have positive area. As the total area of $\mathbf{B}$ is $r$ we have a contradiction.
\end{proof}

Now we can complete the proof of Proposition \ref{prop:proj}. Arguing by contradiction, we assume that 
\begin{equation}\label{assumpa}
p(L_2) \subset C \smallsetminus p(T_\infty).    
\end{equation}
We note that $J'_{t,\infty}$ agrees with $J_{t,\infty}$ away from a neighborhood of $L_2$. Therefore any curves of $\mathcal{F}(J_t)$ which avoid a neighborhood of $L_2$, and in particular are unbroken along $L_2$, will also form leaves of the foliation $\mathcal{F}'(J_t)$.


Our choice of $c$ and the condition that $J$ is standard in $U_\infty$ implies that  $p'_t(T_\infty) = p_t(T_\infty) = p_1(T_\infty)$ for all $t \in [1,c],$ and that $p'_c(L[c])$ is the boundary of a small disk $D_c$ in $S_\infty$ containing $p_1(T_\infty)$. Choosing  $c$ sufficiently close to $b-s+1$ we can refine assumption \eqref{assumpa} to 
\begin{equation}\label{assumpb}
p_1(L_2) \subset C \smallsetminus D_c.    
\end{equation}

Then, as the broken leaves over $\partial C$ are disjoint from $L_2$, we also have
\begin{equation}\label{bound}
    p_1'(L[1])= p_1(L[1]) = \partial C
\end{equation} and hence 
\begin{equation}\label{bound1}
    (p'_1)^{-1}(C)= (p_1)^{-1}(C).
\end{equation}
Similarly, as our almost complex structures remain standard on $U_{\infty}$ we have  
\begin{equation}\label{bound2}
    (p'_t)^{-1}(D_c)= (p_t)^{-1}(D_c),
\end{equation}
for all $t$.

Formulas \eqref{bound1} and \eqref{bound2}, together with Assumption \eqref{assumpb}, imply that
\begin{equation}\label{assumpc}
p'_1(L_2) \subset C; \, \, p'_c(L_2) \subset S_{\infty} \smallsetminus D_c.    
\end{equation}
In fact we also have
\begin{equation}\label{assumpc2}
p'_1(\mathfrak{s}_{0,1}) \subset C; \, \, p'_c(\mathfrak{s}_{0,c}) \subset S_{\infty} \smallsetminus D_c.    
\end{equation}
as the broken $J_{1,\infty}$ or $J_{c,\infty}$ holomorphic curves asymptotic to $L_2$ are disjoint from those asymptotic to $L[1]$ and $L[c]$ respectively (in the first case by the assumption that $J_1$ is regular, and in the second case as $J_t=J$ is standard on $U_{\infty}$).

By Lemmas \ref{lem:fam} and \ref{lem:fam2} the moduli spaces $\mathfrak{s}_{0,t}$ vary continuously with $t$. Using a fixed geodesic on $L_2$ in the foliation class, we can identify a continuous family of planes $\{w_t\}_{t \in [1,c]}$ with $[w_t] \in  \mathfrak{s}_{0,t}$. Let $\{ x_t \}_{t \in [1,c]}$ be a continuous family of points such that $x_t$ is in the image of $ w_t$. By \eqref{assumpc2} we have $$x_1 \in (p'_1)^{-1}(C), \, \, x_c \in (p'_c)^{-1}(S_{\infty} \smallsetminus D_c).$$

Lemmas \ref{lem:fam'} and \ref{lem:fam2} imply that the moduli spaces $\mathfrak{r}'_{0,t}$ and $\mathfrak{r}'_{\infty,t}$ also vary continuously with $t$. Taking paths of representative planes for each and invoking  \eqref{bound} and \eqref{bound2}, we obtain a homotopy between $(p_1')^{-1}(\partial C)$ and $(p_c')^{-1}(\partial D_{c})$. By \eqref{bound2}, the homotopy avoids $(p_1)^{-1}(D_c)$, and so gives a surjection onto $(p'_1)^{-1}(C \smallsetminus D_c)$. Therefore, for the path $\{ x_t \}_{t \in [1,c]}$ there is an $s \in [1,c]$, and a $J'_{s, \infty}$-holomorphic plane representing either $\mathfrak{r}'_{0, s}$ or $\mathfrak{r}'_{\infty, s}$, that intersects $x_s$.

Let $J^N_{t, \infty}$ be the result of stretching $J'_{t, \infty}$ to length $N$ along $L_2$ so that $J^N_{t, \infty} \to J_{t,\infty}$ as $N \to \infty$. Arguing as above, for a fixed $N$, we can find an $s^N \in [1,c]$, and a $J^N_{s^N, \infty}$-holomorphic plane representing  $\mathfrak{r}'_{0, s^N}$ or $\mathfrak{r}'_{\infty, s^N}$ whose image intersects $x_{s^N}$. Passing to a subsequence, we may assume that the $s^N$ converge to $s^{\infty} \in [1,c]$ and that the limits of the $J^N_{s^N, \infty}$-holomorphic planes representing  $\mathfrak{r}'_{0, s^N}$ and $\mathfrak{r}'_{\infty, s^N}$ both converge to $J_{s^{\infty}, \infty}$-holomorphic buildings. One of these limiting buildings, $\bf{B}$, must intersect $x_{s^{\infty}}$ and thus the $J_{s^{\infty}, \infty}$-holomorphic plane $w_{s^{\infty}}$ that contains $x_{s^{\infty}}$ and represents a class in $\mathfrak{s}_{0,s^\infty}$. Either the building $\bf{B}$ includes a curve that covers the plane $w_{s^{\infty}}$, or it does not.

If $\bf{B}$ does include such a curve, this curve will have area at least $r$. However, the building $\bf{B}$ has a total area $r$ and must also include a curve asymptotically to $L[{s^{\infty}}]$ of positive area. This gives a contradiction.

If $\bf{B}$ does not include a curve that covers the plane $w_{s^{\infty}}$, then, by positivity of intersection, it includes a curve that intersects $w_{s^{\infty}}$ positively. However, this would imply that, for large $N$, the planes $w_{s^N}$ representing $\mathfrak{s}_{0,s^N}$  must intersect the planes in either $\mathfrak{r}_{0, s^N}$ or $\mathfrak{r}_{\infty, s^N}$. This is a contradiction whenever the $J^N_{s^N}$ are of foliation type.

\section{The proof of Proposition \ref{prop:submanifolds}}\label{dbuildings}

The symplectic disks $D^1_{\infty}$ and $D^2_{\infty}$ in the statement of Proposition \ref{prop:submanifolds} are obtained by compactifying a plane in $\mathfrak{r}_\infty$ and a plane in $\mathfrak{s}_\infty$, respectively. The remaining submanifolds in the statement of Proposition \ref{prop:submanifolds}  will be obtained using $(d,n_1)$-buildings.  These are defined and analyzed below in the next section. Two existence results for $(d,n_1)$-buildings are established in Section \ref{sub:exist}. The process of obtaining  the desired symplectic submanifolds from these buildings is described in Section \ref{sub:shift}

\subsection{On $(d,n_1)$-buildings and their types}\label{sub:d}
Fix a nonnegative integer $d$ and a generic collection of $2d+1$ points on $L_1 \cup L_2$. Let $n_1 \le 2d+1$ be the number of these points on $L_1$. For a generic choice of tame almost complex structure $J$ on $(X,\Omega)$ (compatible with parameterizations of $L_1$ and $L_2$ as before) there exists a smooth $J$-holomorphic sphere $u\colon S^2 \to X$ that represents the class $(d,1)$ and passes through the $2d+1$ points on $L_1 \cup L_2$. This curve is unique up to reparameterization. 
As one  stretches along  $L_1 \cup L_2$ to obtain the foliation $\mathcal{F}_{L_1,L_2}(J)$ one gets a sequence of pseudo-holomorphic spheres in the class $(d,1)$. Passing to a convergent subsequence, one obtains a holomorphic building, $\mathbf{F}$, which we will refer to as a $(d,n_1)$-building. As before, this building consists of genus zero holomorphic curves in three levels.
Each top level curve of $\mathbf{F}$ can be compactified to yield a map from a surface of genus zero with boundary to $(X,L_1 \cup L_2)$. The components of the boundary correspond to the negative punctures of the curve.  They are mapped to the closed geodesics on $L_1$  or $L_2$ underlying the Reeb orbits to which the corresponding puncture is asymptotic. The middle and bottom level curves of $\mathbf{F}$ can be compactified to yield maps to either $L_1$ or $L_2$ with the same type of boundary conditions. These compactified maps can all be glued
together to form a map $\bar{\mathbf{F}}\colon S^2  \to X$ in the class $(d,1)$. 

\medskip

The curves of a $(d,n_1)$-building are constrained by its companion foliation $\mathcal{F}_{L_1,L_2}(J)$. To describe these constraints, the following notions will be useful.


\begin{definition}\label{sp}
Given an almost complex structure $J$ on $(X,\Omega)$,
a $J_\infty$-holomorphic curve $u$ in $X \smallsetminus (L_1 \cup L_2)$ is said to be {\em essential} if the map $p \circ u$ is injective. Here $p$ is the projection map defined using the foliation $\mathcal{F}_{L_1,L_2}(J)$.
\end{definition}

\begin{definition}\label{ft}
 A puncture of a $J_\infty$-holomorphic curve $u$ in $X \smallsetminus (L_1 \cup L_2)$ is said to be of {\em foliation type with respect to} $L_i$  if it is asymptotic to a closed Reeb orbit which lies on the copy of $S^*\mathbb{T}^2$ that corresponds to $L_i$ and covers a closed geodesic in an integer multiple of the foliation class $\beta_i$. The puncture is of {\em positive (negative) foliation type} if this integer is positive (negative).
\end{definition}




\begin{proposition}\label{prop:type} Each  $(d,n_1)$-building  $\mathbf{ F}$ 
is of one of the following types:




\bigskip

\noindent {\bf Type 1.} $\mathbf{ F}$ has a unique essential curve, $\underline{u}$, which is a $J_{\infty}$-holomorphic sphere with at least one puncture. The image of the injective map $p\circ \underline{u}$ is $S_{\infty}$ minus finitely many points on $p(L_1) \cup p(L_2)$. Any punctures of $\underline{u}$ asymptotic to $L_1$ are of foliation type and are either all positive or all negative. The same holds for any punctures of $\underline{u}$ along $L_2$. 



\bigskip

\noindent {\bf Type 2.1.} $\mathbf{ F}$ has exactly two essential  curves, $u_2$ and $\underline{u}$. The closures of the images of the maps $p \circ u_2$ and $p \circ \underline{u}$  are $A$ and $B \cup C$, respectively. Any punctures of $\underline{u}$ on $L_1$ are all of foliation type and are either all positive or all negative. 




\bigskip

\noindent {\bf Type 2.2.} $\mathbf{ F}$ has exactly two essential curves, $\underline{u}$ and $u_1$. The closures of the images of the maps $p \circ \underline{u}$ and $p \circ u_1$  are $A \cup B$ and $C$, respectively.  Any punctures of $\underline{u}$ on $L_2$ are all of foliation type and are either all positive or all negative. 



\bigskip

\noindent {\bf Type 3.} $\mathbf{ F}$ has exactly three essential curves, $u_1$, $\underline{u}$,  and $u_2$. The closures of the images of the maps  $p \circ u_1$, $p \circ \underline{u}$,  and $p \circ u_2$ are  $C$, $B$ and $A$, respectively. 


\end{proposition}

\begin{proof} This is a refinement of Proposition 3.17 of \cite{hiker}. In that work, the constraint points were not restricted to lie on the relevant Lagrangian tori. This allows for an additional kind of $(d,n_1)$-buildings, Type 0, which is characterized by having a unique essential curve which is a $J_\infty$-holomorphic sphere, without punctures, that maps to $X \smallsetminus (L_1 \cup L_2)$ and represents the class $(d-k,1) \in H^2(X;\ZZ)$ for some $0 \le k \le d-1$.  To prove Proposition \ref{prop:type} it  suffices to show that if $\mathbf{F}$ is a  $(d,n_1)$-building corresponding to $2d+1$ constraint points on $L_1 \cup L_2$, then it cannot be of Type 0. Assume that $\mathbf{F}$ is of Type 0. Then its nonessential curves can be sorted into: a collection of $J_\infty$-holomorphic spheres that cover unbroken leaves of the foliation,
and a collection of sub-buildings that cover broken leaves of the foliation. Together, these nonessential curves compactify to form a disconnected curve in the class $(k,0)\in H^2(X;\ZZ)$. In particular, the number of sub-buildings covering broken leaves is at most $k \le d$. On the other hand, the middle and bottom level curves of every one of these sub-building cover a unique cylinder, in the foliation class, in one of the two copies of $T^*\mathbb{T}^2$ associated to $L_1$ or $L_2$. Hence, each of the $k \le d$ sub-buildings can pick up, at most, one of the generic point constraints. Since there are $2d+1$ of these constraints  and no other way for $\mathbf{F}$ to hit them, we have arrived at a contradiction and $\mathbf{F}$ cannot be of Type 0.

\end{proof}

\subsubsection{Managing the nonessential curves of $(d,n_1)$-buildings}\label{common}

Here we describe a scheme for managing the nonessential curves of $(d,n_1)$-buildings of Types $1$, $2.1$ and $2.2$. In what follows we will denote the collection of essential curves of a $(d,n_1)$-building $\mathbf{F}$ by $\mathbf{u}$.

\bigskip

\noindent {\bf Type $1$.}
Let $\mathbf{F}$ be a $(d,n_1)$-building of Type $1$. Then $\mathbf{u} = \{\underline{u}\}$, where $\underline{u}$ is the essential curve of $\mathbf{F}$ described in Proposition \ref{prop:type}. The image of $p \circ \mathbf{u}$ is equal to $S_{\infty}$ minus a finite collection of points on $p(L_1) \cup p(L_2)$.
The following dichotomy will be useful: the image of every nonessential curve of $\mathbf{F}$ either projects to a point in the image of $p\circ \mathbf{u}$, or it maps to one of the missing points.

The curves whose projections lie in the image of $p\circ \mathbf{u}$, are either $J_\infty$-holomorphic spheres that cover unbroken leaves of the foliation $\mathcal{F}_{L_1,L_2}(J)$, or they are $J_\infty$-holomorphic planes that form a collection of sub-buildings  of $\mathbf{F}$ that each cover a broken leaf of the foliation. Let $N$ be the total number of times these spheres and sub-buildings cover the class $(1,0)$. Then $N=N_0+N_1 + N_2$ where $N_0$ is the number of times an unbroken leaf is covered, and $N_i$ is the number of times a broken leaf on $L_i$ is covered.

The remaining curves of $\mathbf{F}$, whose images don't project to the image of $p\circ \mathbf{u}$, each belong to a sub-building defined by one of the punctures of $\mathbf{u}$. Label these punctures as  $$\{x_1, \dots, x_{k_1}, y_1, \dots, y_{k_2}\}.$$ Here, each $x_i$ is asymptotic to a closed geodesic on $L_1$ of the form $\gamma_i^{\ell_i}$, where $\gamma_i$ represents the foliation class $\beta_1$ and all the $\ell_i$ have the same sign. Similarly, each $y_j$ is asymptotic to a closed geodesic on $L_2$ of the form $\sigma_j^{m_j}$ where $[\sigma_j]=\beta_2$ and all the $m_j$ have the same sign. In these terms, the image of $p\circ \mathbf{u}$ is equal to $S_\infty$ minus the set  $$\{p(\gamma_1), \dots, p(\gamma_{k_1})\}\cup\{ p(\sigma_1), \dots, p(\sigma_{k_2})\} \subset p(L_1) \cup p(L_2).$$ 

The relevant features of the sub-building corresponding to the puncture $x_i$ are as follows. It compactifies to a disk and has a single unmatched end on $\gamma_1^{\ell_i}$. Its top level curves all cover one of the planes in a single matching pair of planes in $\mathfrak{r}_0 \cup \mathfrak{r}_\infty$. Let $M^0_1(i)$ be the degree of the map from the sub-building to the plane in $\mathfrak{r}_0$ and let $M^\infty_1(i)$  be the degree of the map from the sub-building to the plane in $\mathfrak{r}_\infty.$ We then have 
\begin{align}
 \ell_i+M^{\infty}_1(i)-M^0_1(i)=0.   
\end{align}
The description of the sub-buildings corresponding to the $y_j$ is completely analogous, and yields nonnegative numbers $M^0_2(j)$ and $M^\infty_2(j)$ satisfying
\begin{align}
 m_j+M^{\infty}_2(j)-M^0_2(j)=0.   
\end{align} 
Set  $$M_1^0 = \sum_{i=1}^{k_1} M^0_1(i),$$ and define the numbers $M_1^\infty$, $M_2^0$, and $M_2^\infty$ analogously. 

With these labels in place we can derive some equations and inequalities  which will be useful later on. First we note that if the punctures of $\mathbf{u}$ along $L_1$ are all positive then each $\ell_i$ is at least one, and we have 
\begin{align}\label{pos1}
M_1^0 = M_1^\infty + \sum_{i=1}^{k_1} \ell_i \geq M_1^\infty + k_1.
\end{align}
If these punctures are all negative we have 
\begin{align}\label{neg1}
M_1^\infty = M_1^0 - \sum_{i=1}^{k_1} \ell_i \geq M_1^0 + k_1.
\end{align}
The fact that there are  $n_1$ point constraints on $L_1$ implies that 
\begin{align}\label{constraint1}
k_1 + N_1 \geq n_1. 
\end{align}
Computing the $\Omega$-area of the curves of  $\mathbf{F}$ we get
\begin{align}\label{area1}
\Omega (\mathbf{u})+ r(M_1^0 +M_1^\infty + M_2^0 + M_2^\infty)+2rN = 2rd+b.  
\end{align}
Finally, computing the intersection number of the curves of $\mathbf{F}$ with $S_{\infty}$ we get the expression
\begin{align}\label{S1}
\mathbf{u} \cdot S_{\infty}+ M_1^\infty + M_2^\infty + N  =  d.
\end{align}

\noindent {\bf Type $2.1$.}
Let $\mathbf{F}$ be a $(d,n_1)$-building of Type $2.1$. In this case  $\mathbf{u} = \{u_2,\underline{u}\}$. The numbers $N$, $N_0$, $N_1$ and $N_2$ can be defined as above. Labeling the punctures of $\underline{u}$  on $L_1$ as  $\{x_1, \dots, x_{k_1}\}$, each $x_i$ is again asymptotic to a closed geodesic on $L_1$ of the form $\gamma_i^{\ell_i}$, where $[\gamma_i] =\beta_1$ and all the $\ell_i$ have the same sign. For each puncture $x_i$ there is a corresponding sub-building of $\mathbf{F}$ that compactifies to a disk whose boundary maps to $\gamma_1^{\ell_i}$. The top level curves of this subuilding yield degrees $M^0_1(i)$ and $M^\infty_1(i)$ as before,  satisfying
\begin{align}\label{eq:2.1match}
 \ell_i+M^{\infty}_1(i)-M^0_1(i)=0,   
\end{align}
and hence inequality  \eqref{pos1} or \eqref{neg1}. To accommodate the point constraints on $L_1$ we again have 
\begin{align}\label{constraint2}
k_1 + N_1 \geq n_1. 
\end{align}

The remaining top level components of $\mathbf{F}$ cover planes in $\mathfrak{s}_0$ and $\mathfrak{s}_\infty$. Let $\widetilde{M}^0_2$ be the total degree of these curves that  cover planes in $\mathfrak{s}_0$ and let  $\widetilde{M}^\infty_2$ be the total degree of these curves covering planes in $\mathfrak{s}_\infty$. The fact that there are $n_2 = 2d+1-n_1$ point constraints on $L_2$ implies that 
\begin{align}\label{constraints2.1}
 \widetilde{M}^0_2+\widetilde{M}^\infty_2   \geq n_2 - 1.
\end{align}
This follows from an index argument. Indeed, the deformation index of curves in $T^*\mathbb{T}^2$ is $2s-2$, where $s$ is the number of positive ends, so we must have $2s-2 \geq 2n_2$. Since two ends can match the essential curves, the rest must match planes in $\mathfrak{s}_0$ or $\mathfrak{s}_\infty$, and the inequality follows.

In these terms, the computation of the $\Omega$-area of the curves of  $\mathbf{F}$ yields
\begin{align}\label{area2}
\Omega (\mathbf{u})+ r(M_1^0 +M_1^\infty + \widetilde{M}_2^0 + \widetilde{M}_2^\infty)+2rN = 2rd+b,  
\end{align}
and the intersection number of the curves of $\mathbf{F}$ with $S_{\infty}$ yields the expression
\begin{align}\label{S2}
\mathbf{u} \cdot S_{\infty}+ M_1^\infty + \widetilde{M}_2^\infty + N  =  d.
\end{align}

\noindent {\bf Type $2.2$.}
Let $\mathbf{F}$ be a $(d,n_1)$-building of Type $2.2$. In this case  $\mathbf{u} = \{\underline{u}, u_1\}$ and the discussion is completely analagous to that for buildings with Type 2.1, with the role of the puncture $x_i$ on $L_1$ replaced by the role of the puncture $y_j$ on $L_2$  and all the other labels swapped. In particular, we have
\begin{align}
 m_j+M^{\infty}_2(j)-M^0_2(j)=0,   
\end{align}
\begin{align}\label{constraint2.2}
k_2 + N_2 \geq n_2 ,
\end{align}
\begin{align}\label{area2.2}
\Omega (\mathbf{u})+ r(\widetilde{M}_1^0 +\widetilde{M}_1^\infty + M_2^0 + M_2^\infty)+2rN = 2rd+b,  
\end{align}
and 
\begin{align}\label{S2.2}
\mathbf{u} \cdot S_{\infty}+ \widetilde{M}_1^\infty + M_2^\infty + N  =  d.
\end{align}


\subsection{On $(d,n_1)$-buildings of high degree}\label{sub:exist}

In this section we establish two classification results for $(d,n_1)$-buildings for sufficiently large $d$. These are used in the next section to prove the existence of the remaining symplectic submanifolds from Proposition \ref{prop:submanifolds}. 

The following monotonicity lemma, established by Boustany in \cite{boust}, will be used several times to leverage the assumption that $d$ is large. 
\begin{lemma}\label{lem:monotonicity}(\cite{boust}, Lemma 5.6)
Let $S \subset  X \smallsetminus (L_1 \cup L_2)$ be a closed $J_{\infty}$-holomorphic sphere. There exists an $\epsilon = \epsilon(S)>0$ such that for every (possibly) punctured $J_{\infty}$- holomorphic sphere $u$ in $X \smallsetminus (L_1 \cup L_2)$ we have
\begin{align}\label{lem:mon} \Omega (u) \geq \epsilon (u \bullet S).\end{align}
    
\end{lemma}


Our first result concerns the case when all the constraint points are on one of the Lagrangian tori.

\begin{proposition}\label{prop:Fd}
For $d$ sufficiently large, every $(d,2d+1)$-building is of Type 2.2. Each such building, $\mathbf{F}$. consists of two essential planes, $\underline{u}$, and $u_1$, a single lower level curve which maps to $T^*L_1$ and has $2d+1$ (positive) punctures, and $2d-1$ top level planes in $\mathfrak{r}_0 \cup \mathfrak{r}_\infty$. All the curves of $\mathbf{F}$ are somewhere injective and hence regular.
\end{proposition}

\begin{proof}

\noindent The first step of the proof is to use Lemma \ref{lem:monotonicity} and the  labeling described  in Subsection \ref{common} to show that when $d$ is sufficiently large here are no  $(d,2d+1)$-buildings of Type 1 or Type 2.1.  

Arguing by contradiction we assume first that for an arbitrarily large  $d$ there is a $(d,2d+1)$-building $\mathbf{F}$ of Type $1$. Let $\underline{u}$ be the unique essential curve of $\mathbf{F}$.

\medskip

\noindent{\em Case 1.} Assume that the $k_1$ punctures of $\underline{u}$ along $L_1$ are all negative. Then \eqref{neg1} yields $$M_1^\infty \geq M_1^0 + k_1.$$ In the current setting,  \eqref{constraint1} implies that $$k_1 +N_1 \geq 2d+1.$$ Taken together, these two inequalities imply that $M^{\infty}_1+ N \geq 2d+1$ which  contradicts \eqref{S1}.

\medskip

\noindent{\em Case 2.} Assume that the $k_1$ punctures of $\underline{u}$ along $L_1$ are all positive. By \eqref{pos1} we have $$M_1^0 \geq M_1^\infty + k_1.$$ In this case, \eqref{area1} implies that 
\begin{align*}
\Omega(\underline{u})+ r(2M_1^\infty +k_1 + M_2^0 + M_2^\infty)+2r(N_0+N_1+N_2)
\leq 2rd+b  
\end{align*}
Together with the implication from  \eqref{constraint1} that $k_1 +N_1 \geq 2d     +1$ we get
\begin{align*}
\Omega(\underline{u})+ r(2M_1^\infty + M_2^0 + M_2^\infty) +rN_1 +2r(N_0 + N_2)  \leq b-r.  
\end{align*}
This implies that the numbers $\Omega(\underline{u})$,  $M_1^\infty$, $M_2^\infty$, and $N$, which implicitly depend on $d$, are uniformly bounded from above independently of $d$. By \eqref{S1} we must then have  \begin{align*} \lim_{d \to \infty} \underline{u} \cdot S_{\infty} = \infty.\end{align*}  By Lemma \ref{lem:monotonicity}, this contradicts the assertion above that $\Omega(\underline{u})$ is uniformly bounded from above independently of $d$. Cases 1 and 2 imply that for sufficiently large $d$ the building $\mathbf{F}$ cannot be of Type 1.

\medskip

Next, we assume first that for arbitrarily large $d$ there is a $(d,2d+1)$-building $\mathbf{F}$ of Type $2.1$. Let ${\bf{u}}_d=\{ u_2, \underline{u}\}$ be the essential curves of $\mathbf{F}$.

\medskip
 
 \noindent{\em Case 3.} Assume that the punctures of $\underline{u}$, and hence $\mathbf{u}_d$, along $L_1$ are all negative. By inequality \eqref{neg1} we have $M^{\infty}_1 \geq M^1_0 +k_1$. Together with \eqref{constraint2} this implies that $M^\infty_0 + N \geq 2d+1$ which contradicts \eqref{S2}.

\medskip

\noindent{\em Case 4.} Assume that  the punctures of $\underline{u}$ along $L_1$ are all positive. In this case, we have $M_1^0 \geq M_1^\infty + k_1 $ and  $k_1 +N_1 \geq 2d+1$. Together with \eqref{area2}, these imply that
\begin{align*}
\Omega(\mathbf{u}_d)+ r(2M_1^\infty + \widetilde{M}_2^0 + \widetilde{M}_2^\infty)+rN_1 +2r(N_0 + N_2)  \leq b-r.  
\end{align*}
From this it follows that $\Omega(\mathbf{u}_d)$,  $M_1^\infty$, $\widetilde{M}_2^\infty$, and $N$ are uniformly bounded from above independently of $d$. With this, equation  \eqref{S2} implies that  \begin{align*}\lim_{d \to \infty} \underline{u} \cdot S_{\infty} = \infty.\end{align*} By Lemma \ref{lem:monotonicity}, this contradicts the assertion above that $\Omega(\bf{u}_d)$ is uniformly bounded from above independently of $d$. Cases 3 and 4 imply that, for sufficiently large $d$, the building $\mathbf{F}$ cannot be of Type 2.1.

\medskip

Finally, we use index computations to show that any  $(d,2d+1)$-building  cannot be of Type $3$. Let $\mathbf{F}$ be such a building.  Let $s_i$ be the total number of positive ends of the curves of $\mathbf{F}$ that map to $T^*L_i$. The total index of the curves mapping to $T^*L_i$ is at most $2s_i -2$, where the upper bound corresponds uniquely to the case of a single curve with $s_i$ positive ends. 

To accommodate the $2d+1$ constraint points on $L_1$ we must have \begin{equation}\label{eq:s1}
    s_1  \geq 2d+2.
\end{equation}
Since two of these $s_1$ ends match with essential curves of $\mathbf{F}$, there must be $s_1 -2$ addition curves of $\mathbf{F}$ corresponding to each of the remaining ends. These curves can be included in planar sub-buildings that cover curves in $\mathfrak{r}_0 \cup \mathfrak{r}_\infty$. Similarly, 
there are also $s_2 -2$ planar sub-buildings of $\mathbf{F}$ that cover curves in $\mathfrak{s}_0 \cup \mathfrak{s}_\infty$.  

By the index bound above for curves in $T^*L_i$, the total index of the curves of $\mathbf{F}$ (defined to be the sum of the indices of the individual curves minus any matching conditions) is at most \begin{equation}\label{eq:ind3}
\begin{split}
   \mathrm{index}(\mathbf{u}) + (2s_1-2) + (2s_2-2)+ \\
   \sum_{i=1}^{s_1-2}\mathrm{index}(\mathbf{v_i}) + \sum_{j=1}^{s_2-2}\mathrm{index}(\mathbf{w_j}) - (s_1 +s_2). 
\end{split}
\end{equation}
The first term, $\mathrm{index}(\mathbf{u})$, is the total index of the three essential curves, the next two terms are upper bounds for the index of the curves in $T^*L_i$, the two sums represent the total index of the planar sub-buildings mapping into $\mathfrak{r}_0 \cup \mathfrak{r}_\infty$ and $\mathfrak{s}_0 \cup \mathfrak{s}_\infty$ respectively, and the last term 
comes from the matching constraints at Reeb orbits for the Morse-Bott Reeb flow of the flat metric on $\mathbb{T}^2$. 
As planar sub-buildings, like holomorphic planes, have odd index, we have
\begin{equation}\label{eq:ind35}
   \sum_{i=1}^{s_1-2}\mathrm{index}(\mathbf{v_i}) \geq s_1 -2; \, \, \,  \sum_{j=1}^{s_2-2}\mathrm{index}(\mathbf{w_j}) \geq s_2 -2.
\end{equation}
Since the total index is preserved under limits, the total index of $\mathbf{F}$ is equal to $2(2d+1)$, and it then follows from \eqref{eq:s1}, \eqref{eq:ind3} and  \eqref{eq:ind35} that 
\begin{equation}\label{eq:ind4}
   \mathrm{index}(\mathbf{u}) \leq 2(3-s_2). 
\end{equation}
The three essential curves of $\mathbf{F}$ are $u_1$, $\underline{u}$,  and $u_2$. The curves $u_1$ and $u_2$ are holomorphic planes and hence have odd indices. Since $s_2 \ge 2$ , we must  then have $s_2=2$, $\mathrm{index}(u_1) = \mathrm{index}(u_2) =1$, and 
$\mathrm{index}(\underline{u})=0$. Furthermore, all our inequalities become equalities, and hence $s_1 = 2d+2$. For a generic choice of $J$, 
this means the building $\mathbf{F}$ has a single curve in $T^*L_1$, and this curve is rigid (that is, the curve has deformation index $0$ when constrained to intersect the $2d+1$ points). In turn, this means the constraint points determine a finite collection of possible positive ends, and these ends will not, generically,  match with the ends of the rigid curve $\underline{u}$.

\medskip

It remains for us to show that the limit $\mathbf{F}$, of Type 2.2, consists entirely of regular curves.  Arguing as above, the total index of the curves of $\mathbf{F}$ is at most \begin{equation}\label{eq:ind5}
\begin{split}
   \mathrm{index}(\mathbf{u}) + (2s_1-2) + 2s_2 + \\
   \sum_{i=1}^{s_1-2}\mathrm{index}(\mathbf{v_i}) + \sum_{j=1}^{s_2}\mathrm{index}(\mathbf{w_j}) - (s_1 + 2s_2). 
\end{split}
\end{equation} 
The third term, $2s_2$, corresponds to curves of $\mathbf{F}$ covering cylinders in $T^*L_2$, and the next two terms represent the total index of the planar sub-buildings as above with ends on $L_1$ and $L_2$ respectively, and the  
final term corresponds to matching condition constraints.

The analogue of \eqref{eq:ind35} is
\begin{equation}\label{eq:ind36}
   \sum_{i=1}^{s_1-2}\mathrm{index}(\mathbf{v_i}) \geq s_1 -2; \, \, \,  \sum_{j=1}^{s_2}\mathrm{index}(\mathbf{w_j}) \geq s_2 -2.
\end{equation}

Since the total index of the curves of $\mathbf{F}$ is equal to $2(2d+1)$, it follows from \eqref{eq:s1}, \eqref{eq:ind5} and \eqref{eq:ind36}  that 
\begin{equation}\label{eq:ind7}
   \mathrm{index}(\mathbf{u}) \leq 2-s_2. 
\end{equation}
The two essential curves of $\mathbf{F}$ are $\underline{u}$ and $u_2$. Since  $u_2$ is a holomorphic plane it has odd index and we have \begin{equation}\label{eq:ind6}
   \mathrm{index}(\underline{u}) \leq 1-s_2. 
\end{equation} 
Therefore $s_2  \le 1$, and if $s_2=1$ then the essential curve $\underline{u}$ is rigid. In this case our various inequalities must be equalities, in particular $s_1=2d+2$ and there is a single curve in $T^*L_1$ which is also rigid. However for generic $J$ we do not expect to find rigid curves with matching asymptotics. 
Thus, $s_2 =0$ and the inequalities \eqref{eq:ind7} and \eqref{eq:ind6}, and hence \eqref{eq:ind36}, are equalities. This means that our planar sub-buildings are in fact single, somewhere injective holomorphic planes, and so all the curves of $\mathbf{F}$ are somewhere injective as required.

\end{proof}








The primary implication of Proposition \ref{prop:Fd} is the following.

\begin{proposition}\label{subaxis} For sufficiently large $e \in \NN$, there exist embedded symplectic spheres $A_0, A_{\infty}, B_0, B_{\infty} \subset X \smallsetminus (L_1 \cup L_2)$ in the class $(e,1)$, with the following properties:
\begin{enumerate}
\item $A_0$ intersects each plane in $\mathfrak{r}_0$ once and is disjoint from the planes in $\mathfrak{r}_{\infty}$;
\item $A_{\infty}$ intersects each plane in $\mathfrak{r}_{\infty}$ once and is disjoint from the planes in $\mathfrak{r}_0$;
\item $B_0$ intersects each plane in $\mathfrak{s}_0$ once and is disjoint from the planes in $\mathfrak{s}_{\infty}$;
\item $B_{\infty}$ intersects each plane in $\mathfrak{s}_{\infty}$ once and is disjoint from the planes in $\mathfrak{s}_0$;
\item either $A_0$ and $A_{\infty}$ both intersect the planes in $\mathfrak{s}_0$ and are disjoint from the planes in $\mathfrak{s}_{\infty}$, or $A_0$ and $A_{\infty}$ both intersect the planes in $\mathfrak{s}_{\infty}$ and are disjoint from the planes in $\mathfrak{s}_0$;
\item either $B_0$ and $B_{\infty}$ both intersect the planes in $\mathfrak{r}_0$ and are disjoint from the planes in $\mathfrak{r}_{\infty}$, or $B_0$ and $B_{\infty}$ both intersect the planes in $\mathfrak{r}_{\infty}$ and are disjoint from the planes in $\mathfrak{r}_0$.
\end{enumerate}
\end{proposition}

\begin{proof} We will establish the existence of the spheres $A_0$ and $A_{\infty}$ . The existence of the spheres $B_0$ and $B_{\infty}$ follows from the same argument by exchanging the roles of $L_1$ and $L_2$.  The starting point for the argument is an $(e,2e+1)$-building, $\mathbf{F}$, of the Type 2.2 as in Proposition \ref{prop:Fd}. We will obtain $A_0$ and $A_{\infty}$ by  deforming $\mathbf{F}$ near $L_1$ and then compactifying and smoothing.

To deform  $\mathbf{F}$ we will use the procedure developed to prove Proposition 3.24 in \cite{hiker}. Let $\mathbf{v} =(a,b)$ be a vector in $(-\epsilon, \epsilon)^2$ and let $L_1(\mathbf{v})$ be the Lagrangian torus in $\mathcal{U}_1$ defined by $\{p_1=a, p_2=b\}$. Let $u$ be a top level curve of $\mathbf{F}$. By Proposition \ref{prop:Fd}, $u$ is a regular plane asymptotic to $L_1$. With this,  Lemmas 3.27 and 3.29 of \cite{hiker} imply the following.

\begin{lemma}\label{lem:shift1}
 If $\|v\|$ is sufficiently small, and $|a|$ is sufficiently small with respect to $|b|$, then there is an almost complex structure $J_{\mathbf{v}}$ on $X \smallsetminus (L_1 \cup L_1(\mathbf{v})\cup L_2)$ and a $J_{\mathbf{v}}$-holomorphic curve $u(\mathbf{v})$ whose domain is identical to that of $u$ and whose asymptotics are identical to those of $u$ but with the ones on $L_1$ translated to  $L_1(\mathbf{v})$. 
\end{lemma}

For  $\mathbf{v}$ as in Lemma \ref{lem:shift1}, let $\mathbf{F}(\mathbf{v})$ be the building obtained by translating each top level curve $u$ of $\mathbf{F}$ to $u(\mathbf{v})$ and identifying the bottom level curves of of $\mathbf{F}$ with the same curves now formally mapping to a copy of $T^*\TT^2$  that corresponds to $L_1(\mathbf{v})$. 

Compactifying the building $\mathbf{F}(\mathbf{v})$ and smoothing the resulting map near $L_1(\mathbf{v})$, we get a symplectic  embedding $$F_e(\mathbf{v}) \colon S^2 \to X \smallsetminus (L_1 \cup L_2).$$ Arguing as in  Lemma 3.33 of \cite{hiker}, we have the following. 

\begin{lemma}\label{lem:shift2}
   Suppose that  $a \neq 0$. If $b<0$, then $F_e(\mathbf{v})(S^2)$ intersects each plane in $\mathfrak{r}_0$ once and is disjoint from the planes in $\mathfrak{r}_{\infty}$. If $b>0$, then $F_e(\mathbf{v})(S^2)$ is disjoint from the planes in $\mathfrak{r}_0$ and intersects each plane in $\mathfrak{r}_{\infty}$ once.
 \end{lemma}

For $a_1\neq 0$ and  $b_1<0$ as above, set $$A_0 = F_e(\mathbf{v})(S^2).$$
Similarly, choose $a\neq 0$, and  $b>0$, and set $$A_\infty = F_e(\mathbf{v})(S^2).$$ For these choices the symplectic spheres $A_0$ and $A_\infty$ satisfy conditions (1) and (2) of Proposition \ref{subaxis}, respectively. To verify property (5) we note that Proposition \ref{prop:Fd} implies that either $\underline{u} \bullet \mathfrak{s}_0 =1$ and  $\underline{u} \bullet \mathfrak{s}_\infty =0$ or $\underline{u} \bullet \mathfrak{s}_0 =0$ and  $\underline{u} \bullet \mathfrak{s}_\infty =1$. By the construction of $A_0$ and $A_{\infty}$ from $\mathbf{F}(\mathbf{v})$, the same alternative holds for the intersections of $A_0$ and $A_{\infty}$ with  $\mathfrak{s}_0 $ and  $\mathfrak{s}_\infty$.

\end{proof}

We now consider $(d,d)$-buildings.

\begin{proposition}\label{prop:dd} For $d$ sufficiently large, every $(d,d)$-building is of Type 3. Each such building, $\mathbf{F}$, has three essential curves $u_2$, $\underline{u}$, and $u_1$ with areas $s$, $r$, and $b-s$, respectively. Moreover, the only top level curves of $\mathbf{F}$, other than its three essential curves, are $d-1$ broken planes in $\mathfrak{r}_0 \cup \mathfrak{r}_{\infty}$ and $d$ broken planes in $\mathfrak{s}_0 \cup \mathfrak{s}_{\infty}$. 
All the constituent curves of  $\mathbf{F}$ are somewhere injective and hence regular. 
\end{proposition}

\subsubsection{The proof of Proposition \ref{prop:dd}} The proof uses the symplectic spheres of Proposition \ref{subaxis} for a fixed  integer, say  $e$. By relabeling, if necessary, we may assume that $A_0$ and $A_{\infty}$ both intersect $\mathfrak{s}_{\infty}$, and that $B_0$ and $B_{\infty}$ both intersect $\mathfrak{r}_{\infty}$. We may also assume that these spheres are $J$-holomorphic. Indeed, the sphere $A_0$ is already holomorphic away from an $L_1(\mathbf{v_0})$, with $\mathbf{v_0} = (a_1,b_1)$ and $b_1<0$, and since the spheres is symplectic we can deform $J$ in a neighborhood of $L_1(\mathbf{v_0})$ to make $A_0$ holomorphic. We may assume this neighborhood is disjoint from $B_0$ and $B_{\infty}$, since they are both deformations of buildings with no ends on $L_1$. We may also assume the neighborhood of $L_1(\mathbf{v_0})$ is disjoint from $A_{\infty}$. This follows as in Lemma 3.35 from \cite{hiker}, using a continuation argument, as $L_1(\mathbf{v_0})$ may be chosen close to $\mathfrak{r}_0$ while $A_{\infty}$ is disjoint from the planes in $\mathfrak{r}_0$. As the deformation of $J$ making $A_0$ holomorphic does not affect $B_0$, $B_{\infty}$ or $A_{\infty}$, or indeed any of the families of planes $\mathfrak{r}_{0}$, $\mathfrak{r}_{\infty}$, $\mathfrak{s}_{0}$, $\mathfrak{s}_{\infty}$, we may repeat the argument for the other spheres.
As a result of this, the spheres $A_0$, $A_{\infty}$, $B_0$ and $B_{\infty}$ can each play the role of $S$ in the applications of Lemma \ref{lem:monotonicity} below.

\bigskip

\noindent{\bf Step 1.} First we show that $\mathbf{F}$ cannot be of Type 1 when $d$ is sufficiently large. Arguing by contradiction we assume that we can find a Type 1 building $\mathbf{F}$ for an arbitrarily large $d$. This building  has a single essential curve, $\underline{u}$,  and there are four cases to consider corresponding to the possibilities that the punctures of $\underline{u}$ along $L_1$ are all negative/positive, and the punctures of $\underline{u}$ along $L_2$ are all negative/positive. In each case, we will arrive at a contradiction when $d$ is large enough.

\bigskip

\noindent{\it{Case 1.}} Assume first that the punctures of $\underline{u}$ along $L_1$ are all negative, and the punctures of $\underline{u}$ along $L_2$ are all negative. This implies the collection of inequalities:
\begin{align*}
    &M_1^\infty \geq M_1^0 + k_1,\\
    &M_2^\infty \geq M_2^0 + k_2,\\
    &k_1 +N_1 \geq d,\\
    &k_2 +N_2 \geq d+1.
\end{align*}
From these it follows that 
$$M_1^\infty + M_2^\infty + N_1 + N_2 \ge 2d+1$$
which contradicts \eqref{S1}.

\bigskip

\noindent{\it{Case 2.}} If the punctures of $\underline{u}$ along $L_1$ are all negative, and the punctures of $\underline{u}$ along $L_2$ are all positive we have:
\begin{align*}
    &M_1^\infty \geq M_1^0 + k_1,\\
    &M_2^0 \geq M_2^\infty + k_2,\\
    &k_1 +N_1 \geq d,\\
    &k_2 +N_2 \geq d+1,
\end{align*}
and hence
$$M_1^\infty + M_2^0 + N_1 + N_2 \ge 2d+1.$$
However, taking the intersection of $\mathbf{F}$ with the sphere $B_0$ yields 
\begin{align*}
    \underline{u} \bullet B_0 + M_1^\infty+M_2^0 +N =d+e.
\end{align*}
Fixing $e$, these two formulas cannot both hold when $d$ is  sufficiently large.

\bigskip

\noindent{\it{Case 3.}} Arguing similarly, if the punctures of $\underline{u}$ along $L_1$ are all positive, and the punctures of $\underline{u}$ along $L_2$ are all negative we have
$$M_1^0 + M_2^\infty + N_1 + N_2 \ge 2d+1.$$ For fixed $e$ and large $d$ this contradicts the equality
\begin{align*}
  \mathbf{F} \bullet A_0 = \underline{u} \bullet A_0 + M_1^0 +M_2^\infty +N =d+e.
\end{align*}

\bigskip

\noindent{\it{Case 4.}} Finally, if the punctures of $\underline{u}$ along $L_1$ and $L_2$ are all positive then we have 
\begin{align}\label{c41}
M_1^0 + M_2^0 + N_1 + N_2 \ge 2d+1. 
\end{align}
Computing areas we have
\begin{align*}
\Omega(\underline{u}) + r(M_1^0 +M_1^{\infty}+M_2^0 +M_2^{\infty}) + 2rN = 2dr + b
\end{align*}
and using \eqref{c41} this gives
\begin{align}\label{c42}
\Omega(\underline{u}) + r(M_1^{\infty}+ M_2^{\infty})
&\leq  b - r - rN.
\end{align}

From this we see that $N$ is uniformly bounded, independently of $d$, as is the total area of the essential curve $\underline{u}$ and all curves covering $\mathfrak{r}_{\infty}$ and $\mathfrak{s}_{\infty}$.
But, computing intersections with $A_{\infty}$, we have
\begin{align*}
  \mathbf{F} \bullet A_{\infty} = \underline{u} \bullet A_{\infty} + M_1^{\infty} +M_2^{\infty} +N =d+e
\end{align*}
and, by \eqref{c42}, for sufficiently large $d$ this contradicts Lemma \ref{lem:monotonicity}.





\bigskip

\noindent{\bf Step 2.}  Next we show that  $\mathbf{F}$ cannot be of Type 2.1 when $d$ is large. Assume that for an arbitrarily large $d$ we can find a limit $\mathbf{F}$ of Type 2.1. This building has two essential curves 
$\mathbf{u} = \{u_2,\underline{u_d}\}$ and the punctures of $\underline{u}$ along $L_1$ are either all negative or positive.

\bigskip

\noindent{\it{Case 1.}} Suppose the punctures of $\underline{u}$ along $L_1$ are negative.  In this case we have:
\begin{align*}
    &M_1^\infty \geq M_1^0 + k_1,\\
    &k_1 +N_1 \geq d,\\
    &\widetilde{M}^0_2+\widetilde{M}^\infty_2   \geq d+2.
\end{align*}
By the last of these inequalities, either $\widetilde{M}^0_2 \ge (d+2)/2$ or $\widetilde{M}^\infty_2 \ge (d+2)/2$. Suppose that  \begin{equation}\label{m2}\widetilde{M}^\infty_2 \ge (d+2)/2. \end{equation} Then, intersecting $\mathbf{F}$ with $A_\infty$ we get
\begin{equation}\label{ainf}d+e =\mathbf{u} \cdot A_\infty + M_1^{\infty} + \widetilde{M}^\infty_2 +N.\end{equation}
By \eqref{m2} and the inequalities for $M^{\infty}_1$ and $N_1$ above, it follows that the right hand side of  \eqref{ainf} is at least $d + (d+1)/2$.  This produces a contradiction whenever  $d> 2e-2$. 

Similarly, if $\widetilde{M}^0_2 \ge (d+2)/2$, then intersecting $\mathbf{F}$ with $B_0$ we get
\begin{align*}
    d+e = \mathbf{u} \cdot B_0 + M_1^\infty + \widetilde{M}^0_2 +N. 
\end{align*}
The right hand side is at least $d + (d+2)/2$ and we again get a contradiction whenever  $d> 2e-2$.



\bigskip

\noindent{\it{Case 2.}} When the punctures of $\underline{u}$ along $L_1$ are positive, we have:
\begin{align*}
    &M_1^0 \geq M_1^\infty + k_1,\\
    &k_1 +N_1 \geq d,\\
    &\widetilde{M}^0_2+\widetilde{M}^\infty_2   \geq d+2.
 \end{align*}
In particular $M_1^0 + N \ge d$.

Intersecting with $A_0$ we get
\begin{align*}
    d+e = \mathbf{u} \cdot A_0 + M_1^0 + \widetilde{M}^\infty_2 +N,
\end{align*}
and together with the above this gives $\widetilde{M}^\infty_2 \le e$, so $\widetilde{M}^0_2 \ge d+2-e$.

Next we recall the area of our limiting building:
\begin{align*}
\Omega(\underline{u}) + r(M_1^0 + M_1^{\infty}+ \widetilde{M}^0_2+\widetilde{M}^\infty_2) + 2rN.
\end{align*}
Combining with the above inequalities this implies

\begin{align*}
\Omega(\underline{u}) + r(M_1^{\infty}+\widetilde{M}^\infty_2) 
&\leq 2dr + b - r(d - N) -r(d+2-e) - 2rN\\ &= b +r(e-2) - rN.
\end{align*}

This means that for fixed $e$ the count $N$, plus the total area of the essential curves together with all curves in our building lying in $\mathfrak{r}_{\infty}$ and $\mathfrak{s}_{\infty}$, is bounded independently of $d$. Thus all curves intersecting $B_{\infty}$ have uniformly bounded area, but as this intersection number is $d+e$ we contradict Lemma \ref{lem:monotonicity} when $d$ is large.

\bigskip

\noindent{\bf Step 3.}  The proof that $\mathbf{F}$ cannot be of Type 2.2 when $d$ is sufficiently large is nearly identical to the proof in {\bf Step 2}. The only change, after relabeling, is the change of the number of constraint points from $d$ to $d+1$.  This does not affect the argument.

\bigskip

\noindent{\bf Step 4.}  It remains to verify the various properties of the constituent curves of the Type 3 limit $\mathbf{F}$.\\

\noindent{\it Indices.} We begin by considering the indices of the curves of $\mathbf{F}$, starting at bottom level. Let $s_i$ be the total number of positive ends of the curves of $\mathbf{F}$ that map to $T^*L_i$.

The total index of the curves mapping to $T^*L_i$ is denoted by $I_i$, and is at most $2s_i -2$. This upper bound corresponds to the case of a single curve with $s_i$ asymptotic ends. (As shown below, this is the case that actually occurs.) In any case the $I_i$ are necessarily even. To satisfy the point constraints we must have
\begin{equation}\label{consrtaint}
I_1 \ge 2d; \, \, I_2 \ge 2(d+1).
\end{equation}

Since $\mathbf{F}$ is of Type 3, two of the $s_i$ positive ends match with essential curves of $\mathbf{F}$. So, there must be $s_i -2$ additional curves of $\mathbf{F}$ corresponding to each of the remaining ends. Each of these curves form part of a planar sub-building in $\mathfrak{r}_0 \cup \mathfrak{r}_\infty$ when $i=1$ or a  planar sub-building in $\mathfrak{s}_0 \cup \mathfrak{s}_\infty$ when $i=2$. We denote the indices of these sub-buildings by $J_{k,i}$ for $i=1,2$. These numbers are odd, and equal $1$ exactly when the sub-building consists of a single plane.

The total index of the curves of $\mathbf{F}$ is given by \begin{equation}\label{eq:indg3}
   \mathrm{index}(\mathbf{u}) + I_1 + I_2 + \sum_{k=1}^{s_1-2} (J_{k,1}-1) + \sum_{k=1}^{s_2-2} (J_{k,2}-1) - 4 = 2(2d+1), 
\end{equation}
where the first term is the total index of the three essential curves, we subtract $1$ from the index of the planar sub-buildings for the matching conditions are their end, and the final term is the matching condition for the $4$ ends of the essential curves. We recall this sum is $2(2d+1)$ because the total index of a limiting building is equal to the index of holomorphic curves representing the class $(d,1)$.

Hence, using \eqref{consrtaint},
\begin{equation*}
    \mathrm{index}(\mathbf{u}) = \mathrm{index}(u_2)+\mathrm{index}(\underline{u})+\mathrm{index}(u_1) \leq 4,
\end{equation*}
with equality only if $I_1=2d$, $I_2=2(d+1)$ and all $J_{k,i}=1$.

Suppose $\mathrm{index}(\underline{u})=0$, so the essential cylinder is rigid. Then, in order to have matching ends in our building, neither the lower level curves in $T^* L_1$ nor $T^* L_2$ can be rigid, that is, both $I_1 > 2d$ and $I_2 > 2(d+1)$. But as these terms are even, we then have $$\mathrm{index}(\mathbf{u}) = \mathrm{index}(u_2)+\mathrm{index}(\underline{u})+\mathrm{index}(u_1)  \leq 0,$$ and this is a contradiction as holomorphic planes have odd index, and for regular almost complex structures the index is nonnegative.
Similarly we cannot have $\mathrm{index}(\underline{u}) \ge 4$ and both essential planes having odd index, so we conclude that
\begin{equation*}
    \mathrm{index}(u_2)=1, \, \, \mathrm{index}(\underline{u})=2, \, \, \mathrm{index}(u_1)=1.
\end{equation*}



\bigskip

\noindent{\it Areas.} Next we verify the symplectic areas of the essential curves of $\mathbf{F}$ asserted in Proposition \ref{prop:dd}. 

\begin{lemma}\label{lem:s}
The $\Omega$-area of   $u_2$ is equal to $s$.
\end{lemma}

\noindent{\em The Proof of Lemma \ref{lem:s}.} We begin with a sequence of intermediate claims. Let $$v_2 \colon (D,\partial D) \to (X , L_2)$$ be the compactification of $u_2$, where  $D \subset \RR^2$ is the closed unit disk. Setting $\gamma_{2,d} = [v_2(\partial D)]  \in H_1(L_2;\ZZ)$, we have $\gamma_{2,d} \bullet \beta_2 =\pm 1.$

\begin{claim}\label{claim:u}
    There is a  smooth map $$u \colon (D,\partial D) \to (X \smallsetminus S_\infty, L_2)$$ such that 
    \begin{enumerate}
         \item $p \circ u (D)$  is equal to the closure of $A$ in $S_{\infty}$,
         \item $[u(\partial D)] =\gamma_{2,d} + \ell \beta_2 \in H_1(L_2; \ZZ)$ for some $\ell \in \ZZ$,
        \item The Maslov index of $u$ is $2$.
        
    \end{enumerate}
\end{claim}
\begin{proof} 

Let $u'' \colon D \to  S_{\infty}$ be a smooth embedding such that $u''(D)$ is equal to the closure of $A$. 
Trivializing the bundle $p$ over $A$ we can then lift $u''$ to obtain a smooth map $u' \colon (D,\partial D) \to (X \smallsetminus S_\infty, L_2)$ that satisfies conditions (1) and (2). The Maslov index of $u'$ is even. Gluing a suitable multiple of a compactified disk from $\mathfrak{s}_0$ to $u'$ we get a disk $u$ of Maslov index 2 which still maps into $X \smallsetminus S_\infty$ and still satisfies conditions (1) and (2).

\end{proof}

\begin{claim}
A map $u$ as in Claim \ref{claim:u} has symplectic area equal to either $s$ or $2r-s$.
\end{claim} 

\begin{proof} View $L_2$ as a Lagrangian torus in the exact symplectic manifold $X \smallsetminus (S_\infty \cup T_{\infty})$.
By condition (1) of Lemma \ref{claim:u}, $u$ maps into $X \smallsetminus S_\infty \cup T_{\infty}$. Condition (2) implies that  $\{ \beta_2, [u(\partial D)]\}$ is an integral basis of $H_1(L_2; \ZZ)$. The area class and Maslov class of $L_2$ in $X \smallsetminus S_\infty \cup T_{\infty}$ are therefore determined by the areas and indices of the classes $\beta_2$ and  $[u(\partial D)]$. 
(These are well defined independently of a bounding disk, as we work now in $X \smallsetminus (S_\infty \cup T_{\infty}) \subset \RR^4$.)
The area of $\beta_2$ is $r$ and its Maslov index is $2$. By condition (3) of Lemma \ref{claim:u} the Maslov index of $[u(\partial D)]$ is also $2$. Then since $L_2$ is an image of $L(r,s)$ under a symplectomorphism of $\RR^4$, and $\{ \beta_2, [u(\partial D)]\}$ is a basis of $H_1(L_2; \ZZ)$, the area of $[u(\partial D)]$ must either be $s$ or $2r-s$ (depending on the orientation of our integral basis).
\end{proof}

\begin{claim}
The map $v_2$ has the same $\Omega$-area as the map $u$ in Claim \ref{claim:u}.
\end{claim} 

\begin{proof}
The glued map $v_2-u$ is a cycle in $X$ with boundary on $L_2$. Since its intersection number with both $T_0$ and $T_\infty$ is zero, it is homologous to a cycle that is contained in a (compactified) broken sphere of the foliation and has boundary on the equator. Since the Maslov index of the new cycle is zero, its symplectic area also vanishes. Hence, the area of $v_2-u$ is zero, as desired.
    
\end{proof}

We can now complete the proof of Lemma \ref{lem:s}. By the results above, it suffices to prove that the $\Omega$-area of $u_2$, and hence $v_2$, is not equal to $2r-s.$

\bigskip

\noindent{\it Case 1.} Suppose $u_2 \bullet S_\infty=0$.  Then, by positivity of intersection $u_2$ is a pseudo-holomorphic curve in $X \smallsetminus (S_\infty \cup T_\infty \cup L_2) \subset \RR^4$. It then follows from Corollary \ref{>r} that the symplectic area of $u_2$ must be at least $r$, which is greater than $2r-s$ by hypothesis.


\bigskip

\noindent{\it Case 2.} Suppose that $u_2 \bullet S_\infty=m \ge 1$. Since $S_\infty$ has been inflated  by $2r-1$, the  symplectic area of $u_2$ is at least $m(2r-1)$, which, again by hypothesis, is strictly greater than $2r-s$.

\medskip

\noindent This completes the proof of Lemma \ref{lem:s}. 

\medskip

As $L_1$ is the standard product torus, and $u_1$ can be compactified to give a Maslov $2$ disk, we also have the following.
\begin{lemma}\label{lem:b-s}
The $\Omega$-area of $u_1$ is equal to $b-s$.
\end{lemma}

Finally, we add the areas of all curves in our limiting building. Writing $v_{k,1}$ and $v_{k,2}$ for the planar sub-buildings asymptotic to $L_1$ and $L_2$ respectively this gives
\begin{align*}
2dr + b &= \Omega(u_1) + \Omega(u_2) + \Omega(\underline{u}) + \sum_{k=1}^{s_1-2} \Omega(v_{k,1}-1) + \sum_{k=1}^{s_2-2} \Omega(v_{k,2}) \\
{}&\geq (b-s) + s + \Omega(\underline{u}) + r(s_1 -2) + r(s_2-2) \\
{}&\geq b + \Omega(\underline{u}) + r(d-1) + rd\\
{}&= \Omega(\underline{u}) + b + r(2d-1).
\end{align*}
Here, the first inequality uses Lemmas \ref{lem:s} and \ref{lem:b-s}, then just says our planar sub-buildings have area at least $r$. The second inequality applies \eqref{consrtaint}.

It follows that $\Omega(\underline{u}) \le r$, with equality only if $s_1 = d+1$ and $s_2 = d+2$. The essential curves $\underline{u}$ and $u_2$, the planes $v_{k,2}$ and lower level curves in $T^* L_2$ can be glued to give a disk with boundary on $L_1$, lying in $X \smallsetminus T_\infty$. Such a disk has even Maslov class, say $2n$, and then area $s + (n-1)r$. As $\Omega(u_2)=s$ and the $v_{k,2}$ have area a multiple of $r$, we see that $\underline{u}$ must also have area a multiple of $r$. As a holomorphic curve it has positive area, and so $\Omega(\underline{u}) = r$, and consequently $s_1 = d+1$ and $s_2 = d+2$.

With this, the proof of Proposition \ref{prop:dd} is complete. 

\bigskip

An entirely similar argument yields the following.

\begin{proposition}\label{prop:dd+1} For $d$ sufficiently large, every $(d,d+1)$-building is of Type 3. Each such building, $\mathbf{G}$, has three essential curves $u_2$, $\underline{u}$, and $u_1$ with areas $s$, $r$, and $b-s$, respectively. Moreover, the only top level curves of $\mathbf{G}$, other than its three essential curves, are $d$ broken planes in $\mathfrak{r}_0 \cup \mathfrak{r}_{\infty}$ and $d-1$ broken planes in $\mathfrak{s}_0 \cup \mathfrak{s}_{\infty}$. 
All the constituent curves of $\mathbf{G}$ are somewhere injective and hence regular. 
\end{proposition}

\subsection{Completion of the proof of Proposition \ref{prop:submanifolds}}
\label{sub:shift}

Let $D_A$, $O_B$ and $D_C$ be compactifications of the three essential curves of $\mathbf{F}$. Conditions (2)-(8) of Proposition \ref{prop:submanifolds} follow directly from this choice and Proposition \ref{prop:dd}. As established in Proposition 3.24 of \cite{hiker},  the desired embedded symplectic spheres $F$ and $G$ can be obtained by deforming the curves of the buildings $\mathbf{F}$ and $\mathbf{G}$ and then compactifying the new buildings to symplectic spheres. In particular, deformations can be prescribed to achieve the intersection properties (9) in Proposition \ref{prop:submanifolds}. The relevant relabeling needed to invoke Proposition 3.4 of \cite{hiker} directly, is as follows: $$\mathbb{L} \to L_2$$ $$L_{1,1} \to L_1$$ $$\mathbb{E} \to D_A$$ $$E_{1,1} \to D_C$$ $$A_0 \to A$$ $$A_\infty \to C.$$  
Proposition 3.24 in \cite{hiker} unfortunately does not state the intersections with $O_B$ as required for (9)(f), however this is implicit in the proof. Indeed, the sphere $G$ intersects the building $\mathbf{F}$ at least $d$ times in $p^{-1}(A)$ and $d$ times in $p^{-1}(C)$, and is disjoint from $L_1$ and $L_2$. Further we may assume $G$ is holomorphic near any intersections. As both $G$ and $\mathbf{F}$ represent the class $(d,1)$, and $(d,1) \cdot (d,1) = 2d$, this implies that there are actually precisely $2d$ intersections in $p^{-1}(A \cup C)$ and $G$ is disjoint from $O_B$. 

Regarding the deformation $F$ of $\mathbf{F}$, \cite[Lemma 3.33]{hiker} implies that $F$ intersects $\mathbf{F}$ precisely $d$ times in the closure of $p^{-1}(C)$ (we include the closure here to include intersections with planes in $\mathfrak{s}_0 \cup \mathfrak{s}_{\infty}$) and precisely $d-1$ times in $p^{-1}(A)$; in other words, each plane in $\mathfrak{r}_0 \cup \mathfrak{r}_{\infty} \cup \mathfrak{s}_0 \cup \mathfrak{s}_{\infty}$ leads to exactly one intersection. As $F$ is disjoint from $L_1 \cup L_2$ this leaves one additional intersection with $\mathbf{F}$, which occurs on $O_B$.

The remaining property, (1), is a simple consequence of the standard compactification procedure. $\Box$


\section{Proof of Theorem \ref{thm:hm}}

It suffices to identify a Lagrangian torus $\mathbf{L} \subset P(2r,b)$ that is Hamiltonian isotopic to $L(r,s)$, and then to construct a Hamiltonian diffeomorphism of $P(2r,b)$ that  displaces each of the $L_j$ from $\mathbf{L}$.

\medskip

\noindent{\bf Product Hamiltonian flows.} Arguing as in \cite{ho}, see also \cite{hicksmak}, we will construct the desired map using Hamiltonian diffeomorphisms of $\CC^2$ generated by functions of the form $$F(z_1,z_2) = G(z_1)H(z_2).$$ Denoting the time $t$ Hamiltonian flow of  $S \in C^{\infty}(M\times \RR /\ZZ)$ by $\phi^t_S$, we have  
\begin{align}\label{move}
    \phi^1_F(z_1,z_2) = \left(\phi^{H(z_2)}_G (z_1),\, \phi^{G(z_1)}_H (z_2) \right).
\end{align} We now recall a model setting in which maps of this type can be used to kill intersection points efficiently, in terms of the size of their support. 

Let $A$ and $B$ be subsets of $\CC^2$ and let $\mathrm{pr}_i$ be projection to the $z_i$-plane. Suppose that $\mathrm{pr}_1(A) = [-2a,\,a ] \times \{0\}$, for some $a>0$, and  $\mathrm{pr}_1(B) = \{0 \} \times [-1,\, 1]$. 
Suppose also that $$\mathrm{pr}_1(A \cap B)=(0,0).$$

Let $H(z_2)$ be a Hamiltonian, with $\min(H)=0$, whose time one flow disjoins $$\mathrm{pr}_2(\mathrm{pr}_1^{-1}(0,0) \cap A)$$ from $$\mathrm{pr}_2(\mathrm{pr}_1^{-1}(0,0) \cap B).$$
Define $G(z_1)$, near $\mathrm{pr}_1(A)$, to be a smoothing of the piecewise linear function  $$\chi(x_1) = \begin{cases} 0, &\text{for } x_1<-a+\varepsilon , \\
\frac{x+a-\varepsilon}{a-2\varepsilon}, &\text{for } -a+\varepsilon \leq x_1 \leq -\varepsilon , \\1, &\text{for } x_1>-\varepsilon,
\end{cases}$$ for some $0<\varepsilon<a.$ 
On $A$, the map $\phi^1_F$ is given (approximately) by the formula $$((x_1, y_1- H(x_2,y_2)\chi'(x_1)),\phi^{\chi(x_1)}_H(x_2,y_2)).$$
From this, and our choice of $H$, it follows easily that $\phi^1_F(A)$ is disjoint from $B$. Moreover, the projection $\mathrm{pr}_1(\Phi^1_F(A))$ is (approximately) equal to 
$$\mathrm{pr}_1(A) \cup ([-1, \, 0] \times [-\max(H)/a,\, 0]).$$ In particular, in the $z_1$-projection, the action of $\phi^1_F$ can be viewed as being local, and the rectangle $[-1, \, 0] \times [-\max(H)/a,\,0]$ can be viewed as the {\it cost} of disjoining $A$ from $B$. This is illustrated in Figure \ref{fig:model}. In more general situations, like those to come, this rectangle must be disjoint from $\mathrm{pr}_1(B)$ in order to avoid the creation of new intersections.

\begin{figure}[!h]
    \centering
    \includegraphics[width=1\linewidth]{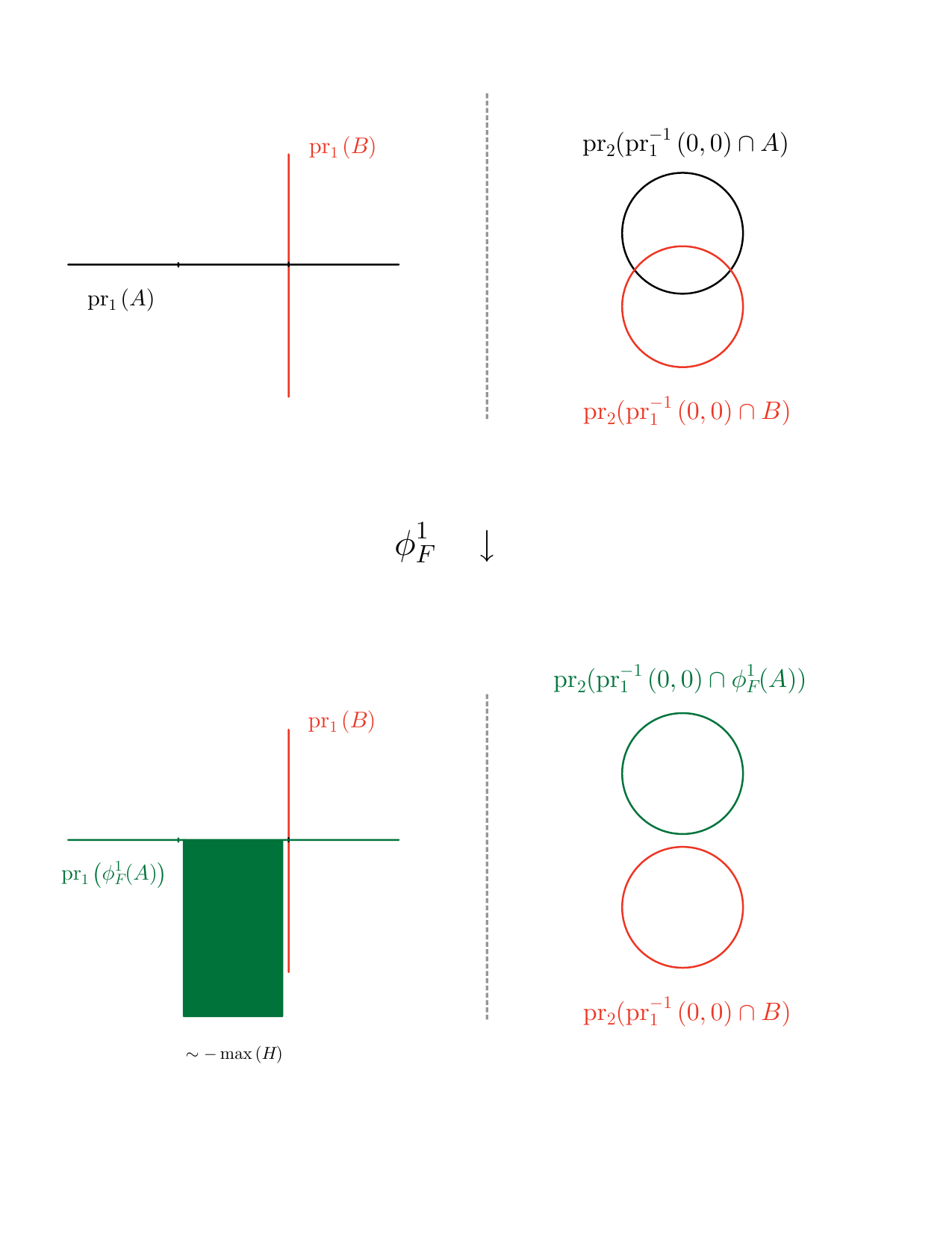}
    \caption{Displacing local models.}
    \label{fig:model}
\end{figure}

\medskip

With this model in hand, we begin the proof of Theorem \ref{thm:hm}. Before specifying $\mathbf{L}$ we reposition the $L_j = L(r) \times \Lambda_j$. We begin by repositioning the closed curves $\Lambda_j$ in $D(b)$. Equipping $D(b)$ with the rescaled polar coordinates $(\rho, \theta) \in [0,b) \times [0,1]$, fix an $a \in \left(\frac{s}{b},\frac{1}{k}\right)$ and move each $\Lambda_j$ so that is a smoothing of the boundary of $$\left[\epsilon, \frac{s}{a} +\epsilon\right] \times \left((j-1)(a+\delta), ja +(j-1)\delta\right)$$ where $\epsilon>0$ is arbitrarily small and $\delta$ is defined by
$ak + (k-1)\delta=1.$
See Figure \ref{fig:wedge}. 

\begin{figure}[!h]
    \centering
    \vspace{-0cm}
    \includegraphics[width=5.5in, height=7in]{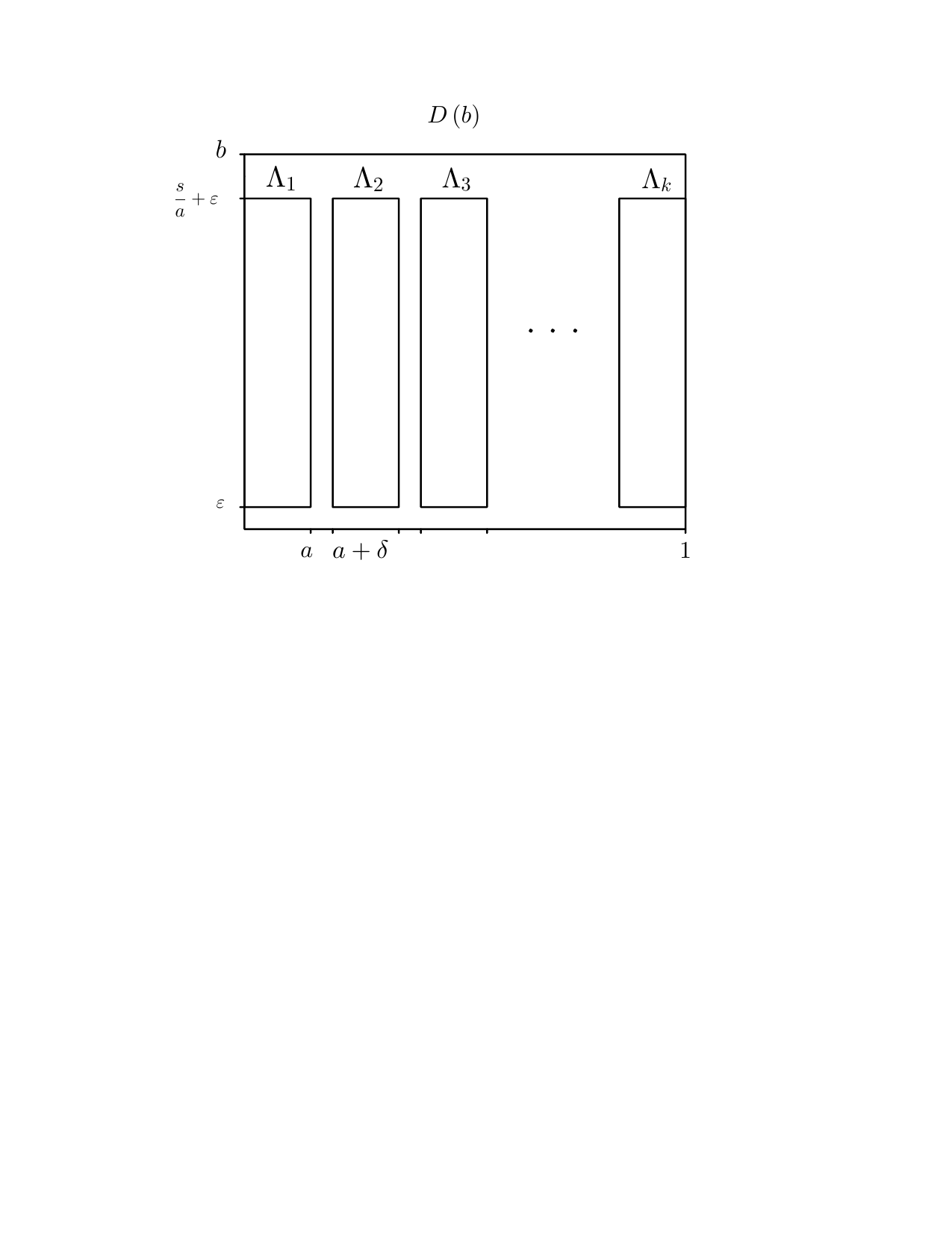}
     \vspace{-8cm}
    \caption{The repositioned circles $\Lambda_j$.}
    \label{fig:wedge}
\end{figure}

Identifying $D(2r)$with a square of side length $\sqrt{2r}$, we fix two smooth curves, $\gamma_1$ and $\gamma_2$, of area $r$ that are smoothings of the piecewise linear curves pictured in Figure \ref{fig:gamma}. The green rectangles appearing in Figure \ref{fig:gamma} all have area $\mathfrak{o}/k$ for the overlap $\mathfrak{o}$ defined as in \eqref{eq:g}. 
\begin{figure}[!h]
    \centering
    \includegraphics[width=\linewidth]{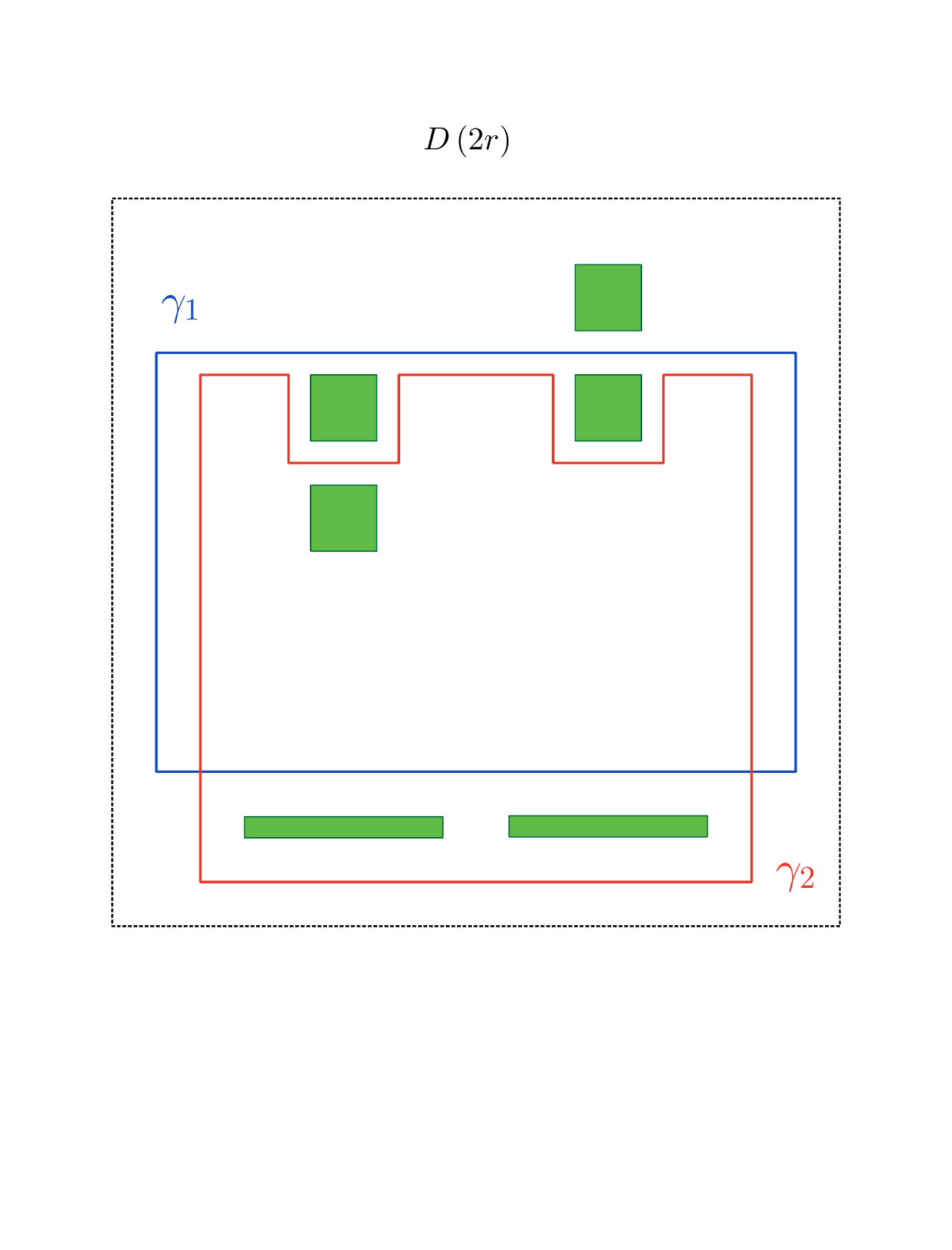}
      \vspace{-3.5cm}
    \caption{Regions in the disk $D(2r)$.}
    \label{fig:gamma}
\end{figure}

The repositioned $L_j$ are then defined by setting  $L_j = \gamma_1 \times \Lambda_j$ if $j$ is odd and $L_j = \gamma_2 \times \Lambda_j$ if $j$ is even. 
The Lagrangian $\mathbf{L}$ is chosen to be of the form $\Gamma \times \Sigma$ where $\Gamma$ is a curve in $D(2r)$ and $\Sigma$ is a closed curve in $D(b)$. We choose $\Gamma$ to be a smoothing of the piecewise linear curve in $D(2r)$, bounding area $r$ pictured in Figure \ref{fig:gamma123}.
\begin{figure}[!h]
        \centering
    \includegraphics[width=\linewidth]{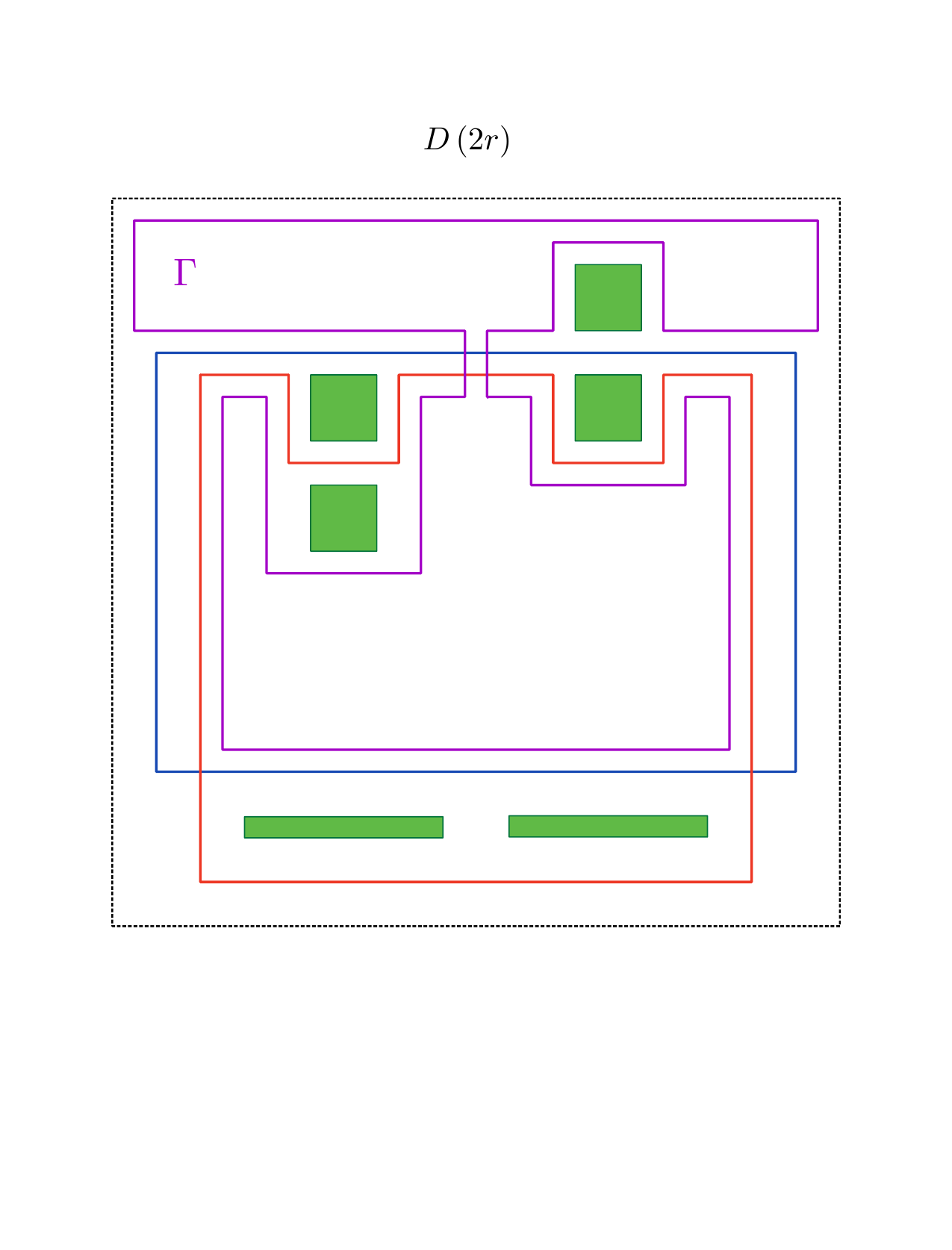}
      \vspace{-3.5cm}
    \caption{The circle $\Gamma$.}
    \label{fig:gamma123}
\end{figure}
The closed curve $\Sigma$ in $D(b)$ of area $s$ is chosen to that it intersects each $\Lambda_j$ as pictured in Figure \ref{fig:Sigma1}. Here, again, the green rectangles are assumed to have area equal to $\mathfrak{o}/k$.
\begin{figure}[!h]
    \centering 
    \includegraphics[width=\linewidth]{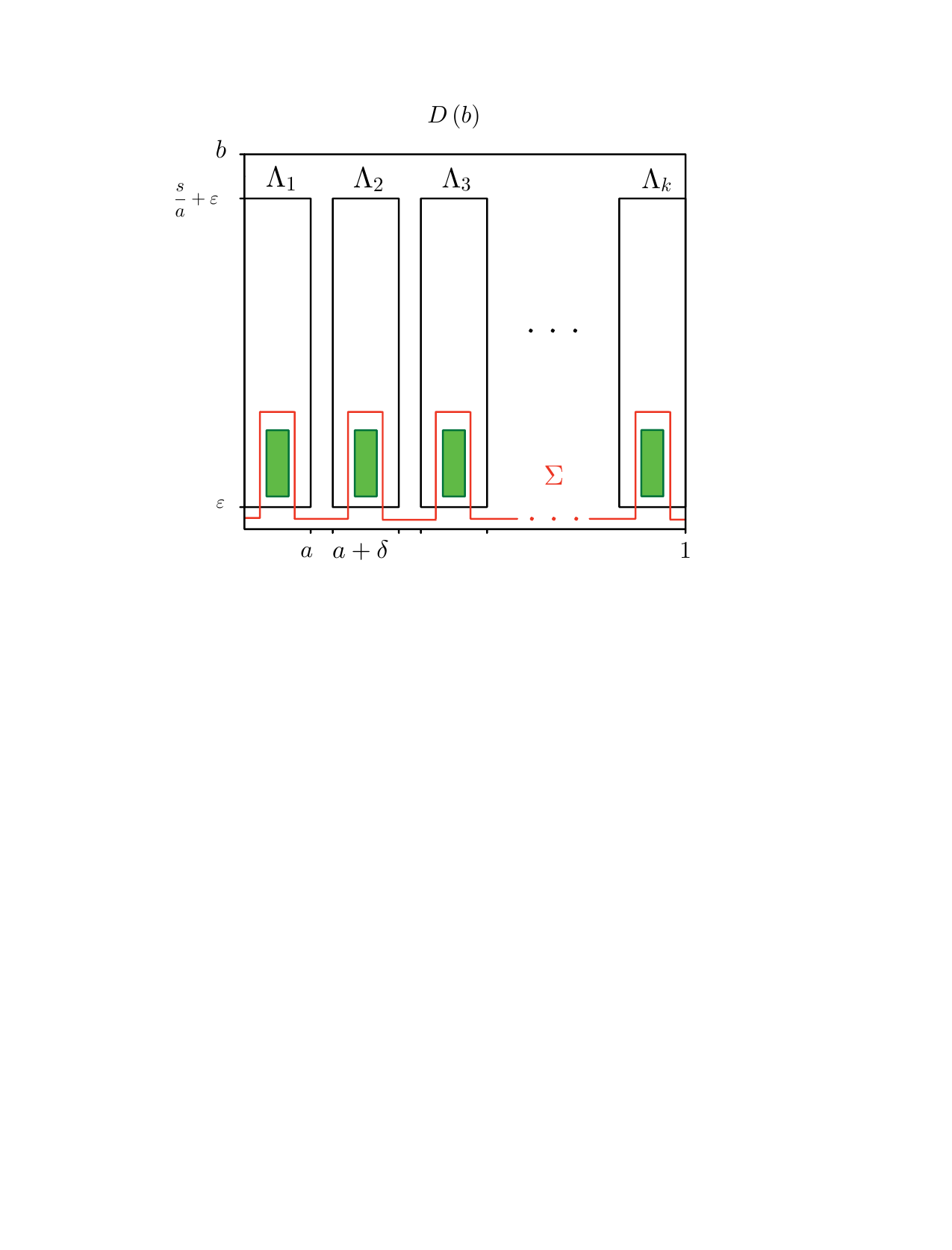}
      \vspace{-8cm}
    \caption{The circle $\Sigma$.}
    \label{fig:Sigma1}
\end{figure}

\begin{remark}The configuration of curves of area $r$ and boxes of area $\mathfrak{o}/k$ pictures in Figure \ref{fig:gamma123} is only possible under the assumption that $r>6\mathfrak{o}/k$.\end{remark}

We now displace the repositioned $L_j$ from $\mathbf{L}$ via a Hamiltonian diffeomorphism of $P(2r,b)$. This will be done in stages, for consecutive pairs of the $L_j$. Consider first $L_1$ and $L_2$. Let $H_1$ be a Hamiltonian on $D(b)$ that displaces $\Lambda_1$ from $\Sigma$ such that $\phi^1_{H_1}(\Lambda_1)$ intersects only $\Lambda_1$ and $\Lambda_2$ as pictured in Figure \ref{fig:H1}.  
\begin{figure}[!h]
    \centering 
    \includegraphics[width=\linewidth]{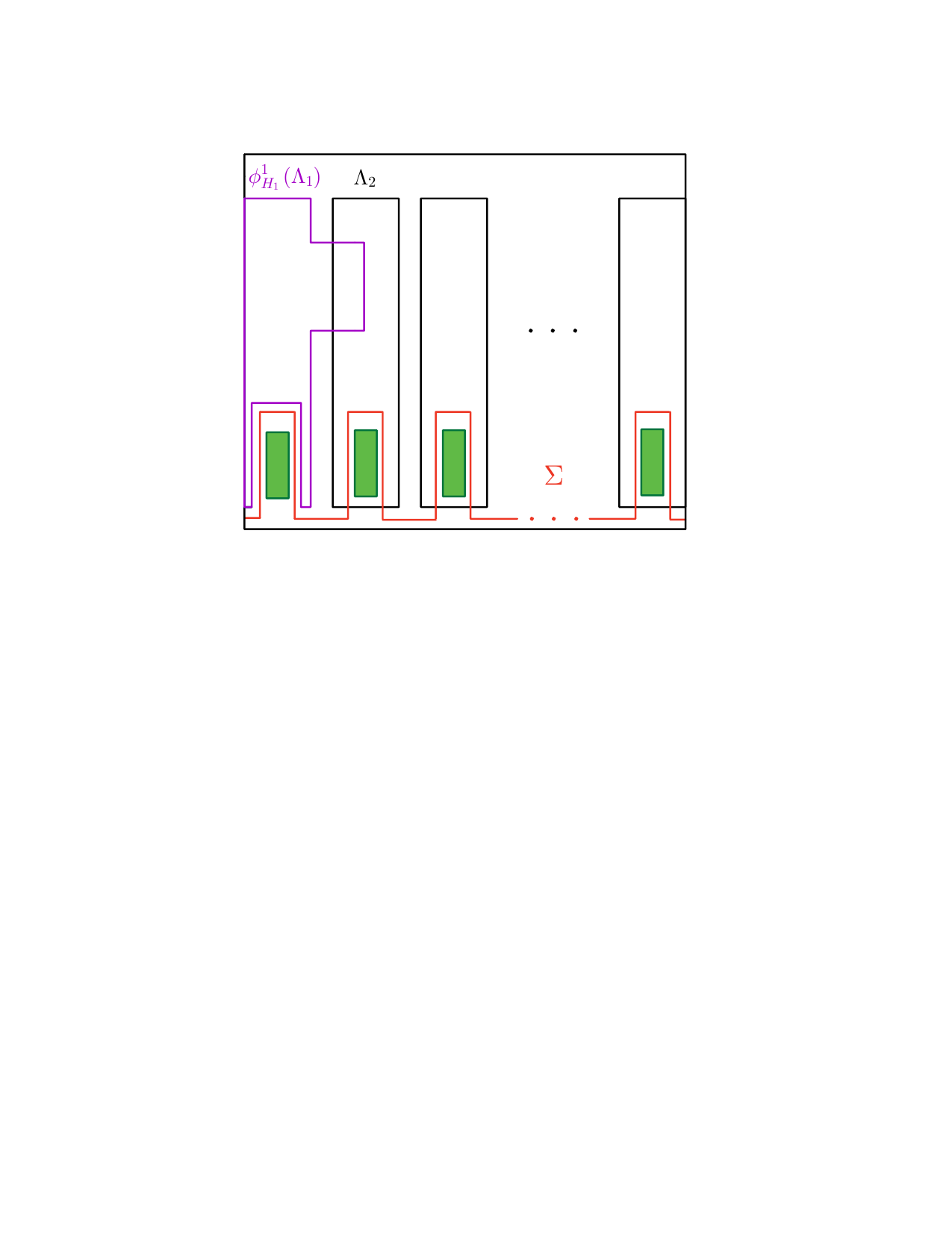}
      \vspace{-8cm}
    \caption{Moving $\Lambda_1$ away from $\Sigma$.}
    \label{fig:H1}
\end{figure}
Similarly, let $H_2$ be a Hamiltonian that displaces $\Lambda_2$ from $\Sigma$ such that $\phi^1_{H_2}(\Lambda_2)$ intersects only $\Lambda_1$ and $\Lambda_2$ as depicted in Figure \ref{fig:H2}. Both $H_1$ and $H_2$ can be chosen so that their minimum value is zero and their maximum value, $\bar{H}$, is the same. This shared maximum value $\bar{H}$ is greater than $\mathfrak{o}/k$ and can be chosen to be arbitrarily close to $\mathfrak{o}/k$.  Note also that both $H_1$ and $H_2$ can be chosen to have support contained in a region of the form $\{\eta < \theta < 2a + \delta-\eta\}$ for some $\eta>0$.

\begin{figure}[!h]
    \centering 
    \includegraphics[width=\linewidth]{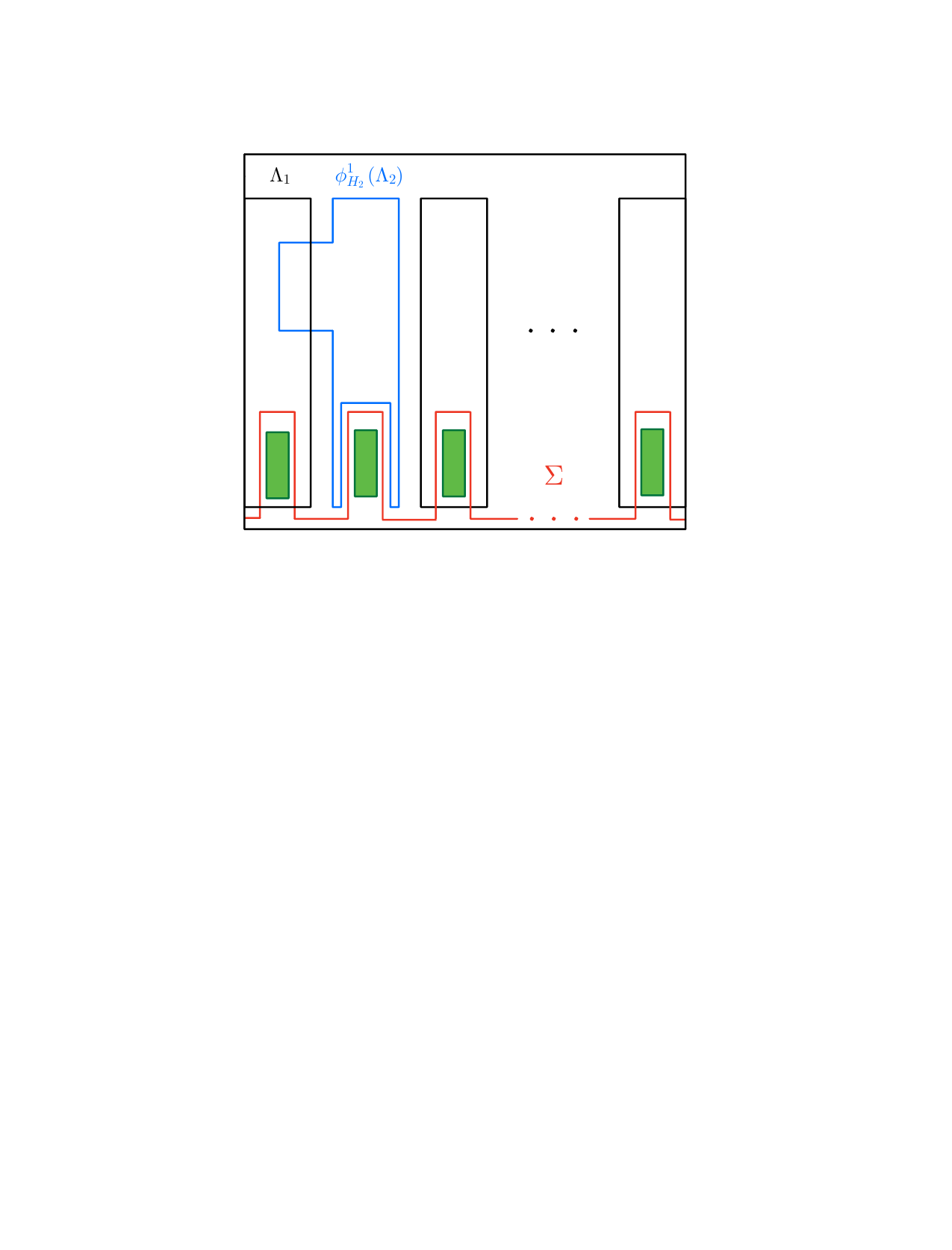}
      \vspace{-8cm}
    \caption{Moving $\Lambda_2$ away from $\Sigma$.}
    \label{fig:H2}
\end{figure}

Let $G$ be a function on $D(2r)$ such that in the region $g$ pictured in Figure \ref{fig:g}, we have $G(z_1) = \chi(x_1)$ where $\chi$ is a smoothing of a piecewise linear function, similar to the one in the  model, whose constant values along a horizontal strip in $g$ are labeled in Figure \ref{fig:g}. Assume also that $G=0$ outside an arbitrarily small neighborhood of the region $g$. 

\begin{figure}[!h]
    \centering 
    \includegraphics[width=\linewidth]{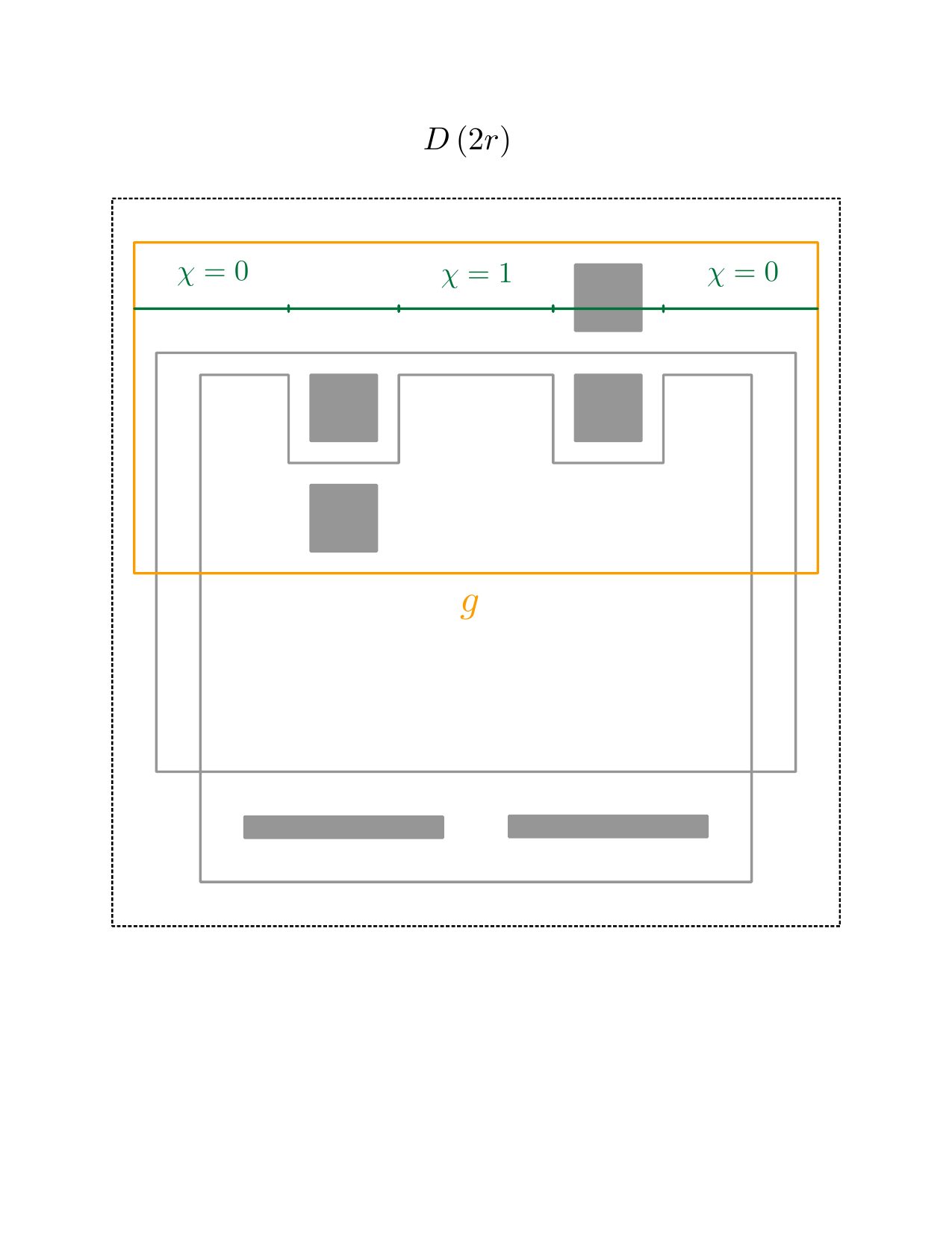}
      \vspace{-3.5cm}
    \caption{The function $G$.}
    \label{fig:g}
\end{figure}

\begin{figure}[!h]
    \centering 
    \includegraphics[width=\linewidth]{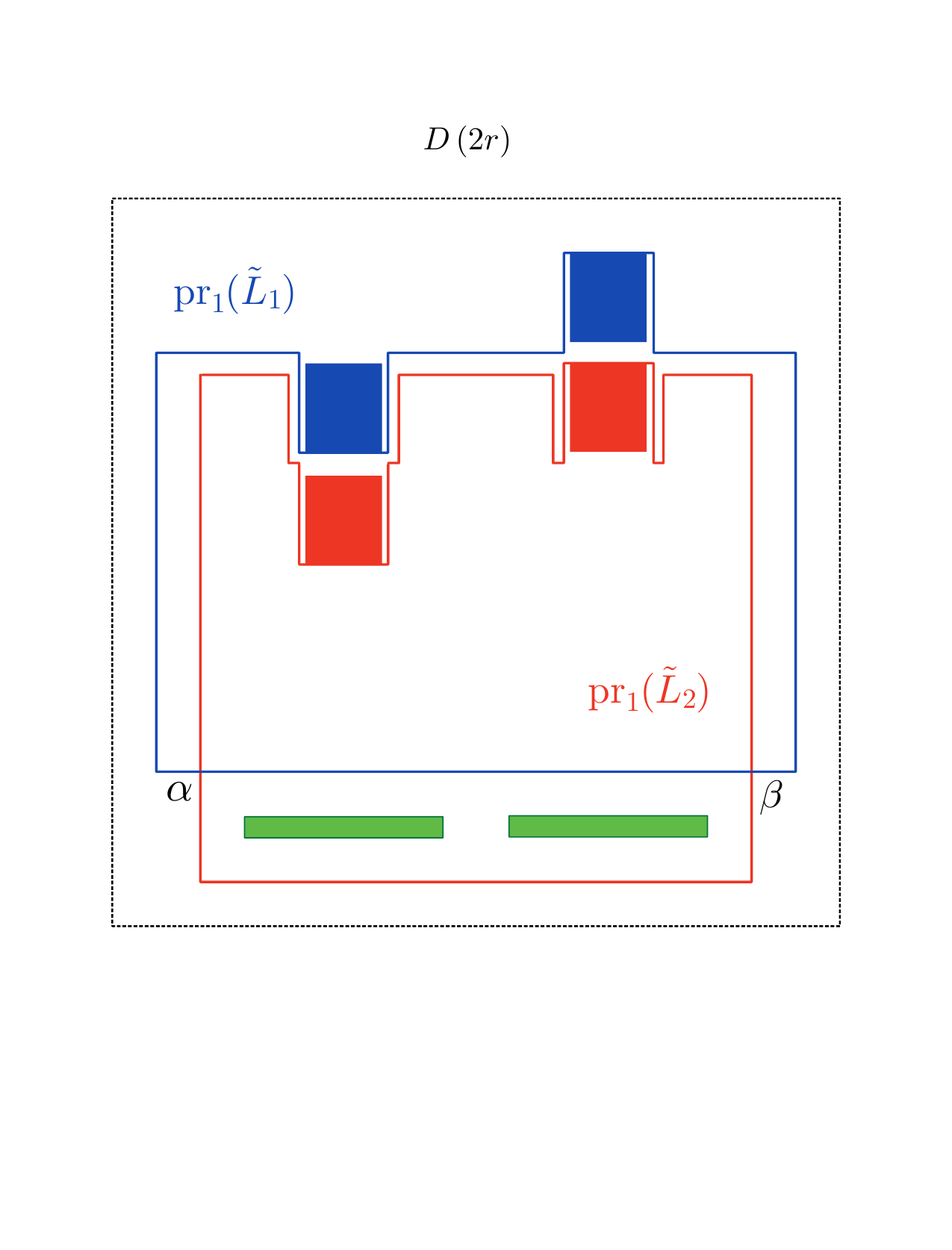}
      \vspace{-3.5cm}
    \caption{Projections of the tori $\tilde{L}_i$.}
    \label{fig:proj}
\end{figure}

Set $F_1(z_1,z_2) =G(z_1)H_1(z_2)$ and $F_2(z_1,z_2) =G(z_1)H_2(z_2)$ and consider the Lagrangian tori $\tilde{L}_1 = \phi^1_{F_1}(L_1)$ and $\tilde{L}_2 = \phi^1_{F_2}(L_2)$. The images of $\tilde{L}_1$ and $\tilde{L}_2$ under $\mathrm{pr}_1$ are pictured in Figure \ref{fig:proj}. Because the additional areas can be arranged to lie in our green regions, these projections intersect in the same two points, $\alpha$ and $\beta$, as do $\gamma_1 =\mathrm{pr}_1(L_1)$ and $\gamma_2=\mathrm{pr}_1(L_2)$. 
Similarly, the images of $\tilde{L}_1$ and $\tilde{L}_2$ under $\mathrm{pr}_1$ also intersect $\Gamma$ in the same points as $\gamma_1 $ and $\gamma_2$, respectively.

\begin{lemma}\label{lem:off}
    For the construction above
    \begin{enumerate}
        \item $\tilde{L}_1$, $\tilde{L}_2$  and $\mathbf{L}$ are disjoint from one another.\\
        \item $\mathrm{pr}_2(\tilde{L}_1)$ and $\mathrm{pr}_2(\tilde{L}_2)$ are contained in a subset of $D(b)$ of the form $\{\eta \leq \theta \leq 2a + \delta-\eta\}$ for some $\eta>0$.
    \end{enumerate}
\end{lemma}

\begin{proof}

It follows from the discussion above that any point of intersection of $\tilde{L}_1$ and $\tilde{L}_2$ must lie in either $(\mathrm{pr}_1)^{-1}(\alpha)$ or $(\mathrm{pr}_1)^{-1}(\beta)$. But the intersections of $\tilde{L}_1$ and $\tilde{L}_2$ with these two fibers are the same as those of $L_1$ and $L_2$ since $G(\alpha)=G(\beta) =0.$ In particular,  $\mathrm{pr}_1^{-1}(\alpha) \cup \tilde{L}_i =\Lambda_i=\mathrm{pr}_1^{-1}(\beta) \cup \tilde{L}_i$ for $i =1,2$. Hence, $\tilde{L}_1$ and $\tilde{L}_2$ are disjoint.

Similarly, any point of intersection of $\tilde{L}_1$ and $\mathbf{L}$ must lie in the fiber of $\mathrm{pr}_1$ over one of the two intersection points of $\mathrm{pr}_1(\tilde{L}_1)$ and $\mathrm{pr}_1(\mathbf{L}) =\Gamma.$ Labeling these two points $\alpha_1$ and $\beta_1$, the definition of $G$ implies that it is identically equal to one near both. Hence, 
$$\mathrm{pr}_1^{-1}(\alpha_1) \cup \tilde{L}_1 = \mathrm{pr}_1^{-1}(\beta_1) \cup \tilde{L}_1 =\phi^1_{H_1}(\Lambda_1)$$ while 
$$\mathrm{pr}_1^{-1}(\alpha_1) \cup \mathbf{L} = \mathrm{pr}_1^{-1}(\beta_1) \cup \mathbf{L} =\Sigma.$$ By the definition of $H_1$, these are disjoint and hence so are $\tilde{L}_1$ and  $\mathbf{L}$. A similar argument yields $\tilde{L}_2 \cap\mathbf{L} = \emptyset.$

The last assertion follows immediately from the fact, mentioned above, that both $H_1$ and $H_2$ can be chosen to have support contained in a region of the desired form.
\end{proof}

Invoking the last assertion of Lemma \ref{lem:off}, and defining $H_i$ analogously to $H_1$ or $H_2$ depending on the parity of $i$, the remaining $\tilde{L}_j$ can be repositioned within $P(2r,b)$, in pairs,  so that in the end they are disjoint from each other and from $\mathbf{L}$, as desired.

\end{document}